\documentclass[10pt,reqno]{amsart}

\usepackage{graphicx}
\usepackage{tabmac}
\usepackage{wasysym} 
\usepackage{amscd,amsmath,amsfonts,amssymb,amsthm,cite,mathrsfs,color}
\usepackage{dsfont} \usepackage{ leftidx}

\newtheorem{theorem}{Theorem}[section]
\newtheorem{lem}[theorem]{Lemma}
\newtheorem{proposition}[theorem]{Proposition}
\newtheorem{coro}[theorem]{Corollary}
\theoremstyle{definition}
\newtheorem {definition}[theorem]{Definition}

\theoremstyle{remark}
\newtheorem{remark}[theorem]{Remark}

\numberwithin{equation}{section}
\numberwithin{theorem}{section}

\usepackage{chngcntr}
\counterwithout{figure}{section}

\def\beq{\begin{equation}}
\def\eeq{\end{equation}}
\def\bae{\begin{eqnarray}}
\def\eae{\end{eqnarray}}

\def\AS{\mathrm{A\!S}}
\def\AA{\mathrm{A\!A}}
\def\SA{\mathrm{S\!A}}
\def\SS{\mathrm{S\!S}}

\def\cd{\circledast}

\def\La{\Lambda}
\def\la{\lambda}
\def\Om{\Omega}

\def\Ga{\Gamma}
\def\ga{\gamma}

\newcommand{\tcercle}[1]{\ensuremath{\setlength{\unitlength}{1ex}\begin{picture}(2.8,2.8)\put(1.4,1.4){\circle{2.7}\makebox(-5.6,0){#1}}\end{picture}}}

\begin{document}

\title[]{Jack polynomials with prescribed symmetry and some of their clustering properties}

\author{Patrick Desrosiers}\address{Instituto Matem\'atica y F\'isica,
Universidad de Talca, 2 Norte 685, Talca,
Chile}\email{Patrick.Desrosiers@inst-mat.utalca.cl}

\author{Jessica Gatica} \address{Instituto Matem\'atica y F\'isica,
Universidad de Talca, 2 Norte 685, Talca,
Chile}\email{jgatica@inst-mat.utalca.cl}

\thanks{\textit{Acknowledgments:} The authors thank Luc Lapointe and Stephen Griffeth for very stimulating and useful discussions.  
The  work of P.D.\ was supported by CONICYT through FONDECYT's grants \#1090034 and \#1131098, and the Anillo de Investigaci\'on ACT56.  J.G.\ is grateful to CONICYT for the award of a doctoral scholarship.  }



\begin{abstract}  We study Jack polynomials in $N$ variables, with parameter $\alpha$, and having a prescribed symmetry with respect to two disjoint subsets of variables.  For instance, these polynomials can exhibit a symmetry of type AS, which means that they are  antisymmetric in the first $m$ variables and symmetric in the remaining $N-m$ variables. One of our main goals is to extend recent works on symmetric Jack polynomials \cite{bh,BarFor,Griffeth} and prove that the Jack polynomials with prescribed symmetry also admit clusters of size $k$ and order $r$, that is, the polynomials  vanish to order $r$ when $k+1$ variables coincide.   

We first prove some general properties for generic $\alpha$, such as their uniqueness as triangular eigenfunctions of  operators of Sutherland type, and the existence of their analogues in infinity many variables.   We then turn our attention to the case with $\alpha=-(k+1)/(r-1)$.  We show that for each triplet $(k,r,N)$, there exist  admissibility conditions on the indexing sets, called superpartitions, that guaranty both the regularity and the uniqueness of the polynomials.  These conditions are also  used to establish similar properties for non-symmetric Jack polynomials.  As a result, we prove that the Jack polynomials with arbitrary prescribed symmetry, indexed by $(k,r,N)$-admissible superpartitions,  admit  clusters of size  $k=1$ and order $r\geq 2$.  In the last part of the article, we find necessary and sufficient conditions for  the invariance under translation of the Jack polynomials with prescribed symmetry AS.  This allows to find special families of superpartitions that imply the existence of clusters of size $k>1$ and order $r\geq 2$.  
\end{abstract}

\maketitle
\vspace{-1.0cm}
\small
\tableofcontents
\vspace{-1.0cm}
\normalsize
\thispagestyle{empty}

\newpage

\normalsize

\section{Introduction}

\subsection{Quantum Sutherland system}

In this article, we study properties of polynomials in many variables that provide the wave functions for the Sutherland model with exchange term, which is 
a famous quantum mechanical many-body problem in mathematical physics. This model describes the evolution of $N$ particles interacting  on the unit circle. 

 To be more explicit, let   $\phi_j\in\mathbb{T}=[0,2\pi)$ be the variable that describes the position of the $j$th particle in the system.  Let also the operator $K_{i,j}$ act on any multivariate function of $\phi_1,\ldots,\phi_N$ by interchanging the variables $\phi_i$ and $\phi_j$.  Finally, suppose that  $g$ is some positive real.  Then, the Sutherland model, with coupling constant $g$ and exchange terms $K_{i,j}$, is defined via the following  Schrodinger operator acting on $L^2(\mathbb{T}^N)$ \cite{polychronakos, pasquier}:
\beq H =-\sum_{i=1}^N\frac{\partial^2}{\partial \phi_i^2}+\frac{1}{2}\sum_{i\neq j}\frac{1}{\sin^2(\frac{\phi_i-\phi_j}{2})}g(g- K_{i,j}) \, .\eeq 
When acting on symmetric functions, the operators $K_{i,j}$ can be replaced by the identity and the standard Sutherland model is recovered \cite{sutherland}. The latter  is intimately related to Random Matrix Theory \cite{forrester}. For $K_{i,j}\neq 1$, the operator $H_K$ was used for  describing  systems of particles with spin (see for instance \cite{kato2009,polychronakos2006}).   

 Up to a multiplicative constant, there is a unique eigenfunction $\Psi_0$ of $H$ with minimal eigenvalue $E_0$ \cite{kato}.  Explicitly, defining $\alpha=g^{-1}$ and $x_j=e^{\mathrm{i}\phi_j}$, where $\mathrm{i}=\sqrt{-1}$, we have  
\beq \Psi_0=\prod_{1\leq i<j\leq N}|x_i-x_j|^{1/\alpha} , \qquad E_0=\frac{ N(N^2-1)}{12\alpha^2}.\eeq 
The operator $H$ admits eigenfunctions of the form $\Psi(x)=\Psi_0(x) P(x)$, where $P(x)$ is a polynomial eigenfunction of the operator $D=\Psi_0^{-1} \circ (H-E_0)\circ \Psi_0$, that is, 
\begin{multline}\label{defD} \qquad D =\sum_{i=1}^N\left(x_i\frac{\partial}{\partial x_i}\right)^2+\frac{2}{\alpha}\sum_{1\leq i<j\leq N}^N\frac{x_ix_j}{x_i-x_j} \left( \frac{\partial}{\partial x_i}-\frac{\partial}{\partial x_j}\right)\\-\frac{2}{\alpha}\sum_{1\leq i<j\leq N} \frac{x_ix_j}{(x_i-x_j)^2}(1-K_{i,j})+\frac{N-1}{\alpha}\sum_{i=1}^Nx_i\frac{\partial}{\partial x_i}\, .\qquad \end{multline}

\subsection{Symmetric Jack polynomials and their clustering}

 Let $\mathscr{S}_{\{1,\ldots,N\}}$ denote the ring of symmetric polynomials in $N$ variables with coefficients in the field  of rational functions in the formal parameter $\alpha$, here denoted by  $\mathbb{C}(\alpha)$.  Any homogenous element of degree $n$ in $\mathscr{S}_{\{1,\ldots,N\}}$  can be indexed by a partition of $n$, which is sequence  $\lambda=(\la_1,\ldots,\la_N)$ such that $\la_1\geq \ldots\geq \la_N\geq 0$ and $\la_1+\ldots+\la_N=n$.  Note that in general, we only write  the non-zero elements of the partition.  Partitions are often sorted with the help of the following partial order, called the dominance order:
\beq \label{defdominance}\la\geq \mu \qquad \Longleftrightarrow \qquad \sum_{i=1}^k\lambda_i\geq \sum_{i=1}^k \mu_i,\quad \forall \, k,\eeq   where it is assumed that both partitions have the same degree $n$. A convenient way to write a symmetric polynomial consists in giving its linear expansion in the basis of monomial symmetric functions  $\{m_\la\}_\la$, where 
\beq m_\la=x_1^{\la_1}\cdots x_N^{\la_N}+\text{distinct permutations}.\eeq

 Since Stanley's seminal work \cite{stanley}, we know that the symmetric Jack polynomial, denoted $P_\la=P_\lambda(x;\alpha)$, is the unique symmetric eigenfunction of \eqref{defD}   that is  monic and triangular  in the monomial basis, where the triangular is taken with respect to the dominance ordering.   In symbols,   $P_\la$ is the unique element of  $\mathscr{S}_{\{1,\ldots,N\}}$  that satisfies the following two properties: 
\begin{align*} \text{(A1)}\qquad  &P_\la=m_\la+\sum_{\mu<\lambda}c_{\lambda,\mu}(\alpha)m_\mu \, ,\\
\text{(A2)}\qquad  &D P_\la= \varepsilon_\la(\alpha) P_\la \, ,
\end{align*}
where $\varepsilon_\la(\alpha)$ is the eigenvalue and will be given later in Lemma \ref{lemordereigen}.  For instance, 
\beq \label{ex1} P_{(4)}=m_{(4)}+ \,{\frac { 6\left( \alpha+1 \right) m_{(2,2)}}{ \left( 2\,
\alpha+1 \right)  \left( 3\,\alpha+1 \right) }}+ \,{\frac {4m_{(3,1)}}{
3\,\alpha+1}}+ \,{\frac {12 m_{(2,1,1)}}{ \left( 2\,\alpha+1 \right) 
 \left( 3\,\alpha+1 \right) }}+\,{\frac {24 m_{(1,1,1,1)}}{ \left( 2\,
\alpha+1 \right)  \left( 3\,\alpha+1 \right)  \left( \alpha+1 \right) 
}}
\eeq and
\beq \label{ex2} P_{(2,2)}=m_{(2,2)}+2\,{\frac {m_{(2,1,1)}}{\alpha+1}}+12\,{\frac {m_{(1,1,1,1)}
}{ \left( \alpha+2 \right)  \left( \alpha+1 \right) }}.
\eeq 

It is worth stressing that uniqueness of the polynomial satisfying (A1) and (A2) remains valid if we suppose that $\alpha$ is a positive real or an irrational (see Section \ref{SectionAlphaGeneric}).  However, when $\alpha$ is a negative rational,   the uniqueness is generally lost.  Worse, as the examples above clearly show,  the Jack polynomials have poles for negative rational values of $\alpha$.   

Nevertheless, Feigin,  Jimbo,  Miwa, and  Mukhin \cite{FJMM}  showed that  for certain classes of partitions, called admissible partitions, the  Jack polynomial are not only regular at certain negative fractional values of $\alpha$ but also exhibit remarkable vanishing properties when some variables coincide.   
\begin{definition}[Admissibility]  \label{defadm}  Let $k$ and $r-1$ be positive integers such that $\gcd(k+1,r-1)=1$.
A partition $\lambda=(\la_1,\ldots,\la_N)$ is said to be \textbf{$(k,r,N)$-admissible} if 
\beq \label{defadm} \la_i-\la_{i+k}\geq r \qquad \forall \, 1\leq i\leq N-k.\eeq
\end{definition} 
Proposition 4.1 in \cite{FJMM} states the following: If  $\lambda$ is $(k,r,N)$-admissible and $\alpha$ is equal to \beq\label{alphakr} \alpha_{k,r}=-\frac{k+1}{r-1}  ,\eeq  then $P_\la(x;\alpha)$ is regular  
and vanishes  when $k+1$ variables coincide, that is,  
$ P_\la(x;\alpha_{k,r}) |_{x_{N-k}=\ldots=x_{N}}=0$.  Bernevig and Haldane \cite{bh} later used the above vanishing property for modeling fractional quantum Hall states with  Jack polynomials.   They moreover conjectured that the  Jack polynomials indexed by $(k,r,N)$-admissible partitions satisfy the following clustering property, which gives a more precise statement about how the polynomials vanish.    

\begin{definition}[Clustering property]  \label{defclustering} Let  $k,r\in\mathbb{Z}_+$.   A symmetric polynomial $P$  admits a \textbf{cluster of size $k$ and order $r$}   if it vanishes to order at least $r$ when $k+1$ of the  variables are equal, that is,  
\beq \label{eqclustering} P(x_1,\ldots,x_{N-k},\overbrace{z\ldots,z}^{k\text{ times}}) =\prod_{j=1}^{N-k}(x_{j}-z)^r Q(x_1,\ldots, x_{N-k},z)\, \eeq for some polynomial $Q$ in $N-k+1$ variables. \end{definition} 

Let us illustrate how the clustering property works by returning to the examples given in \eqref{ex1} and \eqref{ex2}.  Clearly, the partition $(4)$ can be admissible only for $N=2$ and in fact, it is  both $(1,2,2)$-admissible and $(1,4,2)$-admissible. There are two possible values for $\alpha$: $\alpha_{1,2}=-2$ and $\alpha_{1,4}=-2/3$.   One can check that as expected, $P_{(4)}$ admits clusters of size $k=1$ whose respective order is $r=2$ and $r=4$:
$$ P_{(4)}(x_1,z;-2)=\frac15  \left(x_1-z\right) ^{2}\left( 5{x_1}^{2}+6\,{ x_1}{z}+5{{z}}^{2}
 \right)
 \quad \text{and}\quad P_{(4)}(x_1,z;-2/3)=(x_1-z)^4. $$   The partition $(2,2)$ is $(2,2,4)$-admissible and one easily sees that the associated Jack polynomial admits a cluster of size $k=2$ and order $r=2$:
$$ P_{(2,2)}(x_1,x_2,z,z;-3)=(x_1-z)^2(x_2-z)^2. $$     
Note that  for the above examples and contrary to the general case (e.g., see the introduction of \cite{dlm_cmp2}),   the order of vanishing is exactly equal to $r$.

  Baratta and Forrester \cite{BarFor} proved that the Jack polynomials (along with other symmetric polynomials such as Hermite and Laguerre) indexed with $(1,r,N)$-admissible partitions follow  Equation \eqref{eqclustering} at $\alpha_{1,r}$.\footnote{The proof in \cite{BarFor} is not entirely complete, since an implicit assumption about the uniqueness of the solution to (A1) and (A1) was made.  See Remark \ref{remarkvalidity}. }  The same 
authors also proved clustering properties for $k>1$ in the case of partitions associated to translationally invariant Jack polynomials \cite{jl}.   Very recently,  Berkesch, Griffeth, and   Sam proved the general $k\geq 1$ clustering property for Jack polynomials \cite{Griffeth}.   Their method method was based  the representation theory of the rational Cherednik algebra.   In fact, reference \cite{Griffeth} also contains the proof for more general vanishing properties in the case of many clusters, some of them having been conjectured earlier in \cite{bh} .  

\subsection{Jack polynomials with prescribed symmetry}

The  operator $D$ obviously has polynomial eigenfunctions of different symmetry classes.    First, as explained earlier,  there are the symmetric Jack polynomials $P_{\la}(x;\alpha)$.    Second, there are the non-symmetric Jack polynomials $E_\eta(x;\alpha)$, which were introduced by Opdam \cite{opdam}. These polynomials can be defined as the common eigenfunctions of the commuting set $\{\xi_j\}_{j=1}^N$, where each $\xi_{j}$ is a first order difference-differential operator $\xi_{j}$, often called  a Cherednik operator.

However,  as first shown by Baker and Forrester \cite{bf2}, we can use the latter polynomials to construct orthogonal eigenfunctions of $D$ whose symmetry property interpolates between the completely symmetric Jack polynomials, $P_{\la}(x;\alpha)$, and the completely antisymmetric ones, sometimes denoted by $S_{\la}(x;\alpha)$.   In order  words, there exist eigenfunctions that are symmetric in some given subsets of $\{x_1,\ldots,x_N\}$  and antisymmetric with other subsets, all subsets of variables being mutually disjoint.     Such eigenfunctions are called Jack polynomials with prescribed symmetry and were studied in \cite{bf2,kato1998,dunkl,bdf,mcanally}.  Here we limit our study to the case of two subsets. 

 Before given the precise definition of the Jack polynomials with prescribed symmetry, let us introduce some more notation.  For a given set $K=\{k_1,\ldots,k_M\}  \subseteq\{1,\ldots,N\}$, let $\mathrm{Asym}_K$ and $\mathrm{Sym}_K$    respectively  denote the antisymmetrization and the symmetrization operators with respect to the variables $x_{k_1}, \ldots,x_{k_M}$.  If $f(x)$ is an element of $\mathscr{V}=\mathbb{C}(\alpha)[x_1,\ldots,x_N]$, then  $\mathrm{Sym}_K f(x)$ belongs to $\mathscr{S}_K$, the  submodule of $\mathscr{V}$ whose elements are symmetric polynomials in $x_{k_1}, \ldots,x_{k_M}$.  Similarly, $\mathrm{Asym}_K f(x)$ belongs to $\mathscr{A}_K$, the submodule of antisymmetric polynomials in $x_{k_1}, \ldots,x_{k_M}$.

\begin{definition} \label{defJackPrescibed}   For a given positive integer $m\leq N$,  set $I=\{1,\ldots,m\}$ and $J=\{m+1,\ldots,N\}$. \footnote{The above definition could be obviously generalized by considering $I=\{i_1,\ldots,i_m\}$ and $J=\{j_1,\ldots,j_{N-m}\}$ as two general disjoint sets such that $I\cup J=\{1,\ldots, N\}$. However, this would make the presentation more intricate.  One easily goes from one definition to the other by permuting the variables.     } Moreover, let  $\lambda=(\lambda_1,\ldots,\lambda_m)$ and $ \mu=(\mu_1,\ldots,\mu_{N-m})$ be partitions.  The monic Jack polynomial with prescribed symmetry of type antisymmetric-symmetric ($\AS$ for short) and   indexed by the ordered set $\La=(\lambda_1,\ldots,\lambda_m; \mu_1,\ldots,\mu_{N-m})$  is defined as  follows:
\beq P^{\AS}_{\La}(x;\alpha)=c^{\AS}_{\La}  \,\mathrm{Asym}_I \mathrm{Sym}_J E_\eta(x;\alpha),   \eeq 
where  $\eta$ is a composition equal to $(\lambda_m,\ldots,\lambda_1,\mu_{N-m},\ldots,\mu_{1})$  while the normalization factor $c^{\AS}_{\La} $ is   such  that the coefficient of $x_{1}^{\lambda_{1}}\cdots x_{m}^{\lambda_{m}}x_{m+1}^{\mu_{1}}\cdots x_{N}^{\mu_{N-m}}$ in $P^{\AS}_{\La}(x;\alpha)$ is equal to one. 
Other types of Jack polynomials are defined similarly:
\begin{align*} P^{\AA}_{\La}(x;\alpha)&=c^{\AA}_{\La} \, \mathrm{Asym}_I \mathrm{Asym}_J E_\eta(x;\alpha), \\ P^{\SA}_{\La}(x;\alpha)&=c^{\SA}_{\La} \, \mathrm{Sym}_I \mathrm{Asym}_J E_\eta(x;\alpha),\\
 P^{\SS}_{\La}(x;\alpha)&=c^{\SS}_{\La} \, \mathrm{Sym}_I \mathrm{Sym}_J E_\eta(x;\alpha)\, . \end{align*} 
The coefficients $c_\La $ are given in Equations \eqref{coefAS}--\eqref{coefSS}.
\end{definition} 

The above polynomials respectively belong to $\mathscr{A}_I\otimes \mathscr{S}_J$, $\mathscr{A}_I\otimes \mathscr{A}_J$, $\mathscr{S}_I\otimes \mathscr{A}_J$, $\mathscr{S}_I\otimes \mathscr{S}_J$, which are all  vector spaces over $\mathbb{C}(\alpha)$.   These spaces are spanned by monomials, denoted by $m_\La$, each of them being indexed by an ordered pair of  partitions $\La=(\lambda_1,\ldots,\lambda_m; \mu_1,\ldots,\mu_{N-m})$.  Analogously to the Jack polynomials with prescribed symmetry, the monomials are  defined by the action of $\mathrm{Asym}_K$ and $\mathrm{Sym}_K$, where $K$ is either $I$ or $J$, on the non-symmetric monomial $x_1^{\la_1}\cdots x_m^{\la_m}x_{m+1}^{\mu_1}\cdots x_{N}^{\mu_{N-m}}$.  See Section \ref{SectionDefPrescribed} for more details.    

The case $\mathrm{AS}$ is very special since  the polynomials  $P^{\AS}_{\La}(x;\alpha)$ can be used to solve the supersymmetric  Sutherland model \cite{dlm_cmp}, which is a generalization of the above model that also involves Grassmann variables.  In this context,  the indexing set $\La=(\lambda_1,\ldots,\lambda_m; \mu_1,\ldots,\mu_{N-m})$ is called a superposition  -- equivalently, it could be called  an overpartition (see \cite{cl})-- and is such that the partition $\lambda=(\lambda_1,\ldots,\lambda_m)$ is strictly decreasing.  The correct diagrammatic representation of superpartitions, first given in \cite{dlm_comb}, proved to be very useful. It allowed, for instance, the derivation of a very simple evaluation  formula for $P^{\AS}_{\La}(x;\alpha)$ \cite{dlm_imrn}, which in turn lead to the first results regarding the clustering properties of these polynomials \cite{dlm_cmp2}.  We adopt here a slightly more general point of view for superpartitions.

\begin{definition}[Superpartitions and diagrams]\label{defsparts} The ordered set $\La=( \Lambda_1,\ldots,\Lambda_m; \Lambda_{m+1},\ldots,\La_N)$ is a  \textbf{superpartition $\La$  of bi-degree $(n|m)$}  if it satisfies the following conditions:$$  \Lambda_1\geq \cdots\geq \Lambda_m\geq 0 \qquad \Lambda_{m+1}\geq \cdots\geq \Lambda_N\geq 0 \qquad   \sum_{i=1}^N\La_i=n.$$ 
If $(\La_1,\ldots,\La_m)$ is moreover strictly decreasing, then $\La$ is called a \textbf{strict superpartition}. 
Equivalently,  we can write the superpartition $\La$ as a pair of partitions $(\La^\cd,\La^*)$ such that  
$$ \La^\cd=( \Lambda_1+1,\ldots,\Lambda_m+1, \Lambda_{m+1},\ldots,\La_N)^+,\qquad  \La^*=( \Lambda_1,\ldots,\Lambda_m, \Lambda_{m+1},\ldots,\La_N)^+,\qquad $$
where $+$ indicates the operation that reorder the elements of a composition in decreasing order.  The diagram of $\La$ is obtained from that of $\La^\cd$ by replacing the boxes belonging to the skew diagram $\La^\cd/\La^*$ by circles. The \textbf{dominance order for superpartitions} is defined as follows:
$$ \La> \Om\quad \Longleftrightarrow\quad \La^*>\Om^*\quad \text{or}\quad \La^*=\Om^*\quad \text{and}\quad\La^\cd>\Om^\cd.$$  
\end{definition} 

For instance,    the ordered set  $\Lambda= (4,3,0;4)$ is a strict superpartitions of bi-degree $(11|3)$.  It can be written as a pair $(\La^\cd,\La^*)$, where $\La^\cd=(4+1,3+1,0+1,4)^+=(5,4,4,1)$ and $\La^*=(4,4,3,0)$.  The diagram associated to $\La$ is obtained as follows:  
{\small $$ \Lambda^\circledast={\tableau[scY]{&&&&\\&&& \\&&&\\  \\ }} \quad  \Lambda^*={\tableau[scY]{&&&\\&&&\\ &&\\ \bl\\ }} \quad \Longrightarrow \quad \Lambda^\circledast/\La^*={\tableau[scY]{\bl&\bl&\bl&\bl&\\ \bl&\bl&\bl \\ \bl&\bl& \bl&\\    \\ }}
\quad \Longrightarrow \quad  \Lambda={\tableau[scY]{&&&&\bl\tcercle{}\\&&&\\&&&\bl\tcercle{}\\\bl\tcercle{} \\ }}  $$} Similarly, $\Omega= (5,3,1; 2)$ and $\Gamma=(3,1,0; 5,2)$ are   superpartitions of the same bi-degree.  The associated diagrams are respectively   
{\small$$ \Omega={\tableau[scY]{&&&&&\bl\tcercle{} \\&&&\bl\tcercle{}\\&\\&\bl\tcercle{}\\}} \quad\text{and}\qquad \Gamma={\tableau[scY]{&&&& \\&&&\bl\tcercle{}\\&\\&\bl\tcercle{}\\\bl\tcercle{}}} $$
}One easily verifies that $\Omega>\Gamma$, while $\La$ is   comparable with neither $\Omega$ nor $\Gamma$.

\subsection{Main results}

Our first aim is to give a very simple characterization of Jack polynomials with prescribed symmetry that generalizes Properties (A1) and (A2).  For this, we use differential operators of Sekiguchi type: 
\beq \label{defseki} S^*(u)=\prod_{i=1}^N(u+\xi_i)\quad \text{and}\quad   S^\cd(u,v)=\prod_{i=1}^m(u+\xi_i+\alpha)\prod_{i=m+1}^N(v+\xi_i),\eeq
where $u$ and $v$ are formal parameters.  Note we will often set $v=u$ since this case leads to simpler eigenvalues. 
It is a simple exercise to show that the symmetric Jack polynomial $P_\lambda(x;\alpha)$  is an eigenfunction of $S^*(u)$, with eigenvalue
\beq \label{eigenseki} \varepsilon_\lambda(\alpha,u)=\prod_{i=1}^N(u+\alpha \lambda_i-i+1).\eeq The same polynomial cannot be an eigenfunction of $S^\cd(u,v)$, since the latter does not preserve $\mathscr{S}_{\{1,\ldots,N\}}$.  In fact,  $ S^*$ and $S^\cd$ together preserve the spaces $\mathscr{A}_I\otimes \mathscr{S}_J$, $\mathscr{A}_I\otimes \mathscr{A}_J$, $\mathscr{S}_I\otimes \mathscr{A}_J$, and $\mathscr{S}_I\otimes \mathscr{S}_J$.  They moreover serve as generating series for the conserved quantities of the Sutherland model with exchange terms:
$$ S^*(u)=\sum_{d=0}^N u^{N-d}\mathcal{H}_d,\qquad   S^\cd(u,v)=\sum_{d=0}^m\sum_{d'=0}^{N-m} u^{m-d}v^{N-m-d'}\mathcal{I}_{d,d'}\,,$$ where all the operators $\mathcal{H}_d$ and $\mathcal{I}_{d,d'}$ commute among themselves and preserve the spaces mentioned above.     Amongst them, the most important are  
\beq \mathcal{H}= \mathcal{H}_2=\sum_{i=1}^N {\xi_i}^2,\qquad \mathcal{I}= \mathcal{I}_1=\sum_{i=1}^m {\xi_i}\, .\eeq 
{ Note that the operator $D$ introduced in \eqref{defD} is related to the operators $\mathcal{H}_1$ and $\mathcal{H}_2$ via $$ \mathcal{H}_2+(N-1)\mathcal{H}_1=\alpha^2D+\frac{N(N-1)(2N-1)}{6} .$$}

\begin{theorem}[Uniqueness at generic $\alpha$] \label{theo1} Let $\La$ be a superpartition of bi-degree $(n|m)$.  Suppose that $\alpha$ is a formal parameter or a complex number that is neither zero nor a negative rational.     Then, the Jack polynomial with prescribed symmetry $P_\La$   is the unique polynomial  satisfying 
\begin{align*} \text{(B1)} \qquad& P_\Lambda=m_\Lambda + \sum_{\Gamma<\Lambda}c_{\La,\Gamma}m_\Gamma \, , \qquad c_{\La,\Gamma}\in\mathbb{C}(\alpha) ;\\
\text{(B2)} \qquad&\mathcal{H}\,P_\La=\,d_\La\,P_\La\quad \quad \text{and} \quad \mathcal{I}\,P_\La=e_\La\, P_\La .
\end{align*} for some $c_{\La,\Ga}, d_{\La}, e_\La\in\mathbb{C}(\alpha)$.   Moreover, the eigenvalues $d_{\La}$ and $e_\La$ can be computed explicitly; they are   in Equations \eqref{eqeigen1} and \eqref{eqeigen2}, given respectively.
\end{theorem}

Our second aim is to prove clustering properties for Jack polynomials with prescribed symmetry.  This properties appears only for negative fractional values of $\alpha$. As explained in Section 3, Theorem \ref{theo1} is no longer valid for such $\alpha$, so we must restrict ourselves to polynomials indexed by admissible superpartitions.  In the case of strict superpartitions, the appropriate admissibility condition was first given in \cite{dlm_cmp2} -- below, this is called the weak admissibility.  When we symmetrize with respect to the  first set of variables, then a more restrictive definition of the admissibility is required.

\begin{definition}[Admissibility]\label{defadm} Let $k$ and $r-1$ be positive integers such that $\gcd(k+1,r-1)=1$.
The superpartition $\Lambda$ is \textbf{weakly $(k,r,N)$-admissible} if and only if 
$$ \La_i^\cd-\La_{i+k}^* \geq r\qquad \forall \, i\leq N-k, $$
while it is \textbf{moderately $(k,r,N)$-admissible} if and only if 
$$ \La_i^\cd-\La_{i+k}^\cd \geq r\qquad \forall \, i\leq N-k,$$
{and it is \textbf{strongly $(k,r,N)$-admissible} if and only if 
$$ \La_i^\cd-\La_{i+k}^\cd \geq r\quad \forall \, i\leq N-k \qquad \text{and} \qquad \La^* \; \text{is $(k+1,r,N)$-admissible}$$ 
When   $\Lambda$ is said to be $(k,r,N)$-admissible, without specifying weakly, moderately or strongly, it is understood that either $\Lambda$ is strongly $(k,r,N)$-admissible or  $\Lambda$ is both strict  and weakly $(k,r,N)$-admissible.}
\end{definition}

\begin{theorem}[Uniqueness and regularity at $\alpha_{k,r}$] Let $\La$ be a superpartition of bi-degree $(n|m)$ and $(k,r,N)$-admissible.  Then, the Jack polynomial with prescribed symmetry  obtained from (B1) and (B2) is regular at $\alpha=\alpha_{k,r}$. Moreover, it  is the unique polynomial  satisfying  
\begin{align*} \text{(C1)} \qquad & P_\Lambda=m_\Lambda + \sum_{\Gamma<\Lambda}c_{\La,\Gamma}m_\Gamma \, ,\qquad c_{\La,\Gamma}\in \mathbb{C} ,\\
\text{(C2)} \qquad &  S^*(u)\big|_{\alpha=\alpha_{k,r}}\,P_\La=\varepsilon_{\La^*}(\alpha_{k,r},u)\,P_\La \quad \text{and}\quad  S^\cd(u,u) \big|_{\alpha=\alpha_{k,r}}\, P_\La=\varepsilon_{\La^\cd}(\alpha_{k,r},u)\, P_\La.
\end{align*}
The eigenvalues are given in \eqref{eigenseki}.
\end{theorem}

In the case $k=1$, a similar theorem holds for the non-symmetric Jack polynomials indexed by a composition of the form $(\La_1,\ldots,\La_m,\La_{m+1},\ldots, \La_N)$, where the entries $\La_i$ belong to the admissible superpartition of the previous theorem.  Combining this result with Definition \ref{defJackPrescibed} allows us to prove that the general $k=1$ clustering property for Jack polynomials with prescribed symmetry.  In the $\mathrm{AS}$ case, this property was first conjectured in \cite{dlm_cmp2}.    

\begin{proposition}[Clustering property for $k=1$] Let $\La$ be $(1,r,N)$-admissible, where $r$ is even.  For the symmetry types AS, SS, and SA, let $K$ respectively stand for $J$, $J$, and $I$. Then, $$P_\La(x;\alpha_{1,r})=\prod_{\substack{i,j\in K\\  i<j}}(x_i-x_j)^rQ(x),$$
while for the symmetry type AA, $$  P_\La(x;\alpha_{1,r})=\prod_{\substack{1\leq i<j\leq N}}(x_i-x_j)^{r-1}Q(x).$$
  The precise form of $Q(x)$ will be given in Section \ref{SectionClusterk1}.
\end{proposition}

We have not been able to prove the natural generalization of the above proposition: All Jack polynomials with prescribed symmetry,  indexed by $(k,r,N)$-admissible superpartitions, admit a cluster of size $k$ and order $r$ at $\alpha=\alpha_{k,r}$.  However, following an idea of Baratta and Forrester \cite{BarFor}, we know  that if a polynomial is invariant under translation and satisfies  basic factorization and stability properties (see Lemma \ref{simpleprod} and Proposition \ref{stability1} ), then the polynomial can  admit clusters of size $k>1$.  In the last part of the article, we thus turn our attention to the  translationally invariant Jack polynomials with prescribed symmetry.  Exploiting a result obtained in the context of the supersymmetric Sutherland model, so only valid for the $\AS$ case, we find all strict and admissible superpartition that lead to invariant polynomials.

\begin{theorem}[Translation invariance] \label{TheoTransla}Let $\Lambda$ be a strict and weakly $(k,r,N)$-admissible superpartition.  Then,  the Jack polynomial with prescribed symmetry $P^\AS_{\Lambda}(x; \alpha_{k,r})$ is invariant under translation if and on if  one of the following two conditions is satisfied: 
\begin{itemize} \item[\textit{(D1)}] all corners (circles or boxes) of $\Lambda$ are located  at the upper corner of a hook of type $B_{k,r}$, $\tilde B_{k,l}$, $C_{k,r}$, or $\tilde C_{k,l}$, except for one corner, which must be located at the point $(N-k,r)$;
\item[\textit{(D2)}] all corners of  $\Lambda$ are circles such that if they are not interior, they are located at the upper corner of a hook of type  $C_{k,r}$ or $\tilde C_{k,l}$, except for at most one non-interior corner $(i,j)$, which is such that  $i=N+1-\bar k(k+1)$ y $j=\bar k(r-1)+1$ for some $\bar k$. 
\end{itemize}
Types of hooks are given in Figure \ref{fige_scuadras}.  Interior and non-interior corners are defined in Definition \ref{defcorner}.
\end{theorem}

\begin{proposition}[Clustering property for $k\geq1$]   Let $\Lambda$ be  a strict and weakly $(k,r,N)$-admissible superpartition of bi-degree $(n|m)$.  Suppose moreover that $\La$ satisfies (D1) or (D2) and has a length $\ell$ not greater than $N-k$.  Then, for some polynomial $Q$,
$$  P^\AS_{\Lambda}(x_1,\ldots,x_{N-k},\overbrace{z\ldots,z}^{k\text{ times}}; \alpha_{k,r}) =\prod_{j=m+1}^{N-k}(x_{j}-z)^r Q(x_1,\ldots, x_{N-k},z).$$
\end{proposition}

\begin{figure}
\caption{Types of hooks.  From left to right,  $C_{k,r}$, $\tilde C_{k,r}$, $B_{k,r}$ and $\tilde B_{k,r}$}
\label{fige_scuadras}
\begin{center}
\setlength{\unitlength}{4.2pt}
\vspace{0.5cm}
\begin{picture}(84,19)(-10,-3)
\put(-8,5){{$$\tableau[scY]{  \tf &  &\bl $\ldots$  & & \bl \tcercle{} \\  \\ \bl \vspace{-2ex}\vdots \\ \bl& \bl \\  & \bl\\ \\ } $$} } 
\put(-6,13.2){ {$\overbrace{\phantom{xxxxxx}}^{r-1} $} }
  \put(-11.7,3.5){{$ k \left\{\rule[20pt]{0pt}{15pt}\right.$}}   
\put(14,5){{$$\tableau[scY]{  \tf &  &\bl \ldots    & & \bl \tcercle{}\\    \\ \bl \vdots\\ \bl \\ \\ \bl \tcercle{} } $$} }
 \put(16,13.2){ {$\overbrace{\phantom{xxxxxx}}^{r-1} $} }
 \put(6.2,3.6){{$ k+1 \left\{\rule[19.9pt]{0pt}{15pt}\right.$}}   
\put(41,5){{$$\tableau[scY]{  \tf &  &\bl \ldots     & & \\  \\ \bl \vdots\\ \bl \\ \\  \\}$$} }
\put(43,13.2){ {$\overbrace{\phantom{xxxxxxxx}}^{r} $} }
  \put(37.2,3.6){{$ k\left\{\rule[19.9pt]{0pt}{15pt}\right.$}}   
\put(65,5){{$$\tableau[scY]{  \tf &  &\bl \ldots    & & \\     \\ \bl \vdots\\ \bl \\ \\ \bl \tcercle{} } $$} }
 \put(67,13.2){ {$\overbrace{\phantom{xxxxxxxx}}^{r} $} }
  \put(57.1,3.6){{$ k+1\left\{\rule[19.9pt]{0pt}{15pt}\right.$}}    
\end{picture}
\end{center}
\end{figure}
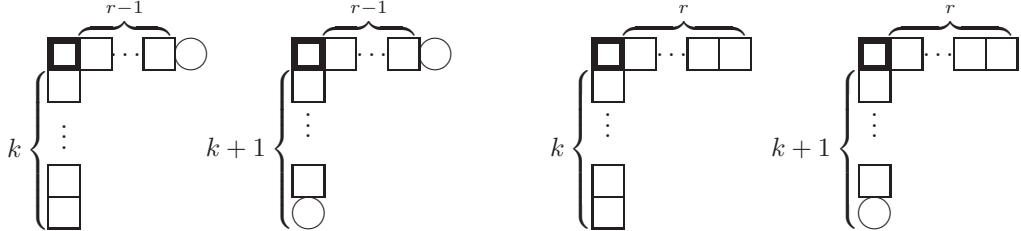
\vspace{0.5cm}

\section{Basic theory for generic $\alpha$}

In this section we develop the basic theory of the Jack polynomials with prescribed symmetry.  We assume here  that $\alpha$ is generic, which means in the present context  that $\alpha$ is either a formal parameter or a complex number that is not zero nor a negative rational.

\subsection{Compositions, partitions, and superpartition} \label{SectionAlphaGeneric}

We recall that a composition  is an ordered list of non-negative integers.  We way that $\eta=(\eta_1,\ldots,\eta_N)$ is a composition of $n$, or has degree $n$, if $$|\eta|:=\sum_{i=1}^N\eta_i=n.$$  

A partition $\la=(\la_1,\ldots,\la_N)$ of $n$ is a composition of $n$ whose elements are  decreasing: $\la_1\geq\cdots\geq \la_N\geq 0$.  The number of non-zero elements in a partition $\la$ is called the length and it is usually denoted by $\ell$ or $\ell(\lambda)$. { Each partition is associated to a diagram that contains $\ell$ rows.  The highest row, which is considered as the first one,    contains $\lambda_1$ boxes, the second row, which is just below the first one, contains $\lambda_2$ boxes, and so on, all boxes being left justified.  The box located in the $i${th} row and the $j${th}   is denoted by  $(i,j)$.  Such coordinates are also called cells.  Given a partition $\lambda$ , its conjugate $\lambda'$ is obtained by reflecting $\lambda$'s  diagram in the main diagonal. Given a cell $s=(i,j)$ in the diagram associated to $\lambda$, we let
\begin{equation*} a_\lambda(s)=\lambda_i-j\qquad a'_\lambda(s)=j-1  \qquad l_\lambda(s)=\lambda_j'-i \qquad l'_\lambda(s)=i-1 .\end{equation*}
The quantities  $a_\lambda(s), a'_\lambda(s), l_\lambda(s), l'_\lambda(s)$ are respectively called the arm-length, arm-colength, leg-length and leg-colength of $s$ in $\lambda$'s diagram.}

Note that throughout the article,  we compare the partitions by using the dominance order, which is defined in \eqref{defdominance}.

The following lemma will be used later in the article.  For $\alpha$ a formal parameter, it was first stated without proof in Stanley's article \cite{stanley}. 
\begin{lem}\label{lemordereigen}For any partition $\la$, let $$b(\la)=\sum_{i=1}^\ell(i-1)\la_i\quad\text{and}\quad \varepsilon_\la(\alpha)=\alpha b(\la')-b(\la).$$ Suppose that $\alpha$ is generic.  Then, $$ \la>\mu\quad\Longrightarrow\quad \varepsilon_\la(\alpha)\neq \varepsilon_\mu(\alpha).$$
\end{lem}
\begin{proof}Let us firs define the lowering operators as follows: 
\beq \label{opL}L_{i,j}(\ldots,\la_i,\ldots,\la_j,\ldots)=
\begin{cases} (\ldots, \la_i-1,\ldots,\la_j+1,\ldots)&\text{if } i<j\text{ and }\la_i-\la_j>1\\
              (\ldots, \la_i,\ldots,\la_j,\ldots)&\text{otherwise}.\end{cases}
							\eeq
Note that in general, if $\la$ is a partition, then $L_{i,j}\la$ is a composition. However,  from \cite[Result (1.16)]{macdonald}, one easily deduces that \beq \label{eqorderopL}\mu<\la \qquad \Longleftrightarrow\qquad \mu =L_{i_k,j_k}\circ \cdots \circ L_{i_1,j_1}\la\eeq  for some sequence  $\left((i_1,j_1),\ldots, (i_k,j_k)\right)$ such that $L_{i_{k'},j_{k'}}\circ \cdots \circ L_{i_1,j_1}\la$ is a partition for all $1\leq k'\leq k$.  

Now, let us suppose that $\bar \la=L_{i,j}\la$ is a partition for some $i<j$.  Then, $b(\bar\la)-b(\la)=j-i>0$.  This  last result together with Equation \eqref{eqorderopL} prove the following: 
$$ \mu<\la\quad \Longrightarrow\quad b(\mu)>b(\la).$$
Moreover, as is well known \cite[Result (1.11)]{macdonald}, $\la>\mu$ if and only if $\mu'>\la'$. Consequently, 
$$ \varepsilon_\la(\alpha)-\varepsilon_\mu(\alpha)=\alpha \big(b(\la')-b(\mu')\big)+b(\mu)-b(\la)=\alpha p+q,  $$ 
where $p$ and $q$ are positive integers. Therefore, $\varepsilon_\la(\alpha)-\varepsilon_\mu(\alpha)=0$ only if $\alpha$ is a negative rational, and the lemma follows. 
\end{proof}

To each composition $\eta$ corresponds a unique partition $\eta^+$, which is obtained from $\eta$ by reordering the elements of $\eta$ in decreasing order:
$$ \eta^+=(\eta^+_1,\ldots,\eta^+_N)   \quad \Longleftrightarrow  \quad\eta^+_i=\eta_{\sigma(i)}\quad \text{for some}\quad \sigma\in S_N \quad \text{such that}\quad  \eta^+_1\geq \ldots\geq \eta^+_N .$$ 
This allows to define  the dominance order between   compositions as follows: 
 $$\eta \succ  \mu \quad \Longleftrightarrow\quad  \eta^{+}>\mu^{+} \quad \text{or}\quad \eta^{+}=\mu^{+} \quad \text{and}\quad \sum_{i=1}^{k}\eta_{i}\geq \sum_{i=1}^{k}\mu_{i}\;\; \forall\, k, $$ where it is also assumed that $\eta\neq \mu$ and of the same degree.

According to Definition \ref{defsparts},  there are two useful ways of writing a superpartition $\La$ as a pair of partitions.  On the one hand,  there is the representation that  provides the correct indices for the  polynomials with prescribed  symmetry: $$\La=(\La_1,\ldots,\La_m;\La_{m+1},\ldots, \La_N)\quad\text{where}\quad \La_1\geq \cdots\geq \La_m\geq 0\quad\text{and}\quad\La_{m+1}\geq \cdots\geq \La_N\geq 0.$$  On the other hand, there is the representation naturally associated to the diagrams:  $$\La=(\La^\cd,\La^*)\quad\text{where}\quad \La^\cd_i\geq \La^\cd_{i+1},\quad \La^*_i\geq \La^*_{i+1},\quad \La^\cd_i-\La^*_i=0,1.$$
The elements of $\La=(\La_1,\ldots,\La_m;\La_{m+1},\ldots, \La_N)$ all come from those  of $\La^*$: $\La_1$ is the first elements of $\Lambda^*$  such that  $\La^\cd_i-\La^*_i=1$ for some $i$, $\La_2$ is the second, and so on, till $\La_m$, which is the smallest elements of $\La^*$ such that $\La^\cd_i-\La^*_i=1$ for some $i$; $\La_{m+1}$ is the first elements of $\Lambda^*$  such that  $\La^\cd_i-\La^*_i=0$ for some $i$, $\La_{m+2}$ is the second, and so on till $\La_N$, which is the smallest elements of $\La^*$ such that $\La^\cd_i-\La^*_i=0$
for some $i$.  Conversely,
$$ \La^\cd=(\La_1+1,\ldots,\La_m+1,\La_{m+1},\ldots, \La_N)^+\quad \text{and}\quad \La^*=(\La_1,\ldots,\La_{m},\La_{m+1},\ldots, \La_N)^+.$$

 Let $\gamma=(\gamma_1,\ldots,\gamma_N)$ be a composition of $n$.   Fix a positive integer $m$ not greater than $N$.  We define the map $\varphi_m$ as
$$ \varphi_m(\gamma)=(\Gamma^*,\Gamma^\cd),\qquad \Gamma^*=(\gamma_1,\ldots,\gamma_N)^+,\qquad \Gamma^\cd=(\gamma_1+1,\ldots,\gamma_m+1,\gamma_{m+1},\ldots,\gamma_N)^+. $$ 
In other words,  $\varphi_m$ maps the composition $\gamma$ to the superpartition $\Gamma=(\Gamma^*,\Gamma^\cd)$ of bi-degree $(n|m)$, which is  equivalent to $$\Gamma=\left( (\gamma_1,\ldots,\gamma_m)^+; (\gamma_{m+1},\ldots,\gamma_N)^+\right).$$ 

\begin{lem} \label{lemspartcomp} Let $\La=\varphi_m(\lambda)$ and $\Gamma=\varphi_m(\gamma)$, where $\lambda$ and $\gamma$ are compositions of the same degree.  If $\lambda \succ \mu$, then $\La>\Gamma$.  
\end{lem}\begin{proof} There are two possible cases.
\begin{itemize} 
\item[(1)] Suppose that  $\lambda^+ > \mu^+$. Then, obviously, $\La^*>\Ga^*$. 

\item[(2)] Suppose that (i) $\lambda^+ =\mu^+$ and  (ii) $\sum_{i=1}^{k}\lambda_{i}\geq \sum_{i=1}^{k}\mu_{i}$, $\forall \,k$. Equation (i)  implies that $\La^*=\Ga^*$.  Equation (ii) implies that $\mu$ is a permutation of $\la$ that can be written  as $$\mu=s_{i_l,j_l}\circ \cdots \circ s_{i_1,j_1}\lambda,$$
 where each $s_{i,j}$ is a transposition  such that 
$$s_{i,j}(\la_1,\ldots\la_i,\ldots,\la_j,\ldots,\la_N)=\begin{cases} (\la_1,\ldots\la_j,\ldots,\la_i,\ldots,\la_N) & \text{if } i<j \text{  and } \la_i>\la_j\\ (\la_1,\ldots\la_i,\ldots,\la_j,\ldots,\la_N)&\text{otherwise}.\end{cases}$$  Now, if $1\leq i<j\leq m$ or $m+1\leq i<j\leq N$, then $\varphi_m(s_{i,j}\la)=\La$.  This means that $s_{i,j}$ induces, via the map $\varphi_m$, a nontrivial action on the superpartition $\La$ only if $i\in I=\{1,\ldots, m\}$ and $j\in J=\{m+1,\ldots, N\}$.  To be more explicit, let $i'$ and $j'$ be such that $\varphi_m$ maps $\la_i$ to $\La_{i'}$ and $\la_j$ to $\La_{j'}$, respectively.  Then,  
$$\varphi_m(s_{i,j}\la)=\hat s_{i',j'}\varphi_m(\la)=\hat s_{i',j'}\La ,$$ where 
$$\qquad \hat s_{i',j'}\La=\begin{cases} \big((\La_1,\ldots,\La_{j'},\ldots,\La_m)^+;(\La_{m+1},\ldots,\La_{i'},\ldots,\La_N)^+\big) & \text{if } i'\in I, j'\in J \text{  and } \La_{i'}>\La_{j'}\\  \La &\text{otherwise}.\end{cases}$$
Therefore,  $\La^*=\Ga^*$ and $$\Ga=\varphi_m(\mu)=\varphi_m(s_{i_l,j_l}\circ \cdots \circ s_{i_1,j_1}\lambda)=s_{i'_l,j'_l}\circ \cdots \circ s_{i'_1,j'_1}\Lambda,$$
which implies that $\Ga^\cd<\La^\cd$, as expected.
\end{itemize}
\end{proof}

\begin{lem}\label{lemordereigen2}For any superpartition $\La$, let  $$\epsilon_\La=\sum_{s\in\La^\cd/\La^*}\big(\alpha \,a'_{\La^\cd}(s)-l'_{\La^\cd}(s)\big).$$  Suppose that $\alpha$ is generic.  Then, $$ \La^*=\Om^*\quad \text{and}\quad \La^\cd >\Om^\cd \quad\Longrightarrow\quad \epsilon_\La(\alpha)\neq \epsilon_\Om(\alpha).$$
\end{lem}
\begin{proof}Let $\Om$ be a superpartition be such that $\Om^*=\La^*$ and  $\Om^\cd=L_{i,j}\La^\cd$ for some $i<j$, where $L_{i,j}$ is the lowering operator defined in Equation \eqref{opL}.  Note that this assumption makes sense only if $\La_i^*>\La_j^*$.   Then, the diagram of $\Om^\cd$ differs from that of $\La^\cd$ only in rows $i$ and $j$, so that 
$$ \sum_{s\in\La^\cd/\La^*}a'_{\La^\cd}(s)-\sum_{s\in\Om^\cd/\Om^*}  a'_{\Om^\cd}(s)=\La^*_i-\La^*_j>0,$$
and 
$$ \sum_{s\in\La^\cd/\La^*}l'_{\La^\cd}(s)-\sum_{s\in\Om^\cd/\Om^*}  l'_{\Om^\cd}(s)= i-j<0.$$ 
Finally, recalling  \eqref{eqorderopL}, we find that
$$ \epsilon_\La(\alpha)-\epsilon_\Om(\alpha)= \alpha p+q ,\quad\text{where}\quad p,q\in\mathbb{Z}_+.$$ 
Clearly, if $\alpha$ is not a negative rational, then $\epsilon_\La(\alpha)-\epsilon_\Om(\alpha)\neq 0$, as expected.
\end{proof}

\subsection{Non-symmetric Jack polynomials}

There are many ways to define  the non-symmetric Jack polynomial \cite{opdam} (see also \cite{knop}).  The most natural for us is to characterize them as triangular eigenfunctions of commuting difference-differential, first found in physics in \cite{pasquier}, and later generalized to other root systems by Cherednik.   We define these operators as follows:
\begin{equation*} \xi_{j}= \alpha x_{j}\partial_{x_{j}}+ \sum_{i<j}\frac{x_{j}}{x_{j}-x_{i}}(1-K_{ij})+\sum_{i>j}\frac{x_{i}}{x_{j}-x_{i}}(1-K_{ij}) -(j-1),\end{equation*}
where   the operators $K_{i,j}$ give the action of the symmetric group on functions of $N$ variables, i.e.,  $$K_{i,j}f(x_{1},\ldots,x_{i},\ldots,x_{j},\ldots,x_{N})=f(x_{1},\ldots,x_{j},\ldots,x_{i},\ldots,x_{N}).$$ Note that we will use the following shorthand notation:
 $$K_{i}=K_{i,i+1}.$$

Let $\eta$ be a composition and let $\alpha$ be formal parameter or a non-zero complex number not equal to a negative rational.   Then, the non-symmetric Jack polynomials $E_\eta(x;\alpha)$, where $\eta$ is a composition, is the unique polynomial satisfying 
\begin{align*} \text{(A1')}\qquad  &E_{\eta}(x;\alpha)=x^\eta +\sum_{\nu\prec \eta} c_{\eta ,\nu} x^{\nu} \, , \qquad c_{\eta ,\nu}\in\mathbb{C}(\alpha), \\
\text{(A2')}\qquad  &\xi_{j} E_{\eta}=\overline{\eta}_{j} E_{\eta} \quad \forall j=1,\ldots,N,
\end{align*}
where the eigenvalues are given by
\begin{equation}\label{eqeigennonsym} 
\overline{\eta}_{j} = \alpha\eta_{j} - \#\{i<j|\eta_{i}\geq \eta_{j}\} - \#\{i>j|\eta_{i}> \eta_{j}\}.
\end{equation}

One important property of the non-symmetric Jack polynomials is their stability with respect to the number of variables (see  \cite[Corollary 3.3]{knop}).  To be more precise, let  $\eta=(\eta_1,\ldots, \eta_N)$ and $\eta_-=(\eta_1,\ldots, \eta_{N-1})$ be compositions.  Then,  \begin{equation}\label{stabilitynonsym}
E_{\eta}(x_{1},\ldots,x_{N})\big|_{x_{N} =0}=
\begin{cases}
0 & \text{if } \eta_{N}>0,\\
E_{\eta_{-}}(x_{1},\ldots,x_{N-1})& \text{if }\eta_{N}=0.
\end{cases}
\end{equation}
We now prove a closely related property that will help us to establish the stability of the Jack polynomials with prescribed symmetry.   

\begin{lem} \label{stabilitynonsym2} Let $\la=(\la_1,\ldots,\la_m)$ and $\mu=(\mu_{m+1},\ldots,\mu_{N-1})$ be partitions.  Let also $$\eta=(\la_m,\ldots,\la_1,0,\mu_{N-1},\ldots,\mu_{m+1})\quad\text{and}\quad\eta_-=(\la_m,\ldots,\la_1,\mu_{N-1},\ldots,\mu_{m+1}).$$  Finally assume that $\mu_{m+1}>0$.  Then,    
\begin{equation*}
E_{\eta}(x_{1},\ldots,x_{m},x_{N},x_{m+1},\ldots,x_{N-1})\big|_{x_{N=0}}=E_{\eta_{-}}(x_{1},\ldots,x_{m},x_{m+1},\ldots,x_{N-1}). 
\end{equation*}
\end{lem}
\begin{proof}
We first note that 
\begin{equation*}
E_{\eta}(x_{1},\ldots,x_{m},x_{N},x_{m+1},\ldots,x_{N-1})=K_{N-1}\ldots K_{m+1} E_{\eta}(x_{1},\ldots,x_{m},x_{m+1},\ldots,x_{N-1},x_{N}).
\end{equation*}
Now, the action of the symmetric group on the non-symmetric Jack polynomials is such that (see  \cite[Eq. (2.21)]{bf2})
\beq \label{permjnosym} 
K_{i}E_{\eta}=
\begin{cases} 
\frac{1}{\delta_{i,\eta}} E_{\eta}+(1-\frac{1}{\delta^{2}_{i,\eta}}) E_{K_{i}(\eta)},&\eta_{i}>\eta_{i+1}\\
E_{\eta} , & \eta_{i}=\eta_{i+1}\\
\frac{1}{\delta_{i,\eta}}E_{\eta}+E_{K_{i}(\eta)} , & \eta_{i}<\eta_{i+1}, 
\end{cases}
\eeq
where $\delta_{i,\eta}=\overline{\eta}_{i}-\overline{\eta}_{i+1}$. In our case, given that we are using a composition in increasing order, we can use successively  the third line of \eqref{permjnosym} and get  
\begin{equation*}
E_{\eta}(x_{1},\ldots,x_{m},x_{N},x_{m+1},\ldots,x_{N-1})= E_{K_{N-1}\ldots K_{m+1}(\eta)}(x_1,\ldots,x_N)+\sum_{\gamma}c_{\lambda,\gamma}E_{\gamma} (x_1,\ldots,x_N).
\end{equation*}
In the last equation, the sum is taken over the compositions $\gamma$ of the form $$\gamma=(\la_m,\ldots,\la_1,\omega(0,\mu_{N-1},\ldots,\mu_{m+1})),$$ where $\omega$ is a permutation given by the composition by a strict subsequence of the transpositions  $K_{N-1},\ldots ,K_{m+1}$; the coefficients $c_{\lambda,\gamma}$ are products of  $1/\delta_{i,j}$.  The important point here is that for any such $\gamma$, we have $\gamma_N\neq 0$.  Moreover, $${K_{N-1}\ldots K_{m+1}(\eta)}=(\la_m,\ldots,\la_1,\mu_{N-1},\ldots,\mu_{m+1},0).$$
   Then, applying the stability property (\ref{stabilitynonsym}), we find 
$$E_{\gamma} (x_1,\ldots,x_N)\big|_{x_N=0}=0\quad\text{and}\quad E_{K_{N-1}\ldots K_{m+1}(\eta)}(x_1,\ldots,x_N)=E_{\eta_-}(x_1,\ldots,x_N),$$
 which completes the proof.  
\end{proof}

\begin{lem}\label{lemsekiguchi}  Let $\gamma$ be a composition.  Then, $E_\gamma(x;\alpha)$ is an eigenfunction of the operators $S^*(u)$ and $ S^\cd(u,v)$ defined in   \eqref{defseki}.  Moreover, let  $\Gamma=\varphi_m(\gamma)$ be the associated superpartition to $\gamma$. 
  Then, 
$$  S^*(u)\,E_\gamma=\varepsilon_{\Gamma^*}(\alpha,u)\,E_\gamma\qquad   S^\cd(u,u)\, E_\gamma=\varepsilon_{\Gamma^\cd}(\alpha,u)\, E_\gamma,  $$
where the eigenvalue  $\varepsilon_{\lambda}(\alpha,u)$ is defined  in  \eqref{eigenseki}. 
\end{lem}
\begin{proof}The fact that the non-symmetric Jack polynomials are eigenfunctions of the Sekiguchi operators immediately follows from $\xi_i E_\gamma=\bar \gamma_i E_\gamma$.  Explicitly,
$$S^*(u)\,E_\gamma=\prod_{i=1}(u+\bar\gamma_i)\, E_\gamma,\qquad S^\cd(u,v)\, E_\gamma=\prod_{i=1}^m(u+\bar\gamma_i+\alpha)\prod_{i=m+1}^N(v+\bar\gamma_i) \, E_\gamma.$$

In order to express the eigenvalues in terms of partitions rather than composition, we need to consider permutations on words with $N$ symbols.  Amongst all the permutations $w$ such that $\gamma=w(\gamma^+)$, there exists a unique one, denoted by $w_\gamma$, of minimal length.  Equivalently, $w_\gamma$ is the smallest element of $S_N$ satisfying \begin{equation}\label{permutew} \gamma_{w_\gamma(i)}=\gamma^+_{i} 
\end{equation}
 Now, let $\delta^-=(0,1,\ldots,N-1)$.  As is well known, the eigenvalue 
$\bar\gamma_i$ is equal to the $i$th element of the composition $\left(\alpha\gamma-w_\gamma\delta^-\right)$, which means that
 $$\bar \gamma_i=\alpha \gamma^+_{w_\gamma^{-1}(i)}-\delta^-_{w_\gamma^{-1}(i)}$$ or equivalently $$  \bar\gamma_{w(i)}=\alpha \gamma^+_i-(i-1).$$  In our case, $\gamma^+=\Gamma^*$, so that  $$\prod_{i=1}(u+\bar\gamma_i)= \prod_{i=1}(u+\alpha\Gamma^*_i-i+1),$$ which is the first expected eigenvalue.
For the second Sekiguchi operator, we note that the shifted composition $(\gamma_1+1,\ldots,\gamma_m+1,\gamma_{m+1},\ldots,\gamma_N)$ is equal to $w_\gamma(\Gamma^\cd)$.  Consequently,  $$ \prod_{i=1}^m(u+\bar\gamma_i+\alpha)\prod_{i=m+1}^N(u+\bar\gamma_i) = \prod_{i=1}^N(u+\alpha \Gamma^\cd_i-i+1),$$ and the lemma follows.   
\end{proof}

\subsection{Jack polynomials with prescribed symmetry}\label{SectionDefPrescribed}

For any subset $K$ of $\{1,\ldots,N\}$, let $S_K$ denote the subgroup of the permutation group of $\{1,\ldots,N\}$  that leaves the complement of $K$ invariant.  The antisymmetrization and symmetrization operators for $K$ are defined as follows:
\beq \label{asymops}\mathrm{Asym}_K f(x)=\sum_{\sigma\in S_K}(-1)^{ \sigma} f(x_{\sigma(1)},\ldots,x_{\sigma(N)})\quad\text{and}\quad\mathrm{Sym}_K f(x)=\sum_{\sigma\in S_K}  f(x_{\sigma(1)},\ldots,x_{\sigma(N)}). \eeq 
Thus, for any pair $(i,j)$ of elements $K$, we have
$$ K_{i,j} \,\mathrm{Asym}_K f(x)=-\mathrm{Asym}_K f(x) \quad\text{and}\quad K_{i,j} \,\mathrm{Sym}_K f(x)=\mathrm{Sym}_K f(x).$$
Note that in the following paragraphs, the set $K$ will be  replaced by either   $I=\{1,\ldots,m\}$ or $J=\{m+1,\ldots,N\}$.

The vector space $\mathscr{A}_I\otimes \mathscr{S}_J|_n$ is composed of all polynomials of total degree $n$ that are antisymmetric with respect to the set of variables $\{x_1,\ldots,x_m\}$, and symmetric with respect to   $\{x_{m+1},\ldots,x_N\}$.  It is spanned  by all polynomials of the form $ \mathrm{Asym}_I \mathrm{Sym}_J x^\eta$, where $\eta$ is a composition of $n$.   However, by considering the symmetry of the polynomials, we see that $\mathscr{A}_I\otimes \mathscr{S}_J|_n$  is spanned by the following set of linearly independent polynomials:
$$ \{m^\AS_\La\, |\,\La \text{ is a strict superpartition of bi-degree }  (n|m)\},$$
where the monomial $m_\La$ is defined as  
\beq \label{defmonomialAS} m^\AS_\La(x)=a_{\lambda}(x_{1},\ldots,x_{m})m_{\mu}(x_{m+1},\ldots,x_{N}), \quad \lambda=(\La_1,\ldots,\La_m),\quad \mu=(\La_{m+1},\ldots,\La_N).\eeq 
We recall that in the last equation, $a_\lambda$ and $m_\mu$ respectively denote the antisymmetric and symmetric monomial functions.  

Similarly,   the following sets provide bases for the vector spaces $\mathscr{A}_I\otimes \mathscr{A}_J|_n$, $\mathscr{S}_I\otimes \mathscr{A}_J|_n$ , $\mathscr{S}_I\otimes \mathscr{S}_J|_n$, respectively:
\begin{align}  &\{m^\AA_\La\, |\,\La \text{ is a strict superpartition of bi-degree }  (n|m) \text{ such that } \La_{m+1}>\cdots>\La_N\},\\
 &\{m^\SA_\La\, |\,\La \text{ is a   superpartition of bi-degree }  (n|m) \text{ such that } \La_{m+1}>\cdots>\La_N\}, \\
 &\{m^\SS_\La\, |\,\La \text{ is a  superpartition of bi-degree }  (n|m) \}, 
\end{align}
where 
\begin{align}  m^\AA_\La(x)&= a_{\lambda}(x_{1},\ldots,x_{m})a_{\mu}(x_{m+1},\ldots,x_{N}) ,\\
m^\SA_\La(x)&= m_{\lambda}(x_{1},\ldots,x_{m})a_{\mu}(x_{m+1},\ldots,x_{N}) ,\\
m^\SS_\La(x)&= m_{\lambda}(x_{1},\ldots,x_{m})m_{\mu}(x_{m+1},\ldots,x_{N}) .\\
\end{align}

We recall that the Jack polynomials with prescribed symmetry AS, AA, SA, SS have been introduced in Definition \ref{defJackPrescibed}.  They are indexed by a superpartition $\La=(\La_{1}, \dots, \La_{m}; \La_{m+1},\ldots, \La_{N})$ and are defined as follows:  
\beq \label{eqprescribedO}P_\La(x;\alpha)=c_\La\, \mathcal{O}_{I,J} E_\eta, \eeq where $\mathcal{O}_{I,J}$ stands for the appropriate composition of antisymmetrization and/or symmetrization operators, and   \beq\label{etalambda} \eta=(\La_{m}, \dots, \La_{1}, \La_{N},\ldots, \La_{m+1}).\eeq  Moreover, the coefficient $c_\La $ is such 
that the polynomial $P_\La$ is monic, i.e.,  the coefficient of $m_\La$ in $P_\La$ is exactly one.  However, our definition is such that only the non-symmetric monomial $\mathcal{O}_{I,J}x^\eta$ contributes to the coefficient of $m_\La$ , so it is an easy exercise to extract the normalization coefficient:
\begin{align}  \label{coefAS}
c^\AS_\La &=  \frac{(-1)^{m(m-1)/2}}{f_{\mu}},\\
c^\AA_\La &=  {(-1)^{m(m-1)/2}} {(-1)^{(N-m)(N-m-1)/2}} ,\\
c^\SA_\La &=  \frac{(-1)^{(N-m)(N-m-1)/2}}{f_{\lambda}},\\
c^\SS_\La &=  \frac{1}{f_{\lambda}f_{\mu}}, \label{coefSS}
\end{align}
where $\lambda=(\La_1,\ldots,\La_m)$, $\mu=(\La_{m+1},\ldots,\La_N)$, $f_{\lambda}=\prod_{i}n_{\lambda}(i)! $, and $n_{\lambda}(i)$ is the multiplicity of $i$ in $\la$.

We now list some properties of the Jack polynomials with prescribed symmetry that immediately follow from their Definition \eqref{eqprescribedO}.

\begin{lem}[Regularity for generic $\alpha$]  $P_\La(x;\alpha)$ is singular only if $\alpha$ is zero or a negative rational.  
\end{lem}
\begin{proof} All the dependence upon $\alpha$ comes from the non-symmetric Jack polynomials, so it is  sufficient to consider the possible singularities of the latter.  Let us now recall a fundamental result of Knop and Sahi \cite{knop}:  There is a $v_\eta(\alpha)\in\mathbb{N}[\alpha]$ such that all the coefficients in $v_\eta(\alpha)E_\eta(x;\alpha)$ also belong to $\mathbb{N}[\alpha]$.  Thus, the only singularities of $E_\eta(x;\alpha)$ are poles, which can occurs at $\alpha=0$ or $\alpha\in\mathbb{Q}_-$.   
\end{proof}


\begin{lem}[Simple product] \label{simpleprod} For any superpartition  $\La=(\La_1,\ldots,\La_m;\La_{m+1},\ldots,\La_{N})$, let $$\La_+=(\La_1+1,\ldots,\La_m+1;\La_{m+1}+1,\ldots,\La_{N}+1).$$  
Then, $$  x_1\cdots x_N\, P_\La(x;\alpha)=P_{\La_+}(x;\alpha).$$  \end{lem}
\begin{proof}  { First, as is well known, } $x_1\cdots x_NE_\eta(x;\alpha)=E_{(\eta_1+1,\ldots,\eta_N+1)}(x;\alpha)$.  Second,  $x_1\cdots x_N$ commutes with any $\mathcal{O}_{I,J}$.  Thus, $$ x_1\cdots x_N\, P_\La(x;\alpha)=c_\La\, \mathcal{O}_{I,J} E_{(\eta_1+1,\ldots,\eta_N+1)}(x;\alpha)=\frac{c_\La}{c_{\La_+}}P_{\La_+}(x;\alpha).$$
Finally, one easily  verifies from Equations \eqref{coefAS}--\eqref{coefSS}, that  $c_\La=c_{\La_+}$.
\end{proof}

\begin{proposition}[Triangularity]\label{proptriangularity} $P_\La=m_\La+\sum_{\Ga<\La}c_{\La,\Ga}m_\Ga$. \end{proposition}
\begin{proof}By definition, $P_\La=c_\La\mathcal{O}_{I,J} E_\eta$, where $\eta$ is given by \eqref{etalambda} and  $ E_{\eta} =x^\eta +\sum_{\nu\prec \eta} c_{\eta ,\nu} x^{\nu}$.  We already know that $c_\La$ guarantees the monocity, i.e., $ c_\La\mathcal{O}_{I,J} x^\eta=m_\La$.   It remains to check that if $\nu\prec\eta$, then $\mathcal{O}_{I,J} x^\nu$ is proportional to $m_\Om$ for some $\Om<\La$. Now,  $\mathcal{O}_{I,J} x^\nu$ is proportional to $m_\Om$, where $\Om=\varphi_m(\nu)$.  Moreover, we know from Lemma \ref{lemspartcomp} that $\nu\prec \eta$, then $\varphi_m(\nu)<\varphi_m(\la)$. This completes the proof.     \end{proof}

\begin{proposition}[Stability for types AS and SS] \label{stability1} Let $\La =(\Lambda_{1},\ldots,\Lambda_{m}; \La_{m+1},\ldots,\La_{N})$ be a superparttion and let
  $\Lambda_{-}  =(\Lambda_{1},\ldots,\Lambda_{m}; \La_{m+1},\ldots,\La_{N-1})$.  Then, the Jack polynomial with prescribed symmetry AS or SS satisfies
\begin{equation*} P_{\Lambda}(x_{1},\ldots,x_{N};\alpha)\big|_{x_{N} =0} = \begin{cases}0,&  \La_{N}>0,\\
P_{\Lambda_{-}}(x_{1},\ldots,x_{N-1};\alpha) ,& \La_{N}=0.\end{cases} \end{equation*} 
\end{proposition}
\begin{proof}The cases AS and SS being similar, we  only give the proof for AS. 

Let $\la=(\Lambda_{1},\ldots,\Lambda_{m})$,  $\mu= (\La_{m+1},\ldots,\La_{N})$,  $\la^-=(\Lambda_{m},\ldots,\Lambda_{1})$,  $\mu^-= (\La_{N},\ldots,\La_{m+1})$.  Let also   $\eta=(\lambda^{-},\mu^{-})$  and   $\eta_{-}=(\lambda^{-},\mu^{-}_{-})$, where $\mu^{-}_{-}=(\mu_{N-1},\ldots,\mu_{m+1})$.
By definition, 
\begin{equation*}
 P_\Lambda^{AS}(x) = \frac{(-1)^{m(m-1)/2}}{f_{\mu}}\mathrm{Asym}_{I}\mathrm{Sym}_{J}E_{\eta}(x;\alpha)
\end{equation*}
The symmetrization operator  can be decomposed  as 
$$ \mathrm{Sym}_{J}=\mathrm{Sym}_{J_-}(1+K_{m+1,N}+K_{m+2,N}+ \ldots + K_{N-1,N}),\qquad J_-=\{m+1,\ldots, N-1\}.$$
It is  more convenient to rewrite the transpositions on the LHS in terms of the elementary transpositions: 
\begin{equation*} K_{i,N}=K_{i}K_{i+1} \ldots K_{N-2} K_{N-1}K_{N-2} \ldots K_{i+1}K_{i}\end{equation*}
By making use of the stability property \eqref{stabilitynonsym} and the action of the symmetric group on the non-symmetric jack polynomials given in \eqref{permjnosym} , we then find that  
$$
K_{N-1}K_{N-2} \ldots K_{i+1}K_{i}E_{\eta}\big|_{x_{N} =0}=
\begin{cases}
0 ,&\eta^{}_{i}>0,\\
E_{\eta^{}_{-}}(x_{1},\ldots,x_{N-1}), &\eta^{}_{i}=0.
\end{cases}
$$
Thus, $\mathrm{Sym}_{J}E_{\eta^{}}(x_{1},\ldots,x_{N})\big|_{x_{N} =0}=0 $ when $\La_N>0$, while 
\begin{multline*} \mathrm{Sym}_{J}E_{\eta^{}}(x_{1},\ldots,x_{N})\big|_{x_{N} =0} \\= \mathrm{Sym}_{J_{-}}\Big(\sum_{\stackrel{i \in \{m+1,\ldots,N-1\}}{\mu^{-}_{i}=0}}K_{i}K_{i+1} \ldots K_{N-2}\Big)E_{\eta^{}_{-}}(x_{1},\ldots,x_{N-1}) = n_\mu(0)\,\mathrm{Sym}_{J_{-}} E_{\eta^{}_{-}}(x_{1},\ldots,x_{N-1})  \end{multline*}  
when $\La_N=0$, and the proposition follows.
\end{proof}

\begin{proposition}[Stability for types SA and SS] \label{stability2} Let $\La =(\Lambda_{1},\ldots,\Lambda_{m}; \La_{m+1},\ldots,\La_{N})$ be a superpartition
and let  $\Lambda_{-}  =(\Lambda_{1},\ldots,\Lambda_{m-1}; \La_{m+1},\ldots,\La_{N})$.  Then, the Jack polynomial with prescribed symmetry SA or SS satisfies
\begin{equation*} P_{\Lambda}(x_{1},\ldots,x_m,\ldots ,x_{N};\alpha)\big|_{x_{m} =0} = \begin{cases}0,&\La_m>0,\\
P_{\Lambda_{-}}(x_{1},\ldots,x_{m-1},x_{m+1},\ldots,x_{N-1};\alpha),&\La_m=0. \end{cases}  \end{equation*} 
\end{proposition}
\begin{proof}  The   cases SA and SS are almost identical, so we only prove the first.  Below, we essentially follow the method used for proving Proposition \ref{stability1}, except that we use Lemma \ref{stabilitynonsym2} rather than Equation \eqref{stabilitynonsym}.

Let $\la=(\Lambda_{1},\ldots,\Lambda_{m})$,  $\mu= (\La_{m+1},\ldots,\La_{N})$,  $\la^-=(\Lambda_{m},\ldots,\Lambda_{1})$,  $\mu^-= (\La_{N},\ldots,\La_{m+1})$.  Let also   $\eta=(\lambda^{-},\mu^{-})$  and   $\eta_{-}=(\lambda^{-},\mu^{-}_{-})$, where $\mu^{-}_{-}=(\mu_{N-1},\ldots,\mu_{m+1})$.
By definition, 
\begin{equation*}
 P_\Lambda^{SA}(x) = \frac{(-1)^{(N-m)(N-m-1)/2}}{f_{\lambda}}\mathrm{Sym}_{I}\mathrm{Asym}_{J}E_{\eta}(x;\alpha)
\end{equation*}
Note that $\mathrm{Sym}_{I}$ and $\mathrm{Asym}_{J}$ commute. 
The symmetrization operator  can be decomposed  as 
$$ \mathrm{Sym}_{I}=\mathrm{Sym}_{I_-}(1+K_{1,m}+K_{2,m}+ \ldots + K_{m-1,m}),$$ where $ I_-=\{1,\ldots, m-1\}$ and 
\begin{equation*} K_{i,m}=K_{i}K_{i+1} \ldots K_{m-2} K_{m-1}K_{m-2} \ldots K_{i+1}K_{i}\end{equation*}
Now, recalling \eqref{permjnosym} and the second stability property for the non-symmetric Jack polynomials, given in Lemma \ref{stabilitynonsym2}, we conclude that
$$
K_{m-1}K_{m-2} \ldots K_{i+1}K_{i}E_{\eta}\big|_{x_{m} =0}=
\begin{cases}
0 ,&\eta^{}_{i}>0,\\
E_{\eta_{-}}(x_{1},\ldots,x_{N-1}), &\eta^{}_{i}=0.
\end{cases}
$$
Thus, $\mathrm{Sym}_{I}E_{\eta^{}}(x_{1},\ldots,x_{N})\big|_{x_{m} =0}=0$ when  $\La_m>0$, while
\begin{multline*} \mathrm{Sym}_{I}E_{\eta^{}}(x_{1},\ldots,x_{N})\big|_{x_{m} =0} \\= \mathrm{Sym}_{I_{-}}\Big(\sum_{\stackrel{i \in \{1,\ldots,m-1\}}{\la^{-}_{i}=0}}K_{i}K_{i+1} \ldots K_{m-2}\Big)E_{\eta^{}_{-}}(x_{1},\ldots,x_{m-1},x_{m+1},\dots, x_{N}) \\= n_\la(0)\,\mathrm{Sym}_{I_{-}} E_{\eta_-}(x_{1},\ldots,x_{m-1},x_{m+1},\dots, x_{N}), \end{multline*}  when $\La_m=0$, 
and the proposition follows.
\end{proof}

The next proposition relates Jack polynomials with prescribed symmetry of different bi-degrees. It uses two basic operation on superpartitions.  The first one is the removal of a column:  $$\mathcal{C}(\La_1,\ldots,\La_m;\La_{m+1},\ldots,\La_N)=(\La_1-1,\ldots,\La_m-1;\La_{m+1}-1,\ldots,\La_N-1) \quad\text{if}\quad \La_i>0\quad\forall \,1\leq i\leq N .$$
 The second one is the removal of a circle:  
$$\tilde{\mathcal{C}}(\La_1,\ldots,\La_m;\La_{m+1},\ldots,\La_N)=(\La_1,\ldots,\La_{m-1};\La_{m+1},\ldots,\La_N) \quad\text{if}\quad \La_m=0.$$  The  operators $ \mathcal{C}$ and $\tilde{\mathcal{C}}$ are illustrated in Figure \ref{FigRemove}

\begin{figure}[h]
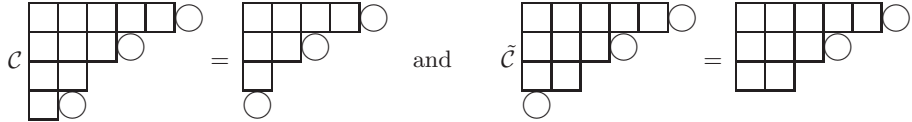

\caption{Operators $ \mathcal{C}$ and $\tilde{\mathcal{C}}$}\label{FigRemove}
{\small$$ \mathcal{C}\; {\tableau[scY]{&&&&&\bl\tcercle{} \\&&&\bl\tcercle{}\\&\\&\bl\tcercle{}\\}}= \;{\tableau[scY]{&&&&\bl\tcercle{} \\&&\bl\tcercle{}\\\\ \bl\tcercle{}\\}} \quad\text{and}\qquad  \tilde{\mathcal{C}}\; {\tableau[scY]{&&&&&\bl\tcercle{} \\&&&\bl\tcercle{}\\&\\\bl\tcercle{}\\}}=\;{\tableau[scY]{&&&&&\bl\tcercle{} \\&&&\bl\tcercle{}\\&\\\bl \\}} $$}
\end{figure}

\begin{proposition}[Removal of a column or a circle] \label{removecircle} Let $\La =(\Lambda_{1},\ldots,\Lambda_{m}; \La_{m+1},\ldots,\La_{N})$ be a superpartition and let $P_{\Lambda}(x_{1},\ldots,x_m,\ldots ,x_{N};\alpha)$ be the associated Jack polynomial with prescribed symmetry AA, AS, SA,or SS. \\

\noindent If $\La_i>0$ for all $1\leq i\leq N$, then  
$$\displaystyle P_{\Lambda}(x_{1},\ldots,x_m,\ldots ,x_{N};\alpha)=(x_1\cdots x_N) P_{\mathcal{C}\Lambda}(x_{1},\ldots,x_m,\ldots ,x_{N};\alpha).$$ 

\noindent If $\La_m=0$, then  
$$\displaystyle P_{\Lambda}(x_{1},\ldots,x_m,\ldots ,x_{N};\alpha)\Big|_{x_m=0}=\epsilon_m P_{\tilde{\mathcal{C}}\Lambda}(x_{1},\ldots,x_{m-1},x_{m+1},\ldots,x_{N-1};\alpha),$$ where 
$\epsilon_m=(-1)^{m(m-1)/2}$ for types AA and AS, while $\epsilon_m=1$ for types SA and SS. 
\end{proposition}
\begin{proof}  The removal of a column follows immediately from Lemma \ref{simpleprod}. For types SA and SS, the removal of a circle follows from the stability property given in Proposition \ref{stability2}.  

It remains to prove the removal of a circle for types AA and AS. Only the AS case is detailed below. Let $\la=(\Lambda_{1},\ldots,\Lambda_{m})$,  $\mu= (\La_{m+1},\ldots,\La_{N})$,  $\la^-=(\Lambda_{m},\ldots,\Lambda_{1})$,  $\mu^-= (\La_{N},\ldots,\La_{m+1})$.  Let also   $\eta=(\lambda^{-},\mu^{-})$  and   $\eta_{-}=(\lambda^{-}_{-},\mu^{-})$, where $\la^{-}_{-}=(\la_{m-1},\ldots,\la_{1})$.
By definition, 
\begin{equation*}
 P_\Lambda^{AS}(x) = \frac{(-1)^{(m)(m-1)/2}}{f_{\mu}}\mathrm{Asym}_{I}\mathrm{Sym}_{J}E_{\eta}(x;\alpha)
\end{equation*}
Note that $\mathrm{Asym}_{I}$ and $\mathrm{Sym}_{J}$ commute. 
The symmetrization operator  can be decomposed  as 
$$ \mathrm{Asym}_{I}=\mathrm{Asym}_{I_-}(1-K_{1,m}-K_{2,m}- \ldots - K_{m-1,m}),$$ where $ I_-=\{1,\ldots, m-1\}$ and 
\begin{equation*} K_{i,m}=K_{i}K_{i+1} \ldots K_{m-2} K_{m-1}K_{m-2} \ldots K_{i+1}K_{i}\end{equation*}
Now, recalling Equation \eqref{permjnosym} and the second stability property for the non-symmetric Jack polynomials, given in Lemma \ref{stabilitynonsym2}, we conclude that
$$
K_{m-1}K_{m-2} \ldots K_{i+1}K_{i}E_{\eta}\big|_{x_{m} =0}=
\begin{cases}
0 ,&\eta^{}_{i}>0,\\
E_{\eta_{-}}(x_{1},\ldots,x_{N-1}), &\eta^{}_{i}=0.
\end{cases}
$$
From the previous line, we can see that the only nonzero contribution comes from the permutation $K_{m-1}K_{m-2} \ldots K_{2}K_{1}$. Thus 
\begin{equation*}
\begin{split}
 \mathrm{Asym}_{I}E_{\eta^{}}(x_{1},\ldots,x_{N})\big|_{x_{m} =0} &=\mathrm{Asym}_{I_{-}}(K_{1}K_{2} \ldots K_{m-2} E_{\eta^{}_{-}}(x_{1},\ldots,x_{m-1},x_{m+1},\dots, x_{N}))\\
&= (-1)^{m-2}\mathrm{Asym}_{I_{-}}E_{\eta^{}_{-}}(x_{1},\ldots,x_{m-1},x_{m+1},\dots, x_{N})
\end{split}
\end{equation*}  
and the proposition follows.
\end{proof}

\begin{proposition}[Eigenfunctions] \label{lemsekiguchi2}  The Jack polynomial with prescribed symmetry,  $P_\La=P_\La(x;\alpha)$,   is an  eigenfunction of the Sekiguchi operators $S^*(u)$ and $  S^\cd(u,v)$ defined in Equation \eqref{defseki}.   Moreover,  
$$  S^*(u)\,P_\La=\varepsilon_{\La^*}(\alpha,u)\,P_\La\qquad   S^\cd(u,u)\, P_\La=\varepsilon_{\La^\cd}(\alpha,u)\, P_\La, $$
where the eigenvalues are given by Equation \eqref{eigenseki}.  
\end{proposition}
\begin{proof}This lemma immediate follows from the following three basic facts: 
\begin{itemize}\item[(1)] $P_\La$ is proportional to $\mathcal{O}_{I,J} E_\la$ for any composition $\la$ such that $\La=\varphi_m(\la)$;
\item[(2)]   The operators $S^*$ and $S^\cd$  commute with $\mathcal{O}_{I,J}$.  
\item[(3)]  By virtue of Lemma \ref{lemsekiguchi}, $E_\la$ is an eigenfunction of $S^*(u)$ and $  S^\cd(u,v)$.  Moreover, if $\varphi_m(\la)=\La$, then   $S^*(u)\,E_\la=\varepsilon_{\La^*}(\alpha,u)\,E_\la$ and $S^\cd(u,u)\, E_\la=\varepsilon_{\La^\cd}(\alpha,u)\, E_\la$.
 \end{itemize}  
\end{proof}

\begin{proof}[Proof of Theorem \ref{theo1}] We want to prove that  the Jack polynomials with prescribed symmetry are the unique unitrangular eigenfunctions of   $\mathcal{H}=\sum_{i=1}^N\xi_i^2$ and $\mathcal{I}=\sum_{i=1}^m\xi_i$.  However, according to Propositions \ref{proptriangularity}   and  \ref{lemsekiguchi2},  we already know that the Jack polynomial with prescribed symmetry   $P_\La$ satisfies   
\begin{align*} \text{(B1)} \qquad& P_\Lambda=m_\Lambda + \sum_{\Gamma<\Lambda}c_{\La,\Gamma}m_\Gamma ;\\
\text{(B2)} \qquad&\mathcal{H}\,P_\La=\,d_\La\,P_\La\quad \quad \text{and} \quad \mathcal{I}\,P_\La=e_\La\, P_\La .
\end{align*} 
Thus, it remains to prove that there is no other polynomial that satisfies (B1) and (B2).  

First, we  need to determine precisely the eigenvalues  
$d_{\La}$ and $ e_\La$. We recall that  $m_\Lambda$ is proportional to $\mathcal{O}_{I,J}x^\eta$, where $ \eta=(\La_{m}, \dots, \La_{1}, \La_{N},\ldots, \La_{m+1})$.  Now, as is well known { (e.g., see conditions (A1') and (A2') in Section 2.2)}, 
$$ \xi_i x^\eta=\bar \eta_i x^\eta+\sum_{\ga\prec\eta}f_{\eta,\ga}x^\ga.$$ 
Then, for any polynomial $g$ such that  $g(\xi_1,\ldots,\xi_N)$ commutes with $\mathcal{O}_{I,J}$,  we have
\begin{multline}\label{generaltriangular} g(\xi_1,\ldots,\xi_N) m_\La\propto \mathcal{O}_{I,J} g(\xi_1,\ldots,\xi_N) x^\eta \\
= \mathcal{O}_{I,J}\left( g(\bar\eta_1,\ldots,\bar\eta_N) x^\eta +\sum_{\ga\prec\eta}f'_{\eta,\ga}x^\ga\right)\propto  g(\bar\eta_1,\ldots,\bar\eta_N) m_\La+\sum_{\Ga<\La}f''_{\La,\Om}\, m_\Om\end{multline} 
Consequently, a triangular polynomial  $Q_\Lambda=m_\Lambda + \sum_{\Gamma<\Lambda}c'_{\La,\Gamma}m_\Gamma$, can be an eigenfunction of $g(\xi_1,\ldots,\xi_N)$  only if its eigenvalue is equal to $g(\bar\eta_1,\ldots,\bar\eta_N)$.  In our case, $Q$ is an eigenfunction of  $\mathcal{H}$ and $\mathcal{I}$, with respective eigenvalues $d_\La$ and $e_\La$,   only if 
$$ d_\La=\sum_{i=1}^N\bar \eta_i^2 \quad\text{and}\quad e_\La=\sum_{i=1}^m\bar \eta_i.$$ Now, as explained in Lemma \ref{lemsekiguchi}, 
$ \sum_{i=1}^N\bar \eta_i^2 =\sum_{i=1}^n\big(\alpha \La_i^*-(i-1)\big)^2$. {By comparing the latter equation  with the explicit expression for the quantity $\varepsilon_\La(\alpha)$, introduced in Lemma \ref{lemordereigen},  we  get  
\beq \label{eqeigen1}  d_\La= 2\,\alpha\, \varepsilon_{\La^*}(\alpha)+\alpha^2 |\La^*|+\frac{N(N-1)(2N-1)}{6}.\eeq 
Returning to the second eigenvalue, we note that because $ \eta=(\La_{m}, \dots, \La_{1}, \La_{N},\ldots, \La_{m+1})$, we can write 
$$  \sum_{i=1}^m\bar \eta_i=\sum_{i}^m\left(\alpha\La_i-\#\{j\,|\,\La_j\geq \La_i\}\right).$$ 
From the comparison of the latter expression  with    the quantity $\epsilon_\La(\alpha)$, given in  Lemma \ref{lemordereigen2}, we then conclude that
\beq \label{eqeigen2}   e_\La= \epsilon_\La(\alpha).\eeq }  

Second, we suppose that there is another $Q_\Lambda=m_\Lambda + \sum_{\Gamma<\Lambda}c'_{\La,\Gamma}m_\Gamma$ such that (i) $P_\La-Q_\La\neq 0$, (ii) $\mathcal{H}\,Q_\La=\,d_\La\,Q_\La$, and (iii) $\mathcal{I}\,Q_\La=e_\La\, Q_\La$.  Condition (i) implies that there is superpartition $\Om$ such that $\Om<\La$ and 
$$ P_\La-Q_\La =a_\Om m_\Om+\sum_{\substack{\Ga<\La\\\Ga<_t\Om}}a_{\Om,\Ga} m_\Ga,$$
where $<_t$ denotes some total order compatible with the dominance order. The substitution of the  last equation into conditions (ii) and (iii) then leads to
\begin{align}\label{eqHm}\mathcal{H}\Big(a_\Om m_\Om+\sum_{\substack{\Ga<\La\\\Ga<_t\Om}}a_{\Om,\Ga} m_\Ga\Big)&=d_\La \Big(a_\Om m_\Om+\sum_{\substack{\Ga<\La\\\Ga<_t\Om}}a_{\Om,\Ga} m_\Ga\Big) \\ \label{eqIm}
\mathcal{I}\Big(a_\Om m_\Om+\sum_{\substack{\Ga<\La\\\Ga<_t\Om}}a_{\Om,\Ga} m_\Ga\Big)&=e_\La \Big(a_\Om m_\Om+\sum_{\substack{\Ga<\La\\\Ga<_t\Om}}a_{\Om,\Ga} m_\Ga\Big). 
\end{align}
However, according to Equation \eqref{generaltriangular}, we have $\mathcal{H} m_\Om=d_\La m_\La+\ldots$ and  $\mathcal{I} m_\Om=e_\La m_\La+\ldots$, where the ellipsis $\ldots$ stand for  linear combinations of monomial indexed by superpartitions strictly smaller than $\Om$ in the dominance order. Consequently, Equations \eqref{eqHm} and \eqref{eqIm} can be rewritten as 
\begin{align*} d_\Om \,a_\Om \, m_\Om+\text{independent terms}&=d_\La \, a_\Om \, m_\Om+\text{independent terms},\\ 
 e_\Om\, a_\Om\, m_\Om+\text{independent terms}&=e_\La\,  a_\Om \,m_\Om+\text{independent terms}, 
\end{align*}
which is possible only if 
$$ d_\La =d_\Om\quad \text{and}\quad e_\La =e_\Om $$
On the one hand, using Lemma \ref{lemordereigen} and $\La>\Om$, we conclude that the first equality is possible  only if $\La^*=\Om^*$.  On the other hand, Lemma \ref{lemordereigen2} and $\La>\Om$ imply that, the second equality is possible  only if $\La^*>\Om^*$.  We thus have a contradiction.  Therefore, there is no polynomial $Q_\La$ satisfying (i), (ii), and (iii).  We have proved the uniqueness of the polynomial satisfying (B1) and (B2).
\end{proof}

\section{Regularity and uniqueness properties at $\alpha=-(k+1)/(r-1)$}

 As mentioned in the Introduction, regularity and uniqueness are  obvious properties only when $\alpha$ is  generic, which means when $\alpha$ is a complex number that is neither zero nor a negative rational.  On the one hand,  non-symmetric Jack polynomials may have poles only for non-generic values of $\alpha$, and when poles occur, then there is non-uniqueness.   Indeed,  following the  argument  \cite[Lemma 2.4]{FJMM}, one easily sees that if the non-symmetric Jack polynomial $E_\eta$ has a pole at some given value of $\alpha_0$, then there exits a composition $\nu\prec\eta$ such that $\varepsilon_{\eta^+}(\alpha_0,u)=\varepsilon_{\nu^+}(\alpha_0,u)$.  On the other hand,  for non-generic values of $\alpha$, non-uniqueness may be observed even for regular  polynomials.      As a basic example, consider   the compositions $\eta=(2,0)$ and $\nu=(1,1)$, which satisfy $\eta \succ \nu$. One  can verify that $E_{\eta}(x_{1},x_{2};\alpha)$ and
$E_{\nu}(x_{1},x_{2};\alpha)$ are regular at $\alpha=0$.  These polynomials nevertheless share the same eigenvalues, i.e., $\overline{\eta}_{j}|_{\alpha=0}=\overline{\nu}_{j}|_{\alpha=0}$ for $j=1,2.$. Hence,  at $\alpha=0$,   any polynomial of the form $E_{\eta}(x_{1},x_{2};\alpha)+a E_{\nu}(x_{1},x_{2};\alpha)$ complies with conditions (A1') and (A2') of Section 2.2, so uniqueness is lost.

Here we find sufficient conditions that allow to preserve both the regularity and the uniqueness.  We indeed prove that if $\alpha=-(k+1)/(r-1)$ and $\La$ is $(k,r,N)$-admissible, then the associated   Jack polynomial with prescribed symmetry is regular and can be characterized as  the unique triangular eigenfunction to differential operators of Sekiguchi type.   Similar results hold for the non-symmetric Jack polynomials.  We use them at the end of the section  to prove the clustering properties for $k=1$.

\subsection{More on admissible superpartitions} 

\begin{lem} \label{leminequality1} Let $\Lambda$ be a weakly $(k,r,N)$-admissible and strict superpartition.  Then both $\Lambda^*$ and $\Lambda^\cd$ are $(k+1,r,N)$-admissible. 
\end{lem} 
\begin{proof} 
According to the weak admissibility condition, we have   $\Lambda^\cd_{i+1}-\Lambda_{i+1+k}^*\geq r$, so that 
$  \Lambda^*_{i}-\Lambda_{i+1+k}^*\geq \Lambda^*_{i+1}-\Lambda_{i+1+k}^*\geq  r-1$.  Now, the equality $\Lambda^*_{i+1}-\Lambda_{i+1+k}^*= r-1$  holds if and only if  $\Lambda^\cd_{i+1}=\Lambda^*_{i+1}+1$.  However,  in the latter case,  $\Lambda^*_{i}\geq \Lambda^\cd_{i+1}>\Lambda^*_{i+1}$.  We therefore have $\Lambda^*_{i}-\Lambda_{i+k+1}^*\geq  r$.

Similarly, we have $\Lambda^\cd_{i}-\Lambda_{i+k}^\cd \geq  r-1$.  The equality $\Lambda^\cd_{i}-\Lambda_{i+k}^\cd =  r-1$ occurs  if and only if $\Lambda^\cd_{i+k}=\Lambda^*_{i+k}+1$, but in this case, $\Lambda^\cd_{i+k}> \Lambda^*_{i+k}>\Lambda^\cd_{i+k+1}$.  Therefore, $\Lambda^\cd_{i}-\Lambda_{i+k+1}^\cd \geq r$.
\end{proof} 

\begin{lem} \label{leminequality2}{  If $\Lambda$ is  $(k,r,N)$-admissible, then
\begin{equation} \label{eqlrminequality} \Lambda^\cd_{i+1}-\Lambda_{i+\rho(k+1)}^*\geq \rho r, \qquad 1\leq i\leq N-\rho(k+1), \, \rho \in \mathbb{Z}_+, \end{equation} or equivalently, 
\begin{equation} \label{eqlrminequality2} \Lambda^\cd_{i-\rho(k+1)}-\Lambda_{i-1}^*\geq \rho r, \qquad \rho(k+1)\leq i-1\leq  N, \, \rho \in \mathbb{Z}_+. \end{equation}}
{In particular, if $\Lambda$ is moderately $(k,r,N)$-admissible, then Equations \eqref{eqlrminequality} and \eqref{eqlrminequality2} hold.}
\end{lem}
\begin{proof} {The moderately and strongly admissible cases are trivial.}  We thus suppose that $\La$ is strict and weakly $(k,r,N)$-admissible.   Firstly,  note that the case $\rho=1$ corresponds to $\Lambda^\cd_{i+1}-\Lambda_{i+k+1}^*\geq r$, which is an immediate consequence of weak admissibility condition. Secondly,  suppose that Eq.\ \eqref{eqlrminequality} is true for some $\rho\geq  1$.
Then, 
$$\Lambda^\cd_{i+1}-\Lambda_{i+(\rho+1)(k+1)}^*=\Lambda^\cd_{i+1}-\Lambda_{i+\rho(k+1)}^*+\Lambda_{i+\rho(k+1)}^*-\Lambda_{i+(\rho+1)(k+1)}^*  \geq \rho r + \Lambda_{i+\rho(k+1)}^*-\Lambda_{i+(\rho+1)(k+1)}^* .$$
However, according to the previous lemma, $ \Lambda_{i+\rho(k+1)}^*-\Lambda_{i+(\rho+1)(k+1)}^* \geq r$.  Consequently, 
$$\Lambda^\cd_{i+1}-\Lambda_{i+(\rho+1)(k+1)}^*\geq \rho r+r,$$  and the lemma follows by induction. 
\end{proof}

{ 
\subsection{Regularity for non-symmetric Jack polynomials}

To demonstrate that some non-symmetric Jack polynomials have no poles, it is necessary to introduce some notation. Let $\eta$ be a composition.  For each cell $s=(i,j)$ in $\eta$'s diagram, we define
\begin{equation*}
\begin{split}
a_{\eta}(s) &=\eta_{i}-j \\
l^{1}_{\eta}(s) &=\#\{k=1,\ldots,i-1|j\leq \eta_{k}+1\leq \eta_{i}\} \\
l^{2}_{\eta}(s) &=\#\{k=i+1,\ldots,N|j\leq \eta_{k}\leq \eta_{i}\} \\
\overline{l}_{\eta}(s) &=l^{1}_{\eta}(s)+l^{2}_{\eta}(s) \\
d_{\eta}(s) &=\alpha(a_{\eta}(s)+1)+\overline{l}_{\eta}(s)+1. 
\end{split}
\end{equation*}

According to the results given in \cite{knop}, we know that $\left(\prod_{s\in \eta}d_{\eta}(s)\right)E_{\eta}$ belongs to $\mathbb{N}[\alpha,x_{1},\ldots,x_{N}]$. Then, if we want to show that $E_{\eta}(x;\alpha)$ has no poles at $\alpha=\alpha_{k,r}$ is sufficient to prove that 
\begin{equation*} \prod_{s\in \eta}d_{\eta}(s) \neq 0 \quad \text{if}\quad  \alpha=\alpha_{k,r}  \end{equation*}

Note that in what follows,  $\lambda^{+}=(\lambda^{+}_1, \ldots,\lambda^{+}_m)$ and $\mu^{+}=(\mu^{+}_1,\ldots,\mu^+_{N-m})$   denote partitions. This notation is used in order to avoid confusion between partitions and compositions.  Moreover, we denote the composition obtained by the concatenation of $\lambda^{+}$ and $\mu^{+}$, which is  $(\lambda^{+}_1, \ldots,\lambda^{+}_m,\mu^{+}_1,\ldots,\mu^+_{N-m})$, as follows:  \beq \label{assumeeta++} \eta=(\lambda^{+},\mu^{+}) .\eeq  

\begin{lem} Let $\eta$ be as in \eqref{assumeeta++}   and  let $\Lambda=\varphi_{m}(\eta)$ be its associated superpartition.  Moreover,  let $\mathrm{BF}(\Lambda)$ be the set of cells belonging simultaneously to a bosonic row (without circle) and a fermionic column (with circle). Then, 
\begin{equation*} \prod_{s\in \eta}d_{\eta}(s) =\prod_{s' \in \mathrm{BF}(\Lambda)} (\alpha(a_{\Lambda^{*}}(s')+1)+l_{\Lambda^{\circledast}}(s')+1) \prod_{s' \in \Lambda^{*}/\mathrm{BF}(\Lambda)}(\alpha(a_{\Lambda^{*}}(s')+1)+l_{\Lambda^{*}}(s')+1) \end{equation*} \label{d_etaassuperpartition}
\end{lem}
\begin{proof}
Given a cell $s=(i,j)$ in $\eta$, let $s'=(i',j)$ be the associated cell in $\Lambda$. We want to  express $d_{\eta}(s)$ as a  function of the arm-length and leg-length of the cell $s'$ in $\La$.
For each cell $s=(i,j)$ in $\eta$, we have $a_{\eta}(s)=a_{\Lambda^{*}}(s')$, while we can rewrite $\overline{l}_{\eta}(s)$ as
\begin{multline}\label{eqjess0}
\overline{l}_{\eta}(s)=\#\{k=1,\ldots,i-1|j=\eta_{k}+1\}\\ +\#\{k=1,\ldots,i-1|j\leq \eta_{k}\leq \eta_{i}-1\}
                      +\#\{k=i+1,\ldots,N|j\leq \eta_{k}\leq \eta_{i}\}.
\end{multline}
The two last terms can be easily expressed $\overline{l}_{\eta}(s)$ with the help of the leg-length of the cell $s'$:
\begin{equation}\label{eqjess1}\#\{k=1,\ldots,i-1|j\leq \eta_{k}\leq \eta_{i}-1\}+\#\{k=i+1,\ldots,N|j\leq \eta_{k}\leq \eta_{i}\}=l_{\Lambda^{*}}(s').\end{equation}
However, for the  first term, we have to distinguish two cases: 
\begin{itemize}
	\item [(i)] If $s=(i,j)$ is such that $j=\eta_{k}+1$ for some $1\leq k\leq i-1$, then it is clear that $s' \in BF(\Lambda)$.  Moreover,
	$$\#\{k=1,\ldots,i-1|j=\eta_{k}+1\}=\#\{k=1,\ldots,m|j=\lambda_{k}+1\}.$$
Since $\#\{k=1,\ldots,m|j=\lambda_{k}+1\}$ counts the number of circles that appear in the column $j$ in $\La$ --more specifically, in the leg-length of the cell $s'$-- we conclude that $\overline{l}_{\eta}(s)=l_{\Lambda^{\circledast}}(s')$. Thus, 
\begin{equation}\label{eqjess2}
d_{\eta}(s)=\alpha(a_{\Lambda^{*}}(s')+1)+l_{\Lambda^{\circledast}}(s')+1.
\end{equation}
\item [(ii)] If $s=(i,j)$ is such that $j \neq \eta_{k}+1$ for each $k=1,\ldots,i-1$,  then it is clear that $s' \in \Lambda^{*}/\mathrm{BF}(\Lambda)$ and also $\overline{l}_{\eta}(s)=l_{\Lambda^{*}}(s')$. Hence, we conclude that
\begin{equation}\label{eqjess3}
d_{\eta}(s)=\alpha(a_{\Lambda^{*}}(s')+1)+l_{\Lambda^{*}}(s')+1.
\end{equation}
\end{itemize}
The substitution of Equations  \eqref{eqjess1}--\eqref{eqjess3}  into \eqref{eqjess0} completes the proof.
\end{proof}

\begin{remark}
it can be shown that if $\Lambda$ is a superpartition such that $\Lambda^{\circledast}_{i}-\Lambda_{i+k}^*\geq r-1$ for each $i=1,\ldots,N-k$, then for each cell $s' \in\Lambda^{*}/\mathrm{BF}(\Lambda)$ we have $\alpha_{k,r}(a_{\Lambda^{*}}(s')+1)+l_{\Lambda^{*}}(s')+1\neq 0$.  However,    we do not have necessarily $\alpha_{k,r}(a_{\Lambda^{*}}(s')+1)+l_{\Lambda^{\circledast}}(s')+1 \neq 0$ if $s' \in \mathrm{BF}(\Lambda)$.
\end{remark}

\begin{lem}\label{regjnonsym} Let $\eta$ be as in \eqref{assumeeta++}   and  let $\Lambda=\varphi_{m}(\eta)$ be its associated superpartition. If $\Lambda$ is strict and weakly $(k,r,N)$-admissible or if moderately $(k,r,N)$-admissible, then $E_{\eta}(x;\alpha)$ does not have poles at $\alpha=\alpha_{k,r}$. \end{lem}
\begin{proof} 
As we have mentioned earlier (see \cite{knop}), to prove that $E_{\eta}(x;\alpha)$ has no poles at $\alpha=\alpha_{k,r}$, it  is sufficient to show that 
$\prod_{s\in \eta}d_{\eta}(s) \neq 0$ if $\alpha=\alpha_{k,r}$.  

Let us  suppose that $\prod_{s\in \eta}d_{\eta}(s)= 0 $ when $\alpha=\alpha_{k,r}$. From the equality obtained in Corollary \ref{d_etaassuperpartition}, we have $\prod_{s\in \eta}d_{\eta}(s)=0$ iff
\begin{equation*}\prod_{s \in \mathrm{BF}(\Lambda)} (\alpha(a_{\Lambda^{*}}(s)+1)+l_{\Lambda^{\circledast}}(s)+1)=0 \quad \text{or} \quad \prod_{s \in \Lambda^{*}/\mathrm{BF}(\Lambda)}(\alpha(a_{\Lambda^{*}}(s)+1)+l_{\Lambda^{*}}(s)+1)=0.\end{equation*}
Now, this is possible iff there exists a cell $s \in \mathrm{BF}(\Lambda)$ such that $\alpha(a_{\Lambda^{*}}(s)+1)+l_{\Lambda^{\circledast}}(s)+1=0$ or if there exists a cell $s \in \Lambda^{*}/\mathrm{BF}(\Lambda)$ such that $\alpha(a_{\Lambda^{*}}(s)+1)+l_{\Lambda^{*}}(s)+1=0$.

First, we suppose  that $s=(i,j) \in \mathrm{BF}(\Lambda)$. Now $\alpha(a_{\Lambda^{*}}(s)+1)+l_{\Lambda^{\circledast}}(s)+1=0$ iff there exists a $\rho \in \mathbb{Z_{+}}$ such that $a_{\Lambda^{*}}(s)+1=\rho(r-1)$ and $l_{\Lambda^{\circledast}}(s)+1=\rho(k+1)$. Using both relations and expressing them in terms of the components of $\Lambda$, we get
\begin{equation*} \Lambda^{*}_{i}-\Lambda^{\circledast}_{i+\rho(k+1)-1}+1=\rho(r-1). \end{equation*}
Moreover, we have by hypothesis, $\Lambda^{*}_{i}=\Lambda^{\circledast}_{i}$ (bosonic row), so that the previous line can be rewritten as
\begin{equation*} \rho(r-1)-1=\Lambda^{\circledast}_{i}-\Lambda^{\circledast}_{i+\rho(k+1)-1}.\end{equation*}
However, by using Lemma \ref{leminequality2}, we also get
\begin{equation*}\Lambda^{\circledast}_{i}-\Lambda^{\circledast}_{i+\rho(k+1)-1}\geq \rho r-1,\end{equation*}
which contradicts  the previous equality.

Second, we suppose that there is a cell $s=(i,j) \in \Lambda^{*}/\mathrm{BF}(\Lambda)$ such that $\alpha(a_{\Lambda^{*}}(s)+1)+l_{\Lambda^{*}}(s)+1=0$. This is possible iff there exists a $\rho \in \mathbb{Z_{+}}$ such that $a_{\Lambda^{*}}(s)+1=\rho(r-1)$ and $l_{\Lambda^{*}}(s)+1=\rho(k+1)$. As in the previous case, using both relations and expressing them in terms of the components of $\Lambda$, we obtain 
\begin{equation*} \rho(r-1)-1\geq \Lambda^{*}_{i}-\Lambda^{*}_{i+\rho(k+1)-1}\geq \rho r-1  \end{equation*}
which is in contradiction with the admissibility condition of $\Lambda$ (see Lemma \ref{leminequality2}).

Therefore, whenever $\alpha=\alpha_{k,r}$ and $\La$ is $(k,r,N)$-admissible, we have  $\prod_{s\in \eta}d_{\eta}(s)\neq 0 $, as expected.  
\end{proof}

\subsection{Regularity for Jack polynomials with prescribed symmetry}

We recall that $\lambda^{+}=(\lambda^{+}_1, \ldots,\lambda^{+}_m)$ and $\mu^{+}=(\mu^{+}_1,\ldots,\mu^+_{N-m})$  are partitions.  Similarly, $\lambda^{-}=(\lambda^{+}_m, \ldots,\lambda^{+}_1)$ and $\mu^{-}=(\mu^{+}_{N-m},\ldots,\mu^+_{1})$ denote compositions whose elements are written in increasing order.  The concatenation of $\lambda^{-}$ and $\mu^{-}$ is given by
$$   (\lambda^{+},\mu^{+})=(\lambda^{+}_m, \ldots,\lambda^{+}_1, \mu^{+}_{N-m},\ldots,\mu^+_{1}).$$

As shown below, the regularity for Jack polynomials with prescribed symmetry cannot be established directly from Definition \ref{defJackPrescibed}.  Indeed, a non-symmetric Jack polynomials indexed by a composition $\eta$ of the form $(\lambda^{-},\mu^{-})$  is in general singular at $\alpha=\alpha_{k,r}$, even if $\eta$ is associated to an admissible superpartition.  We  thus need to use another normalization for the Jack polynomials with prescribed symmetry. 

{
\begin{proposition}\label{normalizAS} Let $\eta=(\lambda^{+},\mu^{+})$ and $\Lambda=\varphi_{m}(\eta)$. Suppose that $\alpha$ is generic.  Then 
\begin{align*}
P_{\La}^{\AS}(x;\alpha)&=\frac{c^\AS_\La}{C^{\AS}_{\Lambda}}\mathrm{Asym}_{I}\mathrm{Sym}_{J}E_{\eta}\,,& 
P_{\La}^{\SS}(x;\alpha)&=\frac{c^\SS_\La}{C^{\SS}_{\Lambda}}\mathrm{Sym}_{I}\mathrm{Sym}_{J}E_{\eta}\,,\\
P_{\La}^{\SA}(x;\alpha)&=\frac{c^\SA_\La}{C^{\SA}_{\Lambda}}\mathrm{Sym}_{I}\mathrm{Asym}_{J}E_{\eta}\,,&
 P_{\La}^{\AA}(x;\alpha)&=\frac{c^\AA_\La}{C^{\AA}_{\Lambda}}\mathrm{Asym}_{I}\mathrm{Asym}_{J}E_{\eta}\,,
\end{align*}
where 
{\small\begin{align*} C^{\AS}_{\Lambda}&=(-1)^{m(m-1)/2}\prod_{s \in \mathrm{FF}^{*}(\Lambda)}\frac{\alpha a_{\Lambda^{\circledast}}(s)+l_{\Lambda^{\circledast}}(s)-1}{\alpha a_{\Lambda^{\circledast}}(s)+l_{\Lambda^{\circledast}}(s)}\prod_{\stackrel{s=(i,j) \in \mathrm{BRD} B}{0\leq \gamma\leq \#\{t>i|\Lambda^{\circledast}_{t}-\Lambda^{*}_{t}=0,\;\Lambda^{*}_{t}=i\}-1}}\frac{\alpha a_{\Lambda^{*}}(s)+l_{\Lambda^{*}}(s)-\gamma+1}{\alpha a_{\Lambda^{*}}(s)+l_{\Lambda^{*}}(s)-\gamma},\\
C^{\SS}_{\Lambda}&=\prod_{\stackrel{s=(i,j) \in \mathrm{FF}^{*}(\Lambda)}{0\leq \gamma\leq \#\{t>i|\Lambda^{\circledast}_{t}-\Lambda^{*}_{t}=1,\; \Lambda^{\circledast}_{t}=i\}-1}}\!\!\!\! \frac{\alpha a_{\Lambda^{\circledast}}(s)+l_{\Lambda^{\circledast}}(s)-\gamma+1}{\alpha a_{\Lambda^{\circledast}}(s)+l_{\Lambda^{\circledast}}(s)-\gamma}\!\!\!\! \prod_{\stackrel{s=(i,j) \;\mathrm{BRD} B}{0\leq \gamma'\leq \#\{t>i|\Lambda^{\circledast}_{t}-\Lambda^{*}_{t}=0,\;\Lambda^{*}_{t}=i\}-1}} \!\!\!\!\frac{\alpha a_{\Lambda^{*}}(s)+l_{\Lambda^{*}}(s)-\gamma'+1}{\alpha a_{\Lambda^{*}}(s)+l_{\Lambda^{*}}(s)-\gamma'},\\
C^{\SA}_{\Lambda}&=(-1)^{(N-m)(N-m-1)/2}\!\!\!\!\prod_{\stackrel{s=(i,j) \in \mathrm{FF}^{*}(\Lambda)}{0\leq \gamma\leq \#\{t>i|\Lambda^{\circledast}_{t}-\Lambda^{*}_{t}=1,\;\Lambda^{\circledast}_{t}=i\}-1}} \!\!\!\! \frac{\alpha a_{\Lambda^{\circledast}}(s)+l_{\Lambda^{\circledast}}(s)-\gamma+1}{\alpha a_{\Lambda^{\circledast}}(s)+l_{\Lambda^{\circledast}}(s)-\gamma}\prod_{s\in\mathrm{BRD} B}\frac{\alpha a_{\Lambda^{*}}(s)+l_{\Lambda^{*}}(s)-1}{\alpha a_{\Lambda^{*}}(s)+l_{\Lambda^{*}}(s)},\\
C^{\AA}_{\Lambda}&=(-1)^{m(m-1)/2}(-1)^{(N-m)(N-m-1)/2}\prod_{s \in \mathrm{FF}^{*}(\Lambda)}\frac{\alpha a_{\Lambda^{\circledast}}(s)+l_{\Lambda^{\circledast}}(s)-1}{\alpha a_{\Lambda^{\circledast}}(s)+l_{\Lambda^{\circledast}}(s)}\prod_{s\in\mathrm{BRD} B}\frac{\alpha a_{\Lambda^{*}}(s)+l_{\Lambda^{*}}(s)-1}{\alpha a_{\Lambda^{*}}(s)+l_{\Lambda^{*}}(s)}\, .
\end{align*}}Note that $\mathrm{FF}(\Lambda)$ denotes the set of cells belonging to a fermionic row and a fermionic column, while  $\mathrm{FF}^{*}(\Lambda)=\mathrm{FF}(\Lambda)\setminus\{s|s \in \Lambda^{\circledast}/\Lambda^{*}\}$. The set $\mathrm{BRD}{B}$   contains all cells $(i,j)$ such that $i$ is a bosonic row, $j$ is the length of some other bosonic row $ i'$ satisfying $\La_i^*>\La^*_{i'}$.  
\end{proposition}
\begin{proof}[Sketch of proof]Let  $\eta^{-}=(\lambda^{-},\mu^{-})$.  The proof consists in calculating the constant of proportionality $C_\La$ such that
$$ \mathcal{O}_{I,J}E_\eta=C_\La \mathcal{O}_{I,J}E_{\eta^-}.$$    
Our method follows  general arguments that are independent of the symmetry type of the polynomials,  so we give the general idea of the proof only for the polynomials of type AS. 

We first note that we can recover $\eta$ from $\eta^{-}$ through the following sequence of transpositions:
\begin{equation*}\eta=\tau_{2} \ldots \tau_{m-1}\tau_{m}\omega_{m+2} \ldots \omega_{N}(\eta^{-})\end{equation*}
where $\tau_{r}= K_{r-1}K_{r-2}\ldots K_{1}$ and $\omega_{r}=K_{r-1}K_{r-2}\ldots K_{m+1}$, except that in  $\omega_{r}$,  we do not consider  transpositions $K_{i}$ such that  $\mu_{i}=\mu_{i+1}$. Thus, we have 
\begin{equation*} E_{\eta}=E_{\tau_{2} \ldots \tau_{m-1}\tau_{m}\omega_{m+2} \ldots \omega_{N}(\eta^{-})} \end{equation*}
Now, given that we are considering $\eta^{-}$ a composition in increasing order, we can use successively the third line of \eqref{permjnosym}.  This yields an expression of the form  $$E_{\tau_{2} \ldots \tau_{m-1}\tau_{m}\omega_{m+2} \ldots \omega_{N}(\eta^{-})}=\mathcal{O}'_I\mathcal{O}'_J \omega_{N} E_{\eta^{-}},$$ where the operators  $\mathcal{O}'_I$ and $\mathcal{O}'_J$ are such that   $$  \mathrm{Asym}_{I}\mathcal{O}'_I=C'_I,\quad \mathrm{Sym}_{J}\mathcal{O}'_I=\mathcal{O}'_I\mathrm{Sym}_{J}, \quad \mathrm{Sym}_{J}\mathcal{O}'_J=C'_J,\quad \mathrm{Asym}_{I}\mathcal{O}'_J=\mathcal{O}'_J\mathrm{Asym}_{J}.$$ The coefficients $C'_I$ and $C'_J$ are obtained by considering all possible combinations of differences of eigenvalues $\overline{\Lambda}_{i}-\overline{\Lambda}_{j}$ with $i<j$, $i,j \in \{1,\ldots,m\}$ and $\Lambda_{i} \neq \Lambda_{j}$ or $i,j \in \{m+1,\ldots,N\}$ and $\Lambda_{i} \neq \Lambda_{j}$. More specifically, 
\begin{equation*}C'_I=(-1)^{m(m-1)/2}\prod_{\stackrel{i<j, \Lambda_{i} \neq \Lambda_{j}}{i,j \in \{1,\ldots,m\}}}\left(1-\frac{1}{\overline{\Lambda}_{i}-\overline{\Lambda}_{j}}\right)\end{equation*}
while  
\begin{equation*}C'_J=\prod_{\stackrel{i<j,\Lambda_{i} \neq \Lambda_{j}}{i,j \in \{m+1,\ldots,N\}}}\left(1+\frac{1}{\overline{\Lambda}_{i}-\overline{\Lambda}_{j}}\right)\end{equation*}
Rewriting the product $C'_I\cdot C'_J$ in a more compact form finally gives the desired expression for $C^{\AS}_{\Lambda}$.
\end{proof}

\begin{lem} Let $\eta=(\lambda^{+},\mu^{+})$ and $\Lambda=\varphi_m(\eta)$. \label{CLambda}

\noindent (i) If $\Lambda$ is strict and weakly $(k,r,N)$-admissible, then $C^{\AS}_{\Lambda}$ has neither  zeros nor singularities at   $\alpha=\alpha_{k,r}$.  

\noindent (ii) If $\Lambda$ is moderately $(k,r,N)$-admissible, then $C^{\SS}_{\Lambda}$ has neither  zeros nor singularities at  $\alpha=\alpha_{k,r}$.

\noindent (iii) If $\Lambda$ is moderately $(k,r,N)$-admissible, then $C^{\SA}_{\Lambda}$ has neither  zeros nor singularities at  $\alpha=\alpha_{k,r}$.

\noindent (iv) If $\Lambda$ is strict and weakly $(k,r,N)$-admissible, then $C^{\AA}_{\Lambda}$ has neither  zeros nor singularities at $\alpha=\alpha_{k,r}$.

\end{lem}
\begin{proof}[Sketch of proof]
This follows almost immediately from the explicit formulas for the coefficient $C_\La$ given in the last proposition.  All cases are similar.  The only noticeable differences are the type of admissibility for each symmetry type and the additional parameter $\gamma$,  which can be controlled with admissibility condition.  Once again, we restrict our demonstration to symmetry  type AS.

Consider $C^{\AS}_{\Lambda}$ and suppose that it has singularities a pole $\alpha=\alpha_{k,r}$.  This happens iff there exists a cell $s \in \mathrm{FF}^{*}$ such that $\alpha a_{\Lambda^{\circledast}}(s)+l_{\Lambda^{\circledast}}(s)=0$ or a cell $s \in \mathrm{BRD} B$  such that $\alpha a_{\Lambda^{*}}(s)+l_{\Lambda^{*}}(s)-\gamma=0$ for some $0\leq \gamma\leq \#\{t>i|\Lambda^{*}_{t}=\Lambda^{*}_{i}\}$.

First, assume that $s=(i,j) \in \mathrm{FF}^{*}$.  Note that $\alpha a_{\Lambda^{\circledast}}(s)+l_{\Lambda^{\circledast}}(s)=0$ iff there exists a positive integer  $\rho$ such that $a_{\Lambda^{\circledast}}(s)=\rho(r-1)$ and $l_{\Lambda^{\circledast}}(s)=\rho(k+1)$.  Using these two relations and expressing them in terms of the components of $\Lambda$, we find 
\begin{equation}\label{regeq1}
\Lambda^{\circledast}_{i}-\Lambda^{\circledast}_{i+\rho(k+1)}=\rho(r-1).
\end{equation}
Now, the weak admissibility condition and Lemma \ref{leminequality1} imply that 
\begin{equation}\label{regeq2}
\rho(r-1)=\Lambda^{\circledast}_{i}-\Lambda^{\circledast}_{i+\rho(k+1)}\geq \rho r.
\end{equation}
Equations \eqref{regeq1} and \eqref{regeq2}  are contradictory.  Hence, the first factor of $C^{\AS}_{\Lambda}$ does not have singularities.

Now, assume  $s \in \mathrm{BRD}B$.   Following a similar  argument, we conclude that the second factor has no singularity. 

In the same way, one can show that $C_{\Lambda}$ has no zero.
\end{proof}


\begin{proposition}[Regularity] \label{regularityJackPresc}Let $\Lambda$ be a $(k,r,N)$-admissible superpartition. Then, $P_{\Lambda}(x_1,\ldots,x_N;\alpha)$ is regular at $\alpha=\alpha_{k,r}$.
\end{proposition}
\begin{proof} Let $\eta=(\lambda^{+},\mu^{+})$ and $\Lambda=\varphi_{m}(\eta)$.
According to Proposition \ref{normalizAS}, for any symmetry type, there are coefficients $c_\La$ and $C_\La$ such that 
$$ P_{\Lambda}(x;\alpha) = \frac{c_\La}{C_\La}\mathcal{O}_{I,J}E_\eta(x;\alpha) $$

The coefficient $c_\La$ is independent of $\alpha$, so it is trivially regular $\alpha=\alpha_{k,r}$.  Given that $\La$ is admissible,  Lemma \ref{CLambda} implies that $C_\La^{-1}$ is also regular at $\alpha=\alpha_{k,r}$.  Finally, by Lemma \ref{regjnonsym}, the non-symmetric Jack polynomial $E_\eta(x;\alpha)$ is  regular at $\alpha=\alpha_{k,r}$. Therefore,  limit 
$$ \lim_{\alpha\to\alpha_{k,r}}\frac{c_\La}{C_\La}\mathcal{O}_{I,J}E_\eta(x;\alpha)$$ is well defined and the proposition follows. 
\end{proof}

\subsection{Uniqueness for Jack polynomials with prescribed symmetry}

\begin{lem} \label{lemabase}
{Let $\Lambda$ be weakly $(k,r,N)$-admissible and strict.  Suppose that for some $\sigma\in S_N$, the superpartition $\Gamma$ satisfies $$\Gamma_i^*=\La^*_{\sigma(i)}+\frac{r-1}{k+1}(\sigma(i)-i),
$$   Then,
$$ \La_i^*<\Ga_i^*\,\Longrightarrow \,\sigma(i)<i \, ,  \qquad  \La_i^*=\Ga_i^*\,\Longrightarrow \,\sigma(i)=i\, , \qquad  \La_i^*>\Ga_i^*\,\Longrightarrow \,\sigma(i)>i.$$
Moreover,
$$ \sigma(i)=\begin{cases}i-k-1& \text{if}\quad  \La_i^*<\Ga_i^* \quad\text{and}\quad \La_{i-1}^*\geq\Ga_{i-1}^*\\ 
i+k+1& \text{if}\quad  \La_i^*>\Ga_i^* \quad\text{and}\quad \La_{i+1}^*\leq\Ga_{i+1}^*.\end{cases}$$} 
\end{lem}

\begin{proof}

Obviously, the equality   
$\Gamma_i^*=\La^*_{\sigma(i)}+\frac{r-1}{k+1}(\sigma(i)-i)$ holds 
only if there is $\rho \in \mathbb{Z}$ such that $\sigma(i)=i+\rho(k+1)$. 

First, we assume that $\La_i^*=\Ga_i^*$. Then, $\La_i^*=\La^*_{i\pm\rho(k+1)}\pm\rho(r-1)$ for some $\rho\geq0$.  Lemma \ref{leminequality1} implies however that $\La_i^*-\La^*_{i+\rho(k+1)}\geq \rho r$ and $\La^*_{i-\rho(k+1)}-\La_i^*\geq \rho r$.  Combining the last relations, we get $\rho(r-1)\geq \rho r$, which implies $\rho=0$.  Consequently, $\La_i^*=\Ga_i^*$ only if $\sigma(i)=i$.

Next, we assume that $\La_i^*>\Ga_i^*$.   We have three possible cases:  
\begin{enumerate}	
	\item $\sigma(i)=i$.  This implies  that $\La_i^*=\Ga_i^*$, which  contradicts our assumption.
	\item  $\sigma(i)=i-\rho(k+1)$ for some positive integer $\rho$.   We then have
	$\La_i^*>\La^*_{i-\rho(k+1)}-\rho(r-1)$.  
	However, according to Lemma \ref{leminequality2} ,we have $\La^\cd_{i-\rho(k+1)}-\La_i^*\geq \rho r$, so that $\La^*_{i-\rho(k+1)}-\La_i^*\geq \rho r-1$.  Combining these equations, we get $\rho(r-1)>\rho r-1$, which contradicts the fact that $\rho\geq 1$.
	\item $\sigma(i)=i+\rho(k+1)$ for some positive integer $\rho$.  In this case, we do not obtain a contradiction. Hence,   $\sigma(i)>i$.	
\end{enumerate}

Similar arguments can be used to prove that if $\La_i^*<\Ga_i^*$, then  $\sigma(i)<i$.

To prove the second part of proposition, we  suppose that $\La_i^*>\Ga_i^*$ while  $\La_{i+1}^*\leq\Ga_{i+1}^*$.  Now, we know that $\Ga_{i+1}^* \leq\Ga_{i}^*$, where $\Ga_{i+1}^*=\La_{i+1}^*+\delta$ for some $\delta\geq 0$, and  $\Ga_{i}^*=\La^*_{i+\rho(k+1)}+\rho(r-1)$ for some $\rho \in \mathbb{Z}_{+}$.  Combining these inequalities, we get  $\La_{i+1}^*+\delta\leq \La^*_{i+\rho(k+1)}+\rho(r-1)$.  However,     $\La_{i+1}^*=\La_{i+1}^\cd-\epsilon$ where $\epsilon=0, 1$. Thus,  $\La_{i+1}^\cd-\La^*_{i+\rho(k+1)}\leq \rho(r-1)-\delta+\epsilon$.  By making use of Lemma \ref{leminequality2}, we get $\rho r\leq \rho(r-1)-\delta+\epsilon$, which implies that $\epsilon=1$, $\delta=0$ and $\rho=1$.  Therefore,  $\La_i^*>\Ga_i^*$ and  $\La_{i+1}^*\leq\Ga_{i+1}^*$ imply  $\sigma(i)=i-k-1$. The case where $\La_i^*<\Ga_i^*$ and $\La_{i+1}^*\geq\Ga_{i+1}^*$ is proved analogously.  
\end{proof}


\begin{lem} \label{lemabase2}
Let $\Lambda$ be moderately or strongly $(k,r,N)$-admissible.  Suppose that for some $\omega\in S_N$, the superpartition $\Gamma$ satisfies $$\Gamma_i^\cd=\La^\cd_{\omega(i)}+\frac{r-1}{k+1}(\omega(i)-i),
$$   Then,
$$ \La_i^\cd<\Ga_i^\cd\,\Longrightarrow \,\omega(i)<i \, ,  \qquad  \La_i^\cd=\Ga_i^\cd\,\Longrightarrow \,\omega(i)=i\, , \qquad  \La_i^\cd>\Ga_i^\cd\,\Longrightarrow \,\omega(i)>i.$$
Moreover,
$$ \omega(i)=\begin{cases}i-k-1& \text{if}\quad  \La_i^\cd<\Ga_i^\cd \quad\text{and}\quad \La_{i-1}^*\geq\Ga_{i-1}^* \\ 
i+k+1& \text{if}\quad  \La_i^\cd>\Ga_i^\cd \quad\text{and}\quad \La_{i+1}^*\leq\Ga_{i+1}^*.\end{cases}$$ 
\end{lem}
\begin{proof}  One essentially follows the same steps as in the  proof of Lemma \ref{lemabase}.
\end{proof}

\begin{lem}\label{sigmaomega}
 {Let $\Lambda$ be a $(k,r,N)$-admissible superpartition and let $\Gamma$ satisfy
$$\Gamma_i^*=\La^*_{\sigma(i)}+\frac{r-1}{k+1}(\sigma(i)-i),\qquad \Gamma_i^\cd=\La^\cd_{\omega(i)}+\frac{r-1}{k+1}(\omega(i)-i)$$ 
for some $\sigma,\omega \in S_{N}$. Then, $\sigma=\omega$.}
\end{lem}
\begin{proof}
{The cases for which  $\La$ is a strict and weakly $(k,r,N)$-admissible superpartition or for which $\La$ is strongly $(k,r,N)$-admissible superpartition are almost identical, so we only prove the first.}
We deduce from the hypothesis that $\sigma(i)\equiv i \mod (k+1)$ and $\omega(i)\equiv i  \mod (k+1)$, so that $\omega(i)=\sigma(i)+t(k+1)$ for some $t \in \mathbb{Z}$.

First, we suppose that $\sigma(i)<\omega(i)$, which implies that $\omega(i)=\sigma(i)+t(k+1)$ for some $t \in \mathbb{Z}_{+}$. Then,
$$\Gamma_i^\cd-\Gamma_i^*=\La^\cd_{\sigma(i)+t(k+1)}-\La^*_{\sigma(i)}+ t(r-1).$$
By Lemma \ref{leminequality1}, we know that $\Lambda^*$ is $(k+1,r,N)$-admissible, which means that $\La^*_{\sigma(i)}-\La^*_{\sigma(i)+t(k+1)}\geq tr$ and  $\La^*_{\sigma(i)}-\La^\cd_{\sigma(i)+t(k+1)}\geq tr-1$.  
Combining the inequalities previously obtained, we get $$0\leq \Gamma_i^\cd-\Gamma_i^*\leq 1-tr+t(r-1)=1-t.$$ This inequality is possible only if $t=1$. We have thus shown that
$$ (i)\quad \Ga_i^\cd=\Ga_{i}^* \qquad(ii)\quad \omega(i)=\sigma(i)+k+1 \qquad (iii)\quad \Lambda_{\sigma(i)}^*-\Lambda_{\omega(i)}^\cd=r-1 .$$
Note that if $\La_{\sigma(i)}^\cd=\La_{\sigma(i)}^*$, then  $\La_{\sigma(i)}^\cd-\La_{\sigma(i)+k+1}^\cd=r-1\geq r$, which is a contradiction. Similarly, one gets a contradiction by supposing $\La_{\omega(i)}^\cd=\La_{\omega(i)}^*$.  Thus, we also have
$$ (iv)\quad \La_{\sigma(i)}^\cd=\La_{\sigma(i)}^*+1  \qquad (v)\quad \La_{\omega(i)}^\cd=\La_{\omega(i)}^*+1.$$

Second, we suppose that $\sigma(i)>\omega(i)$, which implies that $\sigma(i)=\omega(i)+t(k+1)$ for some $t \in \mathbb{Z}_{+}$. Then 
$$\Gamma_i^\cd-\Gamma_i^*=\La^\cd_{\omega(i)}-\La^*_{\omega(i)+t(k+1)}- t(r-1).$$ 
By Lemma \ref{leminequality2} we know that $\La^\cd_{\omega(i)}-\La^*_{\omega(i)+t(k+1)}\geq tr$, so that
$$1\geq \Gamma_i^\cd-\Gamma_i^*\geq tr-t(r-1)=t.$$
The latter inequality holds only if $t=1$. We have thus proved that
$$ (vi)\quad \Ga_i^\cd=\Ga_{i}^*+1 \qquad(vii)\quad \sigma(i)=\omega(i)+k+1 \qquad (viii)\quad \Lambda_{\omega(i)}^\cd-\Lambda_{\sigma(i)}^*=r . $$
Moreover, we deduce from (vi) and the admissibility condition, that
$$ (ix)\quad \La_{\omega(i)}^\cd=\La_{\omega(i)}^*\qquad (x) \quad\La_{\sigma(i)}^\cd=\La_{\sigma(i)}^*.$$

Now, assume that $\sigma$ and $\omega$ do not coincide.  Then, there exists a positive integer  $i$ such that $\omega(i)>\sigma(i)$, which by virtue of the above discussion, implies that $\omega(i)=\sigma(i)+k+1$.  Let $j$ be such that $\omega(i)=\sigma(i)+k+1=\sigma(j)$.  Obviously, $i\neq j$ and  $\sigma(j)\neq \omega(j)$.  Then,  according to conclusions (ii) and (vii) above, only cases can occur: $\omega(j)=\sigma(i)+k+1 \pm(k+1)$. \begin{itemize}
\item  Suppose that $\omega(j)=\sigma(i)+2(k+1)$ and let $j_2$ be such that $\sigma(j_2)=\sigma(i)+2(k+1)$, so that $j_2\neq j$.  Then, conclusions (ii) and (vii) above imply that  $\omega(j_2)=\sigma(i)+2(k+1)\pm(k+1)$. However, only the case $\omega(j_2)=\sigma(i)+3(k+1)$ is possible, since the equality $\omega(j_2)=\sigma(i)+k+1$  implies the contradiction $j_2=i$. 
Similarly, if  $j_3$ is  such that $\sigma(j_3)=\sigma(i)+3(k+1)$, then $\omega(j_3)=\sigma(i)+4(k+1)$.  Continuing in this way, one eventually finds a positive integer $\ell<N$ such that $\omega(\ell)>N$, which  clearly contradicts the fact that $\omega$ is a permutation of $\{1,\ldots,N\}$.  
\item Suppose that $\omega(j)=\sigma(i)$.  Recall that by definition, $\sigma(j)=\sigma(i)+k+1$.  Hence, $\omega(j)=\sigma(j)-k-1<\sigma(j)$. Conclusion  (viii) above then implies that $\La_{\omega(j)}^\cd-\La_{\sigma(j)}^*=r$, which is equivalent to   $\La_{\sigma(i)}^\cd-\La_{\omega(i)}^*=r$.  However, conclusion (iv) implies that  $\La_{\sigma(i)}^\cd=\La_{\sigma(i)}^*+1$.  Combination of the  last two equations finally leads to 
$$r-1=\La_{\sigma(i)}^*-\La_{\omega(i)}^*=\La_{\sigma(i)}^*-\La_{\sigma(i)+k+1}^*.$$ 
This equation contradicts   Lemma \ref{leminequality1}.
\end{itemize}
Therefore, the permutations $\sigma$ and $\omega$  must coincide, as expected.

\end{proof}

\begin{theorem}[Uniqueness at $\alpha=\alpha_{k,r}$] Let $\La$ be a $(k,r,N)$-admissible superpartition.    Assume moreover that $\alpha=\alpha_{k,r}$. Then, the Jack polynomial with prescribed symmetry, here denoted by $P_\Lambda$,  is the unique polynomial  satisfying: 
\begin{enumerate} \item $\displaystyle P_\Lambda=m_\Lambda + \sum_{\Gamma<\Lambda}c_{\La,\Gamma}m_\Ga,\qquad c_{\La,\Gamma}\in \mathbb{C}$,
\item $  S^*\,P_\La=\varepsilon_{\La^*}(\alpha,u)\,P_\La$ and $S^\cd\, P_\La=\varepsilon_{\La^\cd}(\alpha,u)\, P_\La$.
\end{enumerate}
\end{theorem}
\begin{proof}Proceeding as in Theorem \ref{theo1}, we know that there are more than one polynomials satisfying (1) and (2)  only if we can find a superpartition of type $\mathrm{T}$, say $\Ga$, such that   $\La>\Ga$,  $ \varepsilon_{\Ga^*}(\alpha,u)=\varepsilon_{\La^*}(\alpha,u) $, and $\varepsilon_{\Ga^\cd}(\alpha,u)=\varepsilon_{\La^\cd}(\alpha,u)$.  Consequently, in order to prove the uniqueness, it is sufficient to show that if $\Ga<\La$, then 
$ \varepsilon_{\Ga^*}(\alpha,u)\neq \varepsilon_{\La^*}(\alpha,u) $ or $\varepsilon_{\Ga^\cd}(\alpha,u)\neq \varepsilon_{\La^\cd}(\alpha,u)$.

 Let us  assume that we are given a superpartition $\Gamma<\La$ such that $ \varepsilon_{\Ga^*}(\alpha,u)=\varepsilon_{\La^*}(\alpha,u) $ and $\varepsilon_{\Ga^\cd}(\alpha,u)= \varepsilon_{\La^\cd}(\alpha,u)$. Obviously, the last two equality holds if and only if  there are $\sigma,\omega \in S_{N}$ such that \begin{equation} \label{eqsigmaomega} \Gamma_i^*=\La^*_{\sigma(i)}+\frac{r-1}{k+1}(\sigma(i)-i),\qquad \Gamma_i^\cd=\La^\cd_{\omega(i)}+\frac{r-1}{k+1}(\omega(i)-i)\qquad \forall \, i.\end{equation}


  According to Lemma \ref{sigmaomega}, Equation  \eqref{eqsigmaomega} holds only if $\sigma=\omega$.  Now, we recall that by hypothesis, either $\Ga^*<\La^*$ or $\Ga^*=\La^*$ and $\Ga^\cd<\La^\cd$.   Only the former case is nontrivial however.  Indeed, Lemma \ref{lemabase} implies that if $\La_i^*=\Ga_i^*$ for all $i$, then $\sigma$ is  the identity, and so is $\omega$.   In short, whenever Equation \eqref{eqsigmaomega}  and $\Ga^*=\La^*$ hold, we have $\Ga^\cd=\La^\cd$, which is in contradiction with $\Ga^\cd<\La^\cd$.  Thus, we must assume that  $\Ga^*<\La^*$, which implies that there exist  integers $j>1$ and $\epsilon>0$ such that \begin{equation}\label{defj}\Ga^*_j=\La^*_j+\epsilon\quad\text{and}\quad  \Ga^*_i\leq \La^*_i,\quad \forall\, i<j\,.\end{equation}
As a consequence of \eqref{eqsigmaomega} and Lemma \ref{sigmaomega}, there is a permutation  $\sigma$ such that $\sigma(j)\neq j$, 
\begin{equation}\label{sigmaj}\Gamma_j^*=\La^*_{\sigma(j)}+\frac{r-1}{k+1}(\sigma(j)-j),\qquad \Gamma_j^\cd=\La^\cd_{\sigma(j)}+\frac{r-1}{k+1}(\sigma(j)-j), \end{equation}
which is possible only if $\sigma(j)=j \mod(k+1)$.  
\begin{enumerate}
\item If $\sigma(j)=j+\rho(k+1)$ for some positive integer $\rho$, then  $\Ga_j^*=\La^*_j+\epsilon=\La^*_{j+\rho(k+1)}+\rho(r-1)$.  However, the latter equation contradicts  the hypothesis $\epsilon>0$ and Lemma \ref{leminequality2}, according to which $\La^*_j-\La^*_{j+\rho(k+1)}\geq \rho r-1$.   
\item If $\sigma(j)=j-\rho(k+1)$ for some positive integer $\rho$, then $\Ga_j^*=\La^*_{j-\rho(k+1)}-\rho(r-1)$.   Moreover,  we know that $\Ga^*_{j-1}=\La^*_{j-1}-\delta$, for some $\delta\geq0$, and that $\Ga^*_{j-1} \geq \Ga^*_j$.  Combining these equations, we get $\rho(r-1)\geq\delta+ \La^*_{j-\rho(k+1)}-\La^*_{j-1}$.  But by definition, $\La_{j-\rho(k+1)}^*=\La_{j-\rho(k+1)}^\cd-\tilde\epsilon $, where  $\tilde \epsilon=0,1$.   The use of Lemma \ref{leminequality2}  then leads to $\rho(r-1)\geq \delta+\rho r-\tilde\epsilon$.  Hence $\delta=0$, $\tilde\epsilon=1$, and $\rho=1$. In short, we have shown that
$$ \qquad \Ga^*_j=\La^*_{j-k-1}-r+1, \quad \Ga^\cd_j=\La^\cd_{j-k-1}-r+1, \quad  \Ga^*_{j-1}=\La_{j-1}^*,\quad \Ga^\cd_{j-1}=\La_{j-1}^\cd,\quad \La^\cd_{j-k-1}=\La^*_{j-k-1}+1.$$
{Now, if $\La$ is strict and weakly $(k,r;N)$-admissible, then $\Ga^\cd_j=\Ga^*_j+1$ implies $\Ga^*_{j-1}=\La_{j-1}^*\geq \Ga^\cd_j$. Combining the previous equations, we get $\La_{j-1}^*\geq \La^\cd_{j-k-1}-r+1$, which contradicts the weak admissibility condition. 

On the other hand, if $\La$ is strongly $(k,r;N)$-admissible, then $\Ga^\cd_{j-1}=\La_{j-1}^\cd\geq \Ga^\cd_j$ implies $\La^\cd_{j-1}\geq \La^\cd_{j-k-1}-r+1$, which contradicts the strong admissibility condition. 
}
\end{enumerate}
Therefore, whenever $\La$ is $(k,r,N)$-admissible, we cannot find a superpartition $\Ga<\La$ such that  $ \varepsilon_{\Ga^*}(\alpha,u)=\varepsilon_{\La^*}(\alpha,u) $ and $\varepsilon_{\Ga^\cd}(\alpha,u)=\varepsilon_{\La^\cd}(\alpha,u)$.

\end{proof}

\subsection{Uniqueness for non-symmetric Jack polynomials}

\begin{definition}\label{defadmcompo} Let $\lambda=(\lambda_1,\ldots,\lambda_N)$ be a composition and let $\Lambda=\varphi_m(\la)$ be its associated superpartition.  
We say that $\lambda$ is  weakly, moderately, or strongly $(k,r,N|m)$-admissible if and only if $\La$ is respectively  weakly, moderately,  or strongly $(k,r,N)$-admissible.  
\end{definition}

\begin{theorem}[Uniqueness for $k=1$: weak admissibility]\label{jacknonsymk1weak} Let $\lambda=(\eta_1,\ldots,\eta_m,\mu_1,\ldots,\mu_{N-m})$ be a composition formed by the concatenation of the partitions $\eta=(\eta_1,\ldots,\eta_m)$ and $\mu=(\mu_1,\ldots,\mu_{N-m})$.  Assume that $\lambda$ is weakly $(1,r,N|m)$-admissible and $\eta$ is strictly decreasing.   Assume moreover that $\alpha=\alpha_{1,r}$.   Then, the non-symmetric Jack polynomial $E_\lambda$ is the unique polynomial satisfying: 
\begin{enumerate} \item $\displaystyle E_\lambda=x^\lambda + \sum_{\gamma \prec\lambda}c_{\la,\gamma}x^\gamma,\qquad c_{\la,\gamma}\in \mathbb{C}$,
\item $\displaystyle \xi_i\, E_\la=\bar\la_i \, E_\la \quad \forall \, 1\leq i\leq N$,
\end{enumerate}
where the $\bar \la_i$'s denote the eigenvalues introduced in (A2') and \eqref{eqeigennonsym}.

\end{theorem}

\begin{proof}There are more than one polynomials satisfying (1) and (2)  only if there are compositions $\ga$ such that $\ga\prec\la$ and
$(\bar \ga_1,\ldots,\bar \ga_N)=(\bar \la_1,\ldots,\bar \la_N)$.  We can thus establish the uniqueness by showing show that the latter equality is impossible.  Our task will be simplified by working with the associated superpartitions $$\La=\varphi_m(\lambda),\qquad \Gamma=\varphi_m(\gamma).$$  We indeed know that $\Ga<\La$ whenever  $\ga\prec\la$.  Moreover,   according to Lemma \ref{lemsekiguchi}, the equality $(\bar \ga_1,\ldots,\bar \ga_N)=(\bar \la_1,\ldots,\bar \la_N)$ holds only if $ \varepsilon_{\Ga^*}(\alpha,u)=\varepsilon_{\La^*}(\alpha,u) $, and $\varepsilon_{\Ga^\cd}(\alpha,u)=\varepsilon_{\La^\cd}(\alpha,u)$.

 Let us now assume that we are given a superpartition $\Gamma$ such that $ \varepsilon_{\Ga^*}(\alpha,u)=\varepsilon_{\La^*}(\alpha,u) $ and $\varepsilon_{\Ga^\cd}(\alpha,u)= \varepsilon_{\La^\cd}(\alpha,u)$. The last two equalities hold  if and only if there are permutations $\sigma$ and $\omega$ such that \begin{equation} \label{eqsigmaomega2} \Gamma_i^*=\La^*_{\sigma(i)}+\frac{r-1}{2}(\sigma(i)-i),\qquad \Gamma_i^\cd=\La^\cd_{\omega(i)}+\frac{r-1}{2}(\omega(i)-i)\qquad \forall \, i.\end{equation} We recall that by hypothesis, $\La$ is strict and $(1,r,N)$-admissible and $\Ga<\La$, which means that either $\Ga^*<\La^*$ or $\Ga^*=\La^*$ and $\Ga^\cd< \La^\cd$. 

The simplest case is  when  $\Ga^*=\La^*$ and $\Ga^\cd< \La^\cd$.  Indeed, $\Ga^*_i=\La^*_i$ for all $i$ implies $\sigma=\mathrm{id}$, while Lemma \ref{sigmaomega} yields $\sigma=\omega$, so that  $\omega=id$ and  $\Ga^\cd=\La^\cd$.  This contradicts the assumption $\Lambda \neq \Gamma$.  Thus, the equations $\Ga^*=\La^*$, $\Ga^\cd< \La^\cd$, $ \varepsilon_{\Ga^*}(\alpha,u)=\varepsilon_{\La^*}(\alpha,u) $, and $\varepsilon_{\Ga^\cd}(\alpha,u)=\varepsilon_{\La^\cd}(\alpha,u)$ cannot be satisfied simultaneously if $\La$ is strict and $(1,r,N)$-admissible.

We now assume that 
 $\Ga^*<\La^*$.  This condition  implies that there exists  an integer  $j>1$  such that $$\Ga^*_j>\La^*_j \quad\text{and}\quad  \Ga^*_{i}\leq \La^*_i,\quad \forall\, i<j\,.$$  
According to Lemma \ref{lemabase}, satisfying  the first equality in \eqref{eqsigmaomega2} is possible  only if $\sigma(j)=j-2$.  Thus
$$ \Ga_j^*=\La_{j-2}^*-r+1 $$
Now, $\La_{j-2}^\cd=\La_{j-2}^*+\epsilon$ for some $0\leq \epsilon\leq 1$, and $\Ga_{j}^*=\La_{j-1}^*-\delta$ for some $\delta\geq0$. Hence,
$$ \epsilon+r=\La_{j-2}^\cd-\La_{j-1}^*+\delta+1,$$
which is compatible with the admissibility only if $\epsilon=1$ and $\delta=0$. Combining all the previous results, we get  
$$ (i)\quad \Ga_j^*=\La_{j-2}^*-r+1 \qquad(ii)\quad \Ga^*_{j}=\Ga^*_{j-1}=\La^*_{j-1}\qquad (iii)\quad \La_{j-2}^\cd=\La_{j-2}^*+1 \qquad 
$$
By making use of Lemma \ref{lemabase} {together with} $\La^*_{j-2}\geq \Ga^*_{j-2}$ and (ii), we also conclude that either $\sigma(j-2)=j-2$ or $\sigma(j-2)=j$.  The first case is obviously impossible since it contradicts $\sigma(j)=j-2$.   The second case implies  $\Ga^*_{j-2}=\La_j^*+r-1$.
 Lemma \ref{sigmaomega} and (i) imply that $\Ga_j^\cd=\La_{j-2}^\cd-r+1$.  Then, combining this equation with (iii), we get $$ (iv)\qquad \Ga_j^\cd=\Ga_j^*+1.$$  Moreover, Lemma \ref{sigmaomega} and (ii)  imply that  $\Ga^\cd_{j-1}=\La^\cd_{j-1}$.   From this and result (iv), we get $\Ga^\cd_{j-1}=\Ga^*_{j-1}+1$, i.e. the row $j-1$ in $\Ga$ also contains a circle. Combining $\Ga^*_{j-2}\geq\Ga^*_{j-1}$, the admissibility condition and $\Ga^*_{j-2}=\La_j^*+r-1$ we obtain $\Ga^*_{j-2}=\Ga^*_{j-1}$.   Finally, Lemma \ref{sigmaomega} and the last equation yields $\Ga^\cd_{j-2}=\Ga^*_{j-2}+1$. Consequently, 
 $$(v)\quad\Ga_{j-2}^*=\Ga_{j-1}^*=\Ga_{j}^* \qquad (vi)\quad \La_{j-2}^*=\La_{j-1}^*+r-1=\La_{j}^*+2(r-1).$$

Let us recapitulate what we have obtained so far.  We have shown that there exist compositions $\lambda$ and $\gamma$ as in the statement of the theorem such that their associated superpartitions   $\La=\varphi_m(\lambda)$ and $\Gamma=\varphi_m(\gamma)$ satisfy   $ \varepsilon_{\Ga^*}(\alpha,u)=\varepsilon_{\La^*}(\alpha,u) $ and $\varepsilon_{\Ga^\cd}(\alpha,u)=\varepsilon_{\La^\cd}(\alpha,u)$.  However, this occurs only if the equations (i) to (vi) are also satisfied.  We will now make use of this information to prove that   the equality $(\bar \ga_1,\ldots,\bar \ga_N)=(\bar \la_1,\ldots,\bar \la_N)$ is incompatible with the admissibility of $\lambda$.      

Before doing so, we  need to recall how relate the eigenvalues $\bar \la_i$ and $\bar \ga_i$ to the elements of the superpartitions $\La$ and $\Ga$.  Let $w_\gamma$ be the smallest permutation such that $\gamma=w_\gamma(\gamma^+)=w_\gamma(\Gamma^*)$.  Then, 
$\bar\gamma_i$ is equal to the $i$th element of the composition $\left(\alpha\gamma-w_\gamma\delta^-\right)$.  More explicitly, 
 $\bar \gamma_i=\left(w_\gamma(\alpha \Gamma^*-\delta^-)\right)_i$ or equivalently, $\bar \gamma_{w_\gamma (i)}=\alpha \Gamma^*_i-(i-1).$  Similarly, there is a minimal permutation $w_\la$ such that $\la=w_\la(\La^*)$, so that  $\bar \la_{w_\la (i)}=\alpha \La^*_i-(i-1)$.  We stress that in our case $\La^*\neq\Ga^*$, which implies that $w_\la\neq w_\ga$.

Now, let $j$ be the largest integer such that $\Gamma_{j}^*>\Lambda_{j}^*$ and $\Gamma_{j-1}^*\leq\Lambda_{j-1}^*$.  Let also $l=w_{\gamma}(j)$.  Then, according to the above discussion,    $$ \bar \gamma_l=\alpha \Ga^*_j-(j+1).$$ From (i) and (vi) above, we deduce that the last equation can be rewritten as \begin{equation}\label{eqbargamma}\bar \gamma_l =\alpha (\La_{j}^*+r-1)-(j-1) .\end{equation} Moreover, let $j'$ be defined as $w_\lambda^{-1}(l)$.  This implies that 
\begin{equation}\label{eqbarlambda} \bar \la_l =\alpha \La_{j'}^*-(j'-1) .\end{equation}
    Combining Equations \eqref{eqbargamma} and \eqref{eqbarlambda}, we get 
		\begin{equation}\label{eqvalbar} \bar\la_{l}-\bar\gamma_{l}=\alpha(\La_{j'}^*-\La_{j}^*-r+1)+j-j'.\end{equation}
 We are going to use the last equation and prove $ \bar\la_{l}-\bar\gamma_{l}\neq 0$.  Three cases must be analyzed separately:  \\
\begin{itemize}
	\item[(1)]  $\lambda_l=\Lambda_{j}^*$.  Then,  
$
\bar\la_{l}-\bar\gamma_{l} =-\alpha(r-1)
$, 
which is clearly different from $0$.

  \item[(2)]   $\lambda_{l}<\Lambda_{j}^*$.   Then,  $\Lambda_{j'}^*<\Lambda_{j}^*$ and  $j'>j$.  By the admissibility condition, we have $\Lambda_{j}^*-\Lambda_{j'}^*\geq \rho(r-1)$, where  $\rho=j'-j$. Thereby, $\Lambda_{j}^*-\Lambda_{j'}^*= \rho(r-1)+\delta$ for some $\delta\geq 0$. Now,
$$
\bar\la_{l}-\bar\gamma_{l} 
  =-\alpha((\rho+1)(r-1)+\delta)-\rho.$$
Substituting  $\alpha=\alpha_{1,r}=- {2}/({r-1})$ into the last equation, we wee that it is equal to zero  if and only if   $2((\rho+1)(r-1)+\delta)=\rho(r-1)$. This is impossible.

\item[(3)]  $\lambda_{l}>\Lambda_{j}^*$.  Then,  $\Lambda_{j'}^*>\Lambda_{j}^*$ and $j'<j$. Let $\rho=j-j'$. The admissibility condition implies that $\Lambda_{j'}^*-\Lambda_{j}^*\geq \rho(r-1)$.   Thereby, $\Lambda_{j'}^*-\Lambda_{j}^*= \rho(r-1)+\overline{\delta}$ for some $\overline{\delta}\geq 0$. Thus,
$$
\bar\la_{l}-\bar\gamma_{l} =\alpha((\rho-1)(r-1)+\overline{\delta})+\rho.$$
The last equation is zero when $\alpha=\alpha_{1,r}=-{2}/({r-1})$ if and only if $2((\rho-1)(r-1)+\overline{\delta})=\rho(r-1)$, which is equivalent to $(\rho-2)(r-1)+2\overline{\delta}=0$. It is clear that if $\rho>2$, we have  $\bar\la_{l}\neq \bar\gamma_{l}$.   Therefore we have only to analyze the cases for which  $\rho=1$ and $\rho=2$.

On the one hand, if $\rho=1$, then $j'=j-1$ and  $\Lambda_{j'}^*=\Lambda_{j-1}^*$.   Substituting the last equality and (vi) into \eqref{eqvalbar}, we get   
$
\bar\la_{l}-\bar\gamma_{l}  =1
$.

On other hand, if $\rho=2$, then $j'=j-2$ and $\Lambda_{j'}^*=\Lambda_{j-2}^*$. Using once again (vi) and \eqref{eqvalbar}, we find 
$$
\bar\la_{l}-\bar\gamma_{l}  = \alpha(r-1) +2. $$
Replacing $\alpha$ by $\alpha_{1,r}=-\frac{2}{r-1}$ into the last equation,  we get $\bar\la_{l}-\bar\gamma_{l}=0$.  Thus, we have not reached the desired conclusion yet.   However, given that in the present case, we have  $\la_{l-1}>\lambda_{l}=\Lambda_{j-2}^*$ and  $\ga_{l-1}=\gamma_{l}=\Gamma_{j}^*=\Gamma_{j-1}^*=\Gamma_{j-2}^*$, we know that $w_\la^{-1}(l-1)=\bar\j<j-2$, so that   $\Lambda_{\overline{\j}}^*>\Lambda_{j-2}^*$. Let $\overline{\rho}:=j-2-\overline{\j}$.  The admissibility condition then gives $\Lambda_{\overline{\j}}^*-\Lambda_{j-2}^*\geq \overline{\rho}(r-1)$, which is equivalent to $\Lambda_{\overline{\j}}^*=\Lambda_{j-2}^*+ \overline{\rho}(r-1)+\epsilon$ for some  $\epsilon\geq 0$. Then
\begin{equation*}
\begin{split}
\bar\la_{l-1}-\bar\gamma_{l-1}  &= \alpha(\Lambda_{j-2}^*+\overline{\rho}(r-1)+\epsilon)-(\overline{\j}-1)-\alpha \Gamma_{j-1}^* +(j-2)\\
 &=\alpha(\Lambda_{j-2}^*+\overline{\rho}(r-1)+\epsilon)-\alpha(\Lambda_{j-2}^*-(r-1))+\overline{\rho}+1\\
 &=\alpha((\overline{\rho}+1)(r-1)+\epsilon)+\overline{\rho}+1
\end{split}
\end{equation*}	
Finally, the substitution of  $\alpha=\alpha_{1,r}=- {2}/({r-1})$ into the last equation implies that  $ \bar\la_{l-1}=\bar\gamma_{l-1}$ iff $2(\overline{\rho}+1)(r-1)+2\epsilon=(\overline{\rho}+1)(r-1)$, which is impossible. 
\end{itemize}
  
 We have thus shown that there could exist compositions,  $\la$ and  $\gamma$,  such that $\La$ is $(1,r,N)$-admissible,  $\Ga^*<\La^*$ and $ \varepsilon_{\Ga^*}(\alpha,u)=\varepsilon_{\La^*}(\alpha,u) $.  However, when it happens, we also have $(\bar\la_1, \ldots, \bar\la_N)\neq(\bar\ga_1, \ldots, \bar\ga_N)$ and the theorem follows.  
\end{proof}

\begin{theorem}[Uniqueness for $k=1$: moderate admissibility] \label{jacknonsymk1moderate}Let $\lambda=(\eta_1,\ldots,\eta_m,\mu_1,\ldots,\mu_{N-m})$ be a composition formed by the concatenation of the partitions $\eta=(\eta_1,\ldots,\eta_m)$ and $\mu=(\mu_1,\ldots,\mu_{N-m})$.  Assume that $\lambda$ is moderately $(1,r,N|m)$-admissible.   Assume moreover that $\alpha=\alpha_{1,r}$.   Then, the non-symmetric Jack polynomial $E_\lambda$ is the unique polynomial satisfying: 
\begin{enumerate} \item $\displaystyle E_\lambda=x^\lambda + \sum_{\gamma\prec\lambda}c_{\la,\gamma}x^\gamma,\qquad c_{\la,\gamma}\in \mathbb{C}$,
\item $\displaystyle \xi_i\, E_\la=\bar\la_i \, E_\la \quad \forall \, 1\leq i\leq N$,
\end{enumerate}
where the $\bar \la_i$'s denote the eigenvalues introduced in (A2') and \eqref{eqeigennonsym}.

\end{theorem}
\begin{proof} We proceed as in Theorem \ref{jacknonsymk1weak}. We start by introducing the associated superpartitions $\La=\varphi_m(\lambda)$ and $\Gamma=\varphi_m(\gamma)$.  We then assume that we are given a superpartition $\Gamma$ such that $ \varepsilon_{\Ga^*}(\alpha,u)=\varepsilon_{\La^*}(\alpha,u) $ and $\varepsilon_{\Ga^\cd}(\alpha,u)= \varepsilon_{\La^\cd}(\alpha,u)$, which is possible if and only if Equation \eqref{eqsigmaomega2} is satisfied for some $\sigma,\omega\in S_N$. We recall that by hypothesis, $\La$ is moderately $(1,r,N)$-admissible and $\Ga<\La$, which means that either $\Ga^*<\La^*$ or $\Ga^*=\La^*$ and $\Ga^\cd< \La^\cd$.

First, we assume  that $\Ga^*=\La^*$ and $\Ga^\cd< \La^\cd$. This obviously implies that $\Ga^*_i=\La^*_i$ for all $i$, but also that there exists an integer $j>1$  such that $$\Ga^*_j=\La^*_j=\La^\cd_j,\quad \Ga^\cd_j=\La^\cd_j+1\quad\text{and}\quad  \Ga^\cd_i=\La^\cd_i-\delta_i,\quad \delta_i=0,1\quad\forall\, i<j\,.$$
By making use of Lemma \ref{lemabase2}, $\La^\cd_{j}< \Ga^\cd_{j}$ and  $\Ga^*_{j-1}=\La^*_{j-1}$, we conclude that $\omega(j)=j-2$. This implies $\Ga^\cd_j=\La^\cd_{j-2}-r+1$ and $\Ga^\cd_j=\La^\cd_j+1$, so we get $\La^\cd_{j-2}-\La^\cd_j=r$,  which is in contradiction with the admissibility.  

Second, we assume that  $\Ga^*<\La^*$, which implies that there exists a $j>1$  such that $$\Ga^*_j>\La^*_j \quad\text{and}\quad  \Ga^*_{i}\leq \La^*_i,\quad \forall\, i<j\,.$$  
According to Lemma \ref{lemabase}, the first equality in \eqref{eqsigmaomega2} is   possible when $i=j$  only if $\sigma(j)=j-2$.  Thus
$$ \Ga_j^*=\La_{j-2}^*-r+1 $$
Now $\La_{j-2}^\cd=\La_{j-2}^*+\epsilon$ for some $0\leq \epsilon\leq 1$, and $\Ga_{j}^*=\La_{j-1}^*-\delta$ for some $\delta\geq0$. Hence,
$$ \epsilon+r=\La_{j-2}^\cd-\La_{j-1}^*+\delta+1,$$
which is compatible with the admissibility only if $\epsilon=1$ and $\delta=0$. Combining all the previous results, we get  
$$ (i)\quad \Ga_j^*=\La_{j-2}^*-r+1 \qquad(ii)\quad \Ga^*_{j}=\Ga^*_{j-1}=\La^*_{j-1}\qquad (iii)\quad \La_{j-2}^\cd=\La_{j-2}^*+1 \qquad (iv)\quad \La_{j-1}^\cd=\La_{j-1}^* \,.$$

We now turn our attention to  second equality in \eqref{eqsigmaomega2} when $i=j$.   
By assumption we know that $\Ga^*_j>\La^*_j$,  so that $\Ga^\cd_j\geq\La^\cd_j$. By making use of Lemma  \ref{lemabase2},  we get the following two options:\\
\begin{enumerate}
	\item If $\Ga_j^\cd=\La^\cd_{j}$, then $\omega(j)=j$. However by assumption $\Ga_j^*=\La^*_j+\epsilon$ which implies $\Ga_j^\cd=\La^\cd_{j}=\La^*_j+1$ and $\Ga_j^*=\La^*_j+1$, and then $\Ga_j^\cd=\Ga_j^*$. Now, as $\Ga_j^*=\La^\cd_{j-2}-r$ we get $\La^\cd_{j-2}-r=\La_j^\cd$, which is clearly in contradiction with the admissibility.
	\item If $\Ga_j^\cd>\La^\cd_{j}$, then $\omega(j)<j$. Now, from $\Ga_j^\cd>\La^\cd_{j}$ and (ii), we know using Lemma \ref{lemabase2}, that $\omega(j)=j-2$, i.e. 
	$\Ga_j^\cd=\La_{j-2}^\cd-r+1$. Thus, the row $j$ in $\Gamma$ contains a circle.  This in turn implies that $\Ga_{j-1}^\cd=\Ga_{j-1}^*+1$, and also that the row $j-1$ in $\Gamma$ contains a circle. 
	
	So far, considering the row $j$, we have obtained 
	$$ (v)\quad \Ga_j^\cd=\Ga^*_{j}+1 \qquad(vi)\quad \Ga_{j-1}^\cd=\Ga_{j-1}^*+1. $$
	Now, considering (ii), (iv) and (vi), we obtain $\Ga_{j-1}^\cd>\La_{j-1}^\cd$.   Moreover, from  $\Ga^*_{j-2}\leq \La^*_{j-2}$ and Lemma \ref{lemabase2}, we get $\omega(j-1)=j-3$ and
		$\Ga_{j-1}^\cd=\La_{j-3}^\cd-r+1$.  However,  (ii), (iv), and (vi) imply that $\Ga_{j-1}^\cd=\La_{j-1}^\cd+1$.   Combining these equations, we conclude that  $\La_{j-3}^\cd-\La_{j-1}^\cd=r$.
		This violates our assumptions, because the moderate admissibility condition implies that $\La_{j-3}^\cd-\La_{j-1}^\cd\geq 2r$.
\end{enumerate}

We have shown that whenever $\La>\Ga$ and $\La$ is moderately (1,r,N)-admissible, then $(\bar\la_1, \ldots, \bar\la_N)\neq(\bar\ga_1, \ldots, \bar\ga_N)$, and the proof is complete.  
\end{proof}

\subsection{Clustering properties for $k=1$ } \label{SectionClusterk1}

We start by establishing $k=1$ clustering properties for the non-symmetric Jack polynomials.  We then use these results and prove similar properties for the Jack polynomials with prescribed symmetry.

{
\begin{definition} Let $\Lambda$ be a superpartition and let $\lambda$ be a partition. We formally define the superpartition $\Omega=\Lambda+\lambda$ where $\Omega=(\Omega^*,\Omega^{\circledast})$ as $\Omega^*=\Lambda^*+\lambda$ and $\Omega^{\circledast}=\Lambda^{\circledast}+\lambda$.  In terms of the diagrams, $\Omega$ is interpreted as the associated superpartition to the diagram obtained by adding diagram $\Lambda$ and $\lambda$.
\end{definition}
}

Let us illustrate the last definition by computing $\Lambda+\lambda$ when $\Lambda=(5,3,1,0; 4,2,1)$ and $\lambda=\delta_{7}=(6,5,4,3,2,1,0)$.  Obviously, we have 
$$\Lambda= {}{} 
\tableau[scY]{ &  &  & & &\bl \tcercle{} \\& & &\\& & & \bl \tcercle{} \\ & \\ & \bl \tcercle{}   \\\\\bl \tcercle{}} \qquad \Rightarrow \Lambda^{*}= {}{}\tableau[scY]{ &  &  & & &\bl  \\& & &\\& & & \bl  \\ & \\ & \bl    \\\\\bl } \qquad \text{and} \qquad  \Lambda^{\circledast}= {}{} \tableau[scY]{ &  &  & & & \\& & &\\& & &  \\ & \\ &  \\  \\ \\  }$$
Then, 
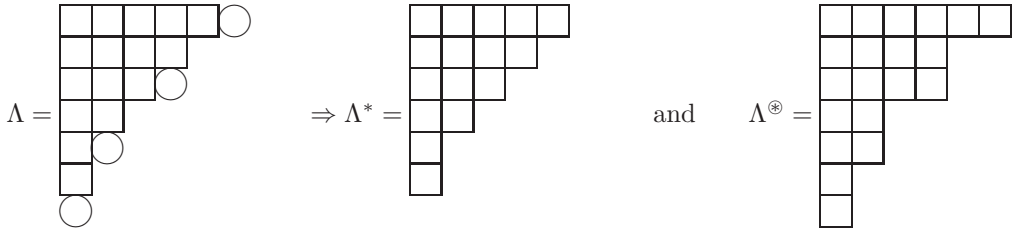
$$\Lambda^*+\lambda= {}{}  \tableau[scY]{ &  &  & &  &\tf &\tf &\tf &\tf &\tf &\tf\\& & & &\tf &\tf &\tf &\tf &\tf\\& &   &\tf &\tf &\tf &\tf \\ &  &\tf &\tf &\tf \\  &\tf &\tf\\  &\tf\\  } \qquad \text{and} \qquad  \Lambda^{\circledast}+\lambda= {}{} \tableau[scY]{ &  &  & & & &\tf &\tf &\tf &\tf &\tf &\tf\\& & & &\tf &\tf &\tf &\tf &\tf\\& & &   &\tf &\tf &\tf &\tf \\ &  &\tf &\tf &\tf \\ & &\tf &\tf\\  &\tf\\ \\  }$$
Thus, the diagram obtained by adding the diagrams associated to $\Lambda$ and $\lambda$ is given by
$$\Lambda+\delta= {}{} \tableau[scY]{ &  &  & & &\tf &\tf &\tf &\tf &\tf &\tf &\tf \bl \tcercle{}\\& & & &\tf &\tf &\tf &\tf &\tf\\& & & \tf   &\tf &\tf &\tf &\tf \bl \tcercle{} \\ &  &\tf &\tf &\tf \\ &  \tf &\tf &\tf\bl \tcercle{}\\  &\tf\\ \bl \tcercle{} }$$
which is equivalent to say that $ \Lambda+\delta=(11,7,3,0; 9,5,2)$.

{ 
\begin{proposition} Let $r$ be even and positive. Let also  $\kappa = (\lambda^{+},\mu^{+})$, where $\lambda^{+}$ is a partition with $m$ parts while $\mu^{+}$ is a strictly decreasing partition with $N-m$ parts. Then
\begin{equation*}
E_{\kappa+(r-1)\delta'}(x_{1},\ldots,x_{N};-2/(r-1)) \propto \prod_{1\leq i<j \leq N}(x_{i}-x_{j})^{r-1} E_{\kappa}(x_{1},\ldots,x_{N};2/(r-1)).
\end{equation*}
In the above equation, $\delta'=\omega_{\kappa}(\delta)$,  where $\delta=(N-1,N-2,\ldots,1,0)$ and    $\omega_{\kappa}$ is the smallest permutation such that $\kappa=\omega_{\kappa}(\kappa^+)$.
\label{agrupk1} 
\end{proposition}
\begin{proof} In what follows, we  set $\La=\varphi_m(\kappa)$ and use the  shorthand notation  $\Delta_{N}=\prod_{1\leq i<j \leq N}(x_{i}-x_{j})$.\\
First, we consider the action of $\xi_{j}$ on  the polynomial $\Delta_{N}^{(r-1)} E_{\kappa}(x;2/(r-1))$: 
\begin{multline*} 
\xi_{j} (\Delta_{N}^{(r-1)} E_{\kappa}(x;2/(r-1))) = \alpha \Delta_{N}^{(r-1)} \left((r-1) \sum_{i=1,i\neq j}^{N}\frac{x_{j}}{x_{j}-x_{i}} E_{\kappa}(x;2/(r-1))+  x_{j}\partial_{x_{j}} E_{\kappa}(x;2/(r-1)) \right)\\ 
+ \Delta_{N}^{(r-1)} \left(\sum_{i<j}\frac{x_{j}}{x_{j}-x_{i}}(1+ K_{ij})E_{\kappa}(x;2/(r-1))+ \sum_{i>j}\frac{x_{i}}{x_{j}-x_{i}}(1+K_{ij})E_{\kappa}(x;2/(r-1))\right) \\
-(j-1)\Delta_{N}^{(r-1)}E_{\kappa}(z;2/(r-1)).
\end{multline*}
Second, we restrict  $\xi_{j}$ by imposing ${\alpha=-2/(r-1)}$, which gives
\begin{multline*} 
\xi_{j}|_{\alpha=-2/(r-1)} (\Delta_{N}^{(r-1)} E_{\kappa}(x;2/(r-1))) = -\frac{2}{r-1} \Delta_{N}^{(r-1)}  x_{j}\partial_{x_{j}} E_{\kappa}(x;2/(r-1)) -(N-1)\Delta_{N}^{(r-1)} E_{\kappa}(z;2/(r-1)) \\
- \Delta_{N}^{(r-1)}\sum_{i=1,i\neq j}^{N}\frac{x_{j}}{x_{j}-x_{i}}(1- K_{ij})E_{\kappa}(z;2/(r-1))-\Delta_{N}^{(r-1)}\sum_{i>j}K_{ij}E_{\kappa}(x;2/(r-1)). 
\end{multline*}
By reordering the terms, we also get 
\begin{equation*} 
\xi_{j}|_{\alpha=-2/(r-1)} (\Delta_{N}^{(r-1)} E_{\kappa}(x;2/(r-1)))=-\Delta_{N}^{(r-1)}\left(\xi_{j}|_{\alpha=2/(r-1)}+2(N-1)\right)E_{\kappa}(x;2/(r-1)).
\end{equation*}
Now, the use of   (A2'), allows us to write 
\begin{equation} \label{eqcheredVanderJack}
\xi_{j}|_{\alpha=-2/(r-1)} (\Delta_{N}^{(r-1)} E_{\kappa}(x;2/(r-1)))=-\left(\overline{\kappa}_{j}|_{\alpha=2/(r-1)}+2(N-1)\right)\Delta_{N}^{(r-1)}E_{\kappa}(x;2/(r-1)).
\end{equation}

We have proved  that $(\Delta_{N}^{(r-1)} E_{\kappa}(x;2/(r-1)))$ is an eigenfunction of $\xi_{j}|_{\alpha=-2/(r-1)}$ for each $j$.  The eigenvalue can be reorganized as follows.  
On the one hand, we know from Equation \eqref{eqeigennonsym} that the eigenvalues ​​associated to $E_{\kappa}(x;2/(r-1))$ restricted to $\alpha=2/(r-1)$ are given by
\begin{equation*} 
\overline{\kappa}_{j}|_{\alpha=2/(r-1)} = \frac{2}{r-1}\kappa_{j} - \#\{i<j|\kappa_{i}\geq \kappa_{j}\} - \#\{i>j|\kappa_{i}> \kappa_{j}\}.
\end{equation*}
Now, given $\kappa_{j}$ in $\kappa$, we know that to $\kappa_{j}$ corresponds a cell in diagram of $\kappa$  and moreover, this cell has an associated cell $s$  in diagram of $\Lambda$. Then, we can express the eigenvalues $\overline{\kappa}_{j}$ in terms of arm-colength and leg-colength of cell $s$ in $\Lambda$. Given that 
$$ a'_{\La^*}(s)= \kappa_{j} - 1 \qquad \text{and} \qquad l'_{\La^*}(s)= \#\{i<j|\kappa_{i}\geq \kappa_{j}\} + \#\{i>j|\kappa_{i}> \kappa_{j} \},$$
 we can rewrite the eigenvalue as 
\begin{equation} \label{eqeigenvaluenonsymKappa}
\overline{\kappa}_{j}|_{\alpha=2/(r-1)} = \frac{2}{r-1} (a'_{\La^*}(s)+1) - l'_{\La^*}(s).
\end{equation}
On the other hand, from Equation \eqref{eqeigennonsym} and considering the composition $\kappa + (r-1)\delta'$, we have
\begin{multline*} 
\overline{(\kappa + (r-1)\delta')}_{j}= \alpha(\kappa_{j} + (r-1)\delta'_{j} ) - \\
\#\{i<j|\kappa_{i}+ (r-1)\delta'_{i}\geq \kappa_{j}+ (r-1)\delta'_{j}\} - \#\{i>j|\kappa_{i}+ (r-1)\delta'_{i}> \kappa_{j}+ (r-1)\delta'_{j} \}.
\end{multline*}
However, we can simplify this expression if we rewrite the eigenvalue in terms of $\Lambda':=\Lambda+(r-1)\delta$ the associated superpartition to $\kappa+(r-1)\delta'$. 
The same way as before, given $(\kappa + (r-1)\delta')_{j}$ in the composition $\kappa + (r-1)\delta'$, we know that to $(\kappa + (r-1)\delta')_{j}$ corresponds a cell in diagram of $\kappa + (r-1)\delta'$  and moreover, this cell has a cell $s'$ associated in diagram of $\Lambda'$. So, we have
\begin{equation*}
\begin{split}
a'_{\Lambda'^{*}}(s') &= \kappa_{j} - 1  + (r-1)\delta'_{j} \\
l'_{\Lambda'^{*}}(s') &= \#\{i<j|\kappa_{i}+ (r-1)\delta'_{i}\geq \kappa_{j}+ (r-1)\delta'_{j}\} + \#\{i>j|\kappa_{i}+ (r-1)\delta'_{i}> \kappa_{j}+ (r-1)\delta'_{j} \}
\end{split}
\end{equation*}
Hence, 
\begin{equation} \label{eqeigenvaluenonsymKappaprime}
\overline{(\kappa+(r-1)\delta')}_{j}|_{\alpha=-2/(r-1)} = -\frac{2}{(r-1)}(a'_{\Lambda'^{*}}(s')+1)- l'_{\Lambda'^{*}}(s').
\end{equation}
Now,  comparing the arm-colenght and leg-colenght of $\La$ and $\La'$, we get 
\begin{equation} \label{armandleg} a'_{\Lambda'^{*}}(s')=a'_{\La^*}(s) + N-l'_{\Lambda'^{*}}(s)-1 \qquad \text{and} \qquad l'_{\Lambda'^{*}}(s')=l'_{\La^*}(s) \end{equation}
Hence, by combining the Equations \eqref{eqcheredVanderJack}, \eqref{eqeigenvaluenonsymKappa}, \eqref{eqeigenvaluenonsymKappaprime} and \eqref{armandleg}, we conclude that 
$$E_{\kappa+(r-1)\delta'}(x;-2/(r-1)) \qquad \text{and} \qquad \Delta_N^{r-1} E_{\kappa}(x;2/(r-1))$$
have the same eigenvalues for each $\xi_{j}$ with $j=1,\ldots,N$.

In brief, we have   proved  that $(\Delta_{N}^{(r-1)} E_{\kappa}(x;2/(r-1)))$ as the same eigenvalues than $E_{\kappa+(r-1)\delta'}(x;-2/(r-1)) $.  Little work also shows that both polynomials exhibit triangular with dominant term $x^{\kappa+(r-1)\delta'}$.  Moreover, because of the form of $\kappa$, the composition $\kappa+(r-1)\delta'$ is weakly $(1,r,N|m)$-admissible.   Therefore, we can make use of Theorem \ref{jacknonsymk1weak} and conclude that
\begin{equation*}
E_{\kappa+(r-1)\delta'}(x_{1},\ldots,x_{N};-2/(r-1)) \propto \prod_{1\leq i<j \leq N}(x_{i}-x_{j})^{r-1}E_{\kappa}(x_{1},\ldots,x_{N};2/(r-1)),
\end{equation*}
i.e.,  the polynomials are equal up to a multiplicative numerical  factor.
\end{proof}

}

\begin{coro}\label{simpleclusternonsym} Let $r>0$ even and let $\lambda$ a partition with $\ell(\lambda)\leq N$. Then 
\begin{equation*}
E_{\lambda+(r-1)\delta_{N}}(z_{1},\ldots,z_{N};-2/(r-1))= \prod_{1\leq i<j \leq N}(x_{i}-x_{j})^{r-1} E_{\lambda}(z_{1},\ldots,z_{N};2/(r-1)).
\end{equation*} \label{CoroClusterPartition}
\end{coro}
\begin{remark}\label{remarkvalidity} The clustering property corresponding  to Corollary \ref{simpleclusternonsym} was first obtained in \cite[Proposition 2]{BarFor}.  The proof given in the latter reference used  the characterization of the non-symmetric Jack polynomials as the unique polynomials satisfying (A1') and (A2').  However,  the problem of the validity of this characterization at $\alpha=\alpha_{k,r}$ was not addressed by the authors.  Our result about the regularity and uniqueness, respectively given in Proposition \ref{regjnonsym} and Theorem \ref{jacknonsymk1weak}, now firmly establish the demonstration proposed in \cite{BarFor}.  
\end{remark}


Before stating the clustering properties for the polynomials with prescribed, we recall two useful formulas.
For this, let 
$$I=\{i_{1}, i_{2},\ldots,i_{n}\},\quad J=\{j_{1}, j_{2},\ldots,j_{m}\},\quad  
  \Delta_I=\prod_{\substack{i,j\in I\\i<j}}(x_i-x_j),\qquad  \Delta_J=\prod_{\substack{i,j\in J\\i<j}}(x_i-x_j).$$
Then, obviously, 
\begin{equation} \label{SymVandermondef}
\begin{split}
\mathrm{Sym}_{I} \Big(\Delta_{I} f(x_{1},\ldots,x_{N})\Big) &= \Delta_{I} \mathrm{Asym}_{I} \Big(f(x_{1},\ldots,x_{N})\Big),\\
\mathrm{Asym}_{J} \Big(\Delta_{J} f(x_{1},\ldots,x_{N})\Big) &= \Delta_{J} \mathrm{Sym}_{J}  \Big(f(x_{1},\ldots,x_{N})\Big).
\end{split}
\end{equation}

\begin{proposition} [Clustering $k=1$] 
Let $r$ be   positive and even. Let also $\Lambda$ be a superpartition of bi-degree $(n|m)$ with $\ell(\Lambda)\leq N$.  
\begin{itemize}
	\item [(i)] If $\Lambda$ is strict and weakly $(1,r,N)$-admissible, then
\begin{equation*}P_{\Lambda}^{\AS}(x_{1},\ldots,x_{N};-2/(r-1))= \prod_{m+1\leq i<j \leq N}(x_{i}-x_{j})^{r} Q(x_{1},\ldots,x_{N}).\end{equation*}
  \item [(ii)] If $\Lambda$ is moderately $(1,r,N)$-admissible, then
\begin{equation*}P_{\Lambda}^{\SS}(x_{1},\ldots,x_{N};-2/(r-1))= \prod_{1\leq i<j \leq m}(x_{i}-x_{j})^{r}\prod_{m+1\leq i<j \leq N}(x_{i}-x_{j})^{r} Q(x_{1},\ldots,x_{N}).\end{equation*} 
  \item [(iii)] If $\Lambda$ is moderately $(1,r,N)$-admissible and it is such that $\La_{m+1}>\ldots>\La_{N}$, then
\begin{equation*}P_{\Lambda}^{\SA}(x_{1},\ldots,x_{N};-2/(r-1))= \prod_{1\leq i<j \leq m}(x_{i}-x_{j})^{r} Q(x_{1},\ldots,x_{N}).\end{equation*}
  \item [(iv)] If $\Lambda$ is strict and weakly $(1,r,N)$-admissible, and it is such that $\La_{m+1}>\ldots>\La_{N}$, then
\begin{equation*}P_{\Lambda}^{\AA}(x_{1},\ldots,x_{N};-2/(r-1))= \prod_{1\leq i<j \leq N}(x_{i}-x_{j})^{r-1} Q(x_{1},\ldots,x_{N}).\end{equation*} 
\end{itemize}  \label{clusteringk1}
In the above equations,  $Q(x_{1},\ldots,x_{N})$ denotes some polynomial, which varies from one symmetry type to another. 
\end{proposition}

\begin{proof} 
Once again, all cases are similar, so we only provide the demonstration for the symmetry type AS, which corresponds to (i) above. 

As before,  we set $I=\{1,\ldots,m\}$ and $J=\{m+1,\ldots,N\}$. According to Definition \ref{defJackPrescibed} and Proposition \ref{normalizAS}, there is  a composition $\eta$, obtained by the concatenation of two partitions,  such that
$$P_{\Lambda}^{\AS}(x_{1},\ldots,x_{N};\alpha) \propto \mathrm{Asym}_{I} \mathrm{Sym}_{J} (E_{\eta}(x_{1},\ldots,x_{N};\alpha)) . $$
Given that  $\La$ is $(1,r,N)$-admissible,  then $\eta$ has the form $\kappa+(r-1)\delta'$ where $\kappa=(\la^+,\mu^+)$ is the composition obtained from $\eta$ after subtraction of the composition $(r-1)\delta'$. Moreover, since $\La$ is strict and weakly $(1,r,N)$-admissible, we know that $\kappa$ is such that $\mu^+$ is  strictly decreasing. Thus,
\begin{equation}\label{longproof1}P_{\Lambda}^{\AS}(x_{1},\ldots,x_{N};-2/(r-1)) \propto \mathrm{Asym}_{I} \mathrm{Sym}_{J} (E_{\kappa+(r-1)\delta'}(x_{1},\ldots,x_{N};-2/(r-1))).\end{equation}
Now, by Proposition \ref{agrupk1}, we also have
\begin{equation}\label{longproof2} E_{\kappa+(r-1)\delta'}(x_{1},\ldots,x_{N};-2/(r-1)) \propto \prod_{1\leq i<j \leq N}(x_{i}-x_{j})^{r-1}E_{\kappa}(x_{1},\ldots,x_{N};2/(r-1)).\end{equation}
The substitution of \eqref{longproof2} into \eqref{longproof1}, followed by the use of \eqref{SymVandermondef},  leads to 
\begin{multline*}
P_{\Lambda}^{\AS}(x_{1},\ldots,x_{N};-2/(r-1)) \propto \\  (\Delta_{J})^{(r-1)} (\Delta_{I})^{(r-1)} \mathrm{Sym}_{I}\left(\prod_{i=1}^{m}\prod_{j=m+1}^{N}(x_{i}-x_{j})^{(r-1)} \mathrm{Asym}_{J} E_{\kappa}(x_{1},\ldots,x_{N};2/(r-1))\right)
\end{multline*}
Now, we know that $\mathrm{Asym}_{J} E_{\kappa}(x_{1},\ldots,x_{N};2/(r-1))$ is antisymmetric with respect to the set of variables indexed by $J$,  so we can factorize the antisymmetric factor $\prod_{m+1\leq i<j \leq N}(x_{i}-x_{j})$. Exploiting once again \eqref{SymVandermondef}, we  finally obtain   
\begin{equation*}
P_{\Lambda}^{\AS}(x_{1},\ldots,x_{N};-2/(r-1)) \propto \prod_{m+1\leq i<j \leq N}(x_{i}-x_{j})^{r}  Q(x_{1},\ldots,x_{N}),
\end{equation*}
where
\begin{multline}\label{exactQ}
Q(x_{1},\ldots,x_{N})=\prod_{1\leq i<j \leq m}(x_{i}-x_{j})^{r-1}\prod_{i=1}^{m}\prod_{j=m+1}^{N}(x_{i}-x_{j})^{(r-1)}
\\
\times \mathrm{Sym}_{I}\left(\frac{\mathrm{Asym}_{J} E_{\kappa}(x_{1},\ldots,x_{N};2/(r-1))}{\prod_{m+1\leq i<j \leq N}(x_{i}-x_{j})}\right).\end{multline}
\end{proof}

\begin{remark} The case (i) was first conjectured in \cite{dlm_cmp2}  in the context of symmetric polynomials in superspace.  All other cases are new.
\end{remark}

\begin{coro} Let $\alpha=-\frac{2}{r-1}$ and let $r$ be positive and even. Moreover, for any positive integer $\rho$, let 
$$ \rho \delta_N=\big(\rho(N-1),\rho(N-2),\ldots,\rho,0\big).$$
 Then, the   antisymmetric Jack polynomial satisfies 
\begin{equation*}
S_{(r-1)\delta_{N}}(x_1,\ldots,x_N;\alpha)=\prod_{1\leq i<j \leq N}(x_{i}-x_{j})^{(r-1)},
\end{equation*}
while the  symmetric Jack polynomial satisfies 
$$ P_{ r\delta_N}(x_1,\ldots,x_N;\alpha)=\prod_{1\leq i<j \leq N}(x_{i}-x_{j})^{r}.$$
\end{coro}
\begin{proof}
We recall that if $\ell(\la)=N$, then $$S_\la(x;\alpha)=P^{\AS}_{(\la;\emptyset)}(x;\alpha)\quad\text{and}\quad P_{ \lambda}(x,\alpha)=P^{\AS}_{(\emptyset;\; \lambda)}(x,\alpha).$$  
The first result then follows  from Proposition \ref{clusteringk1} and Equation \eqref{exactQ} for the case with $m=N$ and $\kappa=\emptyset$.  The second result also follows from Proposition \ref{clusteringk1} and Equation \eqref{exactQ}, but this time,  with $m=0$ and $\kappa=\delta_N$.
\end{proof}

\section{Translation invariance and some clustering properties for $k\geq 1$}

In this section, we first generalize the work of Luque and Jolicoeur about translationally invariant Jack polynomials \cite{jl}.  We indeed find the necessary and sufficient conditions for the translational invariance of the Jack polynomial with prescribed symmetry AS.  To be more precise, let  
\beq  P_\La=P^\AS_\La(x_1,\ldots,x_N;\alpha) ,\label{eqInvJack} \eeq
and suppose that
\begin{align} \alpha&=\alpha_{k,r}\label{eqInvAlpha},\\ 
\La &\quad \text{is a strict and weakly} \,(k,r,N)\text{-admissible superpartition.}\label{eqInvAdm}
\end{align}
Then, as was stated in Theorem \ref{TheoTransla},  $P_\La$ is invariant under translation  if and only if   conditions (D1) and (D2) are satisfied.  The latter conditions concern  the corners in the diagram of $\La$. The proof relies on combinatorial formulas obtained in \cite{dlm_cmp2} that generalize Lassalle's results \cite{lassalle,lassalle2}  about the action of the operator 
\beq \label{opL+}L_+=\sum_{i=1}^N\frac{\partial}{\partial x_i}.\eeq 
We then apply the result about the translationally invariant polynomials and prove that certain Jack polynomials with prescribed symmetry AS admit clusters of size $k$ and order $r$.

\subsection{Generators of translation}

The action of $L_+$ on a Jack polynomial with prescribed symmetry AS, $P^\AS_\La(x;\alpha)$, is in general very complicated.  However it can be decomposed in terms of two basic operators, $Q_\ocircle$ and $Q_\Box$.  Their respective action  on $P^\AS_\La(x;\alpha)$ can be translated into simple transformations of the diagram of $\La$, namely the removal of a circle  and the conversion of a box into a circle.

Now, let $I=\{1,\ldots,m\}$, $I_+=\{1,\ldots,m+1\}$, $I_-=\{1,\ldots,m-1\}$,  $J=\{m+1,\ldots, N\}$, $J_+=\{m,\ldots, N\}$, and $J_-=\{m+2,\ldots, N\}$. We define $Q_\ocircle$ and $Q_\Box$ as follows:   
\begin{align}
Q_\ocircle\;: \qquad & \;\mathscr{A}_I\otimes \mathscr{S}_J\longrightarrow \mathscr{A}_{I_-}\otimes \mathscr{S}_{J_+}\; ;  &\hspace{-0ex} f&\longmapsto (1+\sum_{i=m+1}^N K_{i,m})f\; ,&&1\leq m\leq N,\\
Q_{\Box}\;:  \qquad & \;\mathscr{A}_I\otimes \mathscr{S}_J\longrightarrow \mathscr{A}_{I_+}\otimes \mathscr{S}_{J_-}\; ;  &\hspace{-0ex}  f&\longmapsto (1-\sum_{i=1}^m K_{i,m+1})\circ\frac{\partial f}{\partial x_{m+1}}\;, &&0\leq m\leq N-1.
\end{align}
Note that for the extreme case $m=0$, we set $Q_\ocircle=0$.  Similarly, for $m=N$, we set $Q_{\Box}=0$. 

\begin{lem} \label{QQL+}On the space $\mathscr{A}_I\otimes \mathscr{S}_J$, we have $\displaystyle  Q_\ocircle\circ Q_{\Box}+Q_{\Box}\circ Q_\ocircle=L_+$. 
\end{lem}
\begin{proof} Let $f$ be an element of  $\mathscr{A}_I\otimes \mathscr{S}_J$, which means that $f$ is a polynomial in the variables $x_{1},\ldots,x_{N}$ that is antisymmetric with respect to $x_1,\ldots, x_m$ and symmetric with respect to $x_{m+1},\ldots, x_N$. We must show that 
\beq  \label{eqlemQ0}(Q_\ocircle\circ Q_{\Box})( f) +(Q_{\Box}\circ Q_\ocircle)(f) =\sum_{i=1}^N\frac{\partial f }{\partial x_{i}}.\eeq

On the one hand, 
\begin{multline}   \label{eqlemQ1}(Q_\ocircle\circ Q_{\Box})(f)=\\ \frac{\partial f}{\partial x_{m+1}}-\sum_{i=1}^{m}K_{i,m+1}\frac{\partial f}{\partial x_{m+1}}+\sum_{j=m+2}^{N}K_{j,m+1}\frac{\partial f}{\partial x_{m+1}}-\sum_{j=m+2}^{N}K_{j,m+1} \sum_{i=1}^{m}K_{i,m+1}\frac{\partial f}{\partial x_{m+1}}.\end{multline} 
However, the  symmetry  properties of  $f$ imply 
$$\sum_{j=m+2}^{N}K_{j,m+1}\frac{\partial f}{\partial x_{m+1}}=\sum_{j=m+2}^{N}\frac{\partial f}{\partial x_{j}} \quad \text{and} \quad \sum_{j=m+2}^{N}K_{j,m+1} \sum_{i=1}^{m}K_{i,m+1}\frac{\partial f}{\partial x_{m+1}}=\sum_{i=1}^{m}\sum_{j=m+2}^{N}\frac{\partial}{\partial x_{i}}(K_{i,j} f).$$
By substituting the last equalities into \eqref{eqlemQ1}, we obtain 
\beq   \label{eqlemQ1b} (Q_\ocircle\circ Q_{\Box})(f)=\frac{\partial f}{\partial x_{m+1}}+\sum_{j=m+2}^{N}\frac{\partial f}{\partial x_{j}}-\sum_{i=1}^{m}\sum_{j=m+1}^{N}\frac{\partial}{\partial x_{i}}(K_{i,j} f).\eeq

On the other hand, 
\beq \label{eqlemQ2} (Q_{\Box}\circ Q_\ocircle)(f)=\frac{\partial f}{\partial x_{m}}+\sum_{j=m+1}^{N}\frac{\partial}{\partial x_{m}}(K_{j,m}f)-\sum_{i=1}^{m-1}K_{i,m}\frac{\partial f}{\partial x_{m}}-\sum_{i=1}^{m-1}K_{i,m}\frac{\partial}{\partial x_{m}} \left( \sum_{j=m+1}^{N}K_{j,m}f \right).\eeq
Once again, the symmetry properties of $f$ allow to simplify this equation.  Indeed,  
$$\sum_{i=1}^{m-1}K_{i,m}\frac{\partial f}{\partial x_{m}}=-\sum_{i=1}^{m-1}\frac{\partial f}{\partial x_{i}} \quad \text{and} \quad \sum_{i=1}^{m-1}K_{i,m}\frac{\partial}{\partial x_{m}} \left( \sum_{j=m+1}^{N}K_{j,m}f \right)=-\sum_{i=1}^{m-1}\sum_{j=m+1}^{N}\frac{\partial}{\partial x_{i}}(K_{i,j}f). $$
Then, 
\beq   \label{eqlemQ2b} (Q_{\Box}\circ Q_\ocircle)(f) = \frac{\partial f}{\partial x_{m}}+\sum_{i=1}^{m-1}\frac{\partial f}{\partial x_{i}}+\sum_{i=1}^{m}\sum_{j=m+1}^{N}\frac{\partial}{\partial x_{i}}(K_{i,j}f).\eeq 

We finally sum Equations \eqref{eqlemQ1b}  and \eqref{eqlemQ2b}.  This yields Equation \eqref{eqlemQ0}, as expected.
\end{proof}

The explicit action of $Q_\ocircle$ and $Q_\Box$ on the polynomial  $P^\AS_\Lambda(x;\alpha)$ can be read off from Proposition 9 of \cite{dlm_cmp2}.  Indeed, this proposition is concerned with the action of differential operators --related to the super-Virasoro algebra-- on the Jack superpolynomials, denoted by  $P_\Lambda(x;\theta;\alpha)$, which contain Grassmann variables $\theta_1,\ldots,\theta_N$.  Among the operators studied in \cite{dlm_cmp2}, there are $$Q^\perp=\sum_i\frac{\partial}{\partial \theta_i}\quad\text{and}\quad q=\sum_i\theta_i\frac{\partial}{\partial x_i} .$$  Now,  a Jack superpolynomial of degree $m$ in the variables $\theta_i$, can be decomposed as follows \cite{dlm_cmp}: 
$$  P_\Lambda(x;\theta;\alpha)=\sum_{1\leq j_1<\ldots<j_m\leq N}\theta_{j_1}\cdots\theta_{j_m}f^{ j_1,\ldots,j_m}(x;\alpha), $$
where $f^{ j_1,\ldots,j_m}(x;\alpha)$ belongs to the space $\mathscr{A}_ {\{j_1,\ldots,j_m\}}\otimes \mathscr{S}_{\{1,\ldots,N\}\setminus \{j_1,\ldots,j_m\}}$ and is an eigenfunction of the operator $D$ defined in \eqref{defD}.  This means in particular that  $f^{ 1,\ldots,m}(x;\alpha)$ is exactly equal to our $P^\AS_\La(x;\alpha)$. It is then an easy exercise to show that the formula for the action of $Q^\perp$ on $  P_\Lambda(x;\theta;\alpha)$ provides the formula for  the action of $Q_\ocircle$ on $   P^\AS_\La(x;\alpha)$.  Similarly,   $q  P_\Lambda(x;\theta;\alpha)$  is related to $Q_\Box P^\AS_\La(x;\alpha)$.  Note that the formulas obtained in \cite{dlm_cmp2} are given in terms of the following upper and lower-hook lengths: 
\begin{equation}
\begin{split}
h_\Lambda^{(\alpha)}(s)& = l_{\Lambda^{\circledast}}(s) + \alpha(a_{\Lambda^{\ast}}(s)+ 1)\\
h_\alpha^{(\Lambda)}(s)& = l_{\Lambda^{\ast}}(s) + 1 + \alpha(a_{\Lambda^{\circledast}}(s))
\end{split}
\end{equation} 

\begin{proposition}\label{actionqQ}\cite{dlm_cmp2}
The action of the operators $Q_\ocircle$ and $Q_{\Box}$ on the Jack polynomial with prescribed symmetry $P_\La=P_\Lambda^\AS{(x;\alpha)}$ is 
\begin{equation} \label{Qcirle}
Q_\ocircle(P_\Lambda )=\sum_{\Omega}(-1)^{\#\Omega^{\circ}}\left(\prod_{s \in row _{\Omega^{\circ}}}\frac{h_\alpha^{(\Omega)}(s)}{h_\alpha^{(\Lambda)}(s)}\right)(N+1-i+\alpha(j-1))P_\Omega 
\end{equation}
\begin{equation} \label{Qsquare}
Q_{\Box}(P_\Lambda )=\sum_{\Omega}(-1)^{\#\Omega^{\circ}}\left(\prod_{s \in row _{\Omega^{\circ}}}\frac{h_\Lambda^{(\alpha)}(s)}{h_\Omega^{(\alpha)}(s)}\right) P_\Omega 
\end{equation}
where the sum is taken in \eqref{Qcirle} over all $\Omega's$ obtained by removing a circle from $\Lambda$; while the sum is taken in \eqref{Qsquare} over all $\Omega's$ obtained by
converting a box of $\Lambda$ into a circle. Also, in each case $\Lambda$ and $\Omega$ differ in exactly one cell which we call the marked cell and whose position is denoted in the formulas by $(i,j)$. 
The symbol $\#\Omega^{\circ}$ stands for the number of circles in $\Omega$ above the marked cell.  
The symbol $row _{\Omega^{\circ}}$ stands for the row of $\Omega$ and $\Lambda$ to the left of the marked cell.
\end{proposition}

\begin{remark} Let $\Lambda$ be a superpartition such that in the corresponding diagram, all  corners are boxes.    Then, in Equation \eqref{Qcirle}, we cannot remove any circle from the diagram of $\La$ and we are forced to conclude that $Q_\ocircle P_\Lambda=0$.  This is coherent with the fact that in such case, $P_\Lambda^\AS{(x;\alpha)}$ is a symmetric polynomial and  according with our convention, $Q_\ocircle f=0$ for all $f\in  \mathscr{S}_{\{1,\ldots,N\}}$.   

Similarly, if $\Lambda$ be a superpartition such that in its diagram, all corners  are circles.  Then,   we cannot transform a box in the diagram of $\La$ into a circle.  This complies with our convention. Indeed,  in such case,  $P_\Lambda^\AS{(x;\alpha)}$ is an  antisymmetric polynomial and we have set $Q_{\Box}f=0$ for all $f\in  \mathscr{A}_{\{1,\ldots,N\}}$. \end{remark}

\subsection{General invariance}

We will determine whether a Jack polynomial with prescribed symmetry is translationally invariant by looking at the shape of the diagram associated to the  indexing superpartition.  We will pay a special attention to the corners in the diagram.

\begin{definition} \label{defcorner} Let $D$ be  the diagram associated to the superpartition  $\Lambda$. The cell $(i,j)\in D$ is a corner if $(i+1,j)\notin D$ and $(i,j+1) \notin D$.   We say that the corner  $(i,j)$ is an outer corner if the row $i-1$ and the column $j-1$ do not have corners.  On the contrary,   the corner  $(i,j)$ is an inner   corner  if the row $i-1$ and the column $j-1$ have corners. A corner that neither outer nor inner is a bordering corner. Note that in the above definitions, it is assumed that each point of the form $(0,j)$ or $(i,0)$ is a corner. 
\end{definition}

\begin{lem}\label{LemCorners} Let $D'$ be the diagram obtained by removing the corner $(i,j)$ from diagram $D$, which contains $c$ corners.  Then, the number of corners in $D'$ is:
\begin{itemize} \item  $c-1$ if $(i,j)$ is an inner corner;
\item $c$ if  $(i,j)$ is a bordering corner;
\item $c+1$ if $(i,j)$ is an outer corner.
\end{itemize}   
\end{lem}
\begin{proof}This follows immediately from the above definitions.
\end{proof}

\begin{lem} \label{lemaQbox} Assume   \eqref{eqInvJack}, \eqref{eqInvAlpha}, and \eqref{eqInvAdm}.   Then, $Q_{\Box}(P_\La)=0$    if and only if  $\Lambda$ is   such that all the corners in its diagram are circles.
\end{lem} 
\begin{proof}[Proof]  
According to Proposition \ref{actionqQ},   $Q_{\Box} (P_{\Lambda})$ vanishes if and only if each corner of $\Lambda$ is either a circle or a box located at $(i,j')$ such that for some  $j<j'$, we have 
$$h_\Lambda^{(\alpha_{k,r})}(i,j) = l_{\Lambda^{\circledast}}(i,j) + \alpha_{k,r}(a_{\Lambda^{\ast}}(i,j)+ 1)=0$$

  Now, $h_\Lambda^{(\alpha_{k,r})}(i,j)=0$ only if for some positive integer $\bar k$, we have  $a_{\Lambda^{\ast}}(i,j)+ 1=\bar k(r-1)$ and $l_{\Lambda^{\circledast}}(i,j)=(k+1)\bar k$. This implies
 \begin{equation}
\label{1eradesigualdad}\Lambda^{\circledast}_i- \Lambda^{*}_{i+\bar k(k+1)}\leq \bar k r-\bar k.
\end{equation}  

On the other hand, Lemma  \ref{leminequality2} implies that $ \Lambda^{\circledast}_{i+1}- \Lambda^{*}_{i+\bar k(k+1)}\geq \bar k r$.  Moreover, $\Lambda_{i}^{\circledast}\geq\Lambda_{i+1}^{\circledast}$, so that $\Lambda^{\circledast}_i- \Lambda^{*}_{i+\bar k(k+1)}\geq \bar k r$. This inequality contradicts \eqref{1eradesigualdad}.

Therefore, if $\Lambda$ is a  $(k,r,N)$-admissible superpartition, $Q_{\Box} (P_\Lambda)$ vanishes if and only if  all the corners in  $\Lambda$ are circles.
\end{proof}

The conditions for the vanishing of the action of $Q_{\ocircle}$ on a Jack polynomial with prescribed symmetry are more involved.  They require a finer characterization of the different types of hooks formed from the corners of the diagrams.  

\begin{definition} Let $D$ be the diagram associated to the superpartition $\Lambda$.  Let  $(i,j)\in D$ be a circled corner. We say that   $(i,j)$ is the upper corner of a hook of type:
\begin{itemize} \item[a)] $C_{k,r}$ if the box $(i,j-r)\in D$ and it satisfies $l_{\Lambda^{\ast}}(i,j-r)=l_{\Lambda^{\circledast}}(i,j-r)=k$; 
 \item[b)] $\tilde C_{k,r}$ if the box $(i,j-r)\in D$ and it satisfies $l_{\Lambda^{\ast}}(i,j-r)=k$  together with $l_{\Lambda^{\circledast}}(i,j-r)=k+1$. 
\end{itemize}
Similarly, when $(i,j)\in D$ is a boxed corner, we  say $(i,j)$ is the upper corner of a hook of type:
\begin{itemize} \item[c)] $B_{k,r}$ if the box $(i,j-r)\in D$ and it satisfies $l_{\Lambda^{\ast}}(i,j-r)=l_{\Lambda^{\circledast}}(i,j-r)=k$. 
 \item[d)] $\tilde B_{k,r}$ if the box $(i,j-r)\in D$ and it satisfies $l_{\Lambda^{\ast}}(i,j-r)=k$  together with $l_{\Lambda^{\circledast}}(i,j-r)=k+1$. 
\end{itemize}
The hooks are illustrated in Figure \ref{fige_scuadras}.
\end{definition}

Let us consider a concrete example.  For this we fix $k=4$, $r=3$ and $N=18$.  The following diagram is associated to a strict  and weakly $(k,r,N)$-admissible superpartition:

{\small
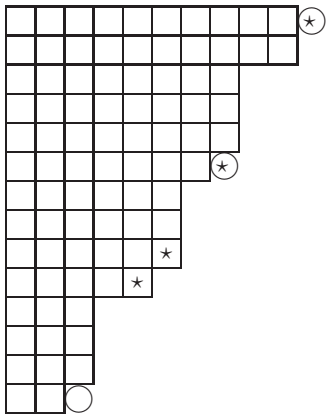
$$\tableau[scY]{  &&&&&&&&&&\bl\tcercle{$\star$}\\&&&&&&&&&\\&&&&&&&\\&&&&&&&\\&&&&&&&\\&&&&&&&\bl\tcercle{$\star$}\\&&&&&\\&&&&&\\&&&&&{\star}\\&&&&{\star}\\&&\\&&\\&&\\&&\bl\tcercle{}}$$
}

\noindent Each cell marked with a star is the upper corner of one of the four types of hooks.    The first one,   located at the position $(1,11)$, is the upper corner of a  hook of type $\tilde C_{4,3}$.  The  one, located at the position $(6,8)$, belongs to a hook of type $C_{4,3}$.   Similarly, the third and the fourth corners, are the upper corners  of hooks of type $\tilde B_{4,3}$ and   $B_{4,3}$, respectively.

\begin{lem} \label{lemaQcircle}  Assume   \eqref{eqInvJack}, \eqref{eqInvAlpha}, and \eqref{eqInvAdm}.  Then, $Q_{\ocircle}(P_\La)=0$ if and only if each corner in the diagram of $\Lambda$ is either:
\begin{itemize}
	\item [(i)]  a box;  
	\item [(ii)]  a circle and the upper corner of a hook of type $C_{k,r}$ or $\tilde C_{k,r}$;
	\item [(iii)]   a circle with coordinates  $(i,j)$ such that $i=N+1-\bar k(k+1)$ and $j=\bar k(r-1)+1$ for some positive integer $\bar k$.    
\end{itemize}
Note that there is at most one corner $(i,j)$ satisfying the criterion (iii).  
\end{lem}
\begin{proof}
According to Proposition \ref{actionqQ},  $Q_\ocircle (P_{\Lambda})=0$  iff, each corner $(i,j)$ satisfies at least one of the following criteria:
\begin{enumerate} \item the cell $(i,j)$ is a box; \item the cell $(i,j)$ is a circle and there is a $j'<j$ such that  $h_\alpha^{(\Omega)}(i,j')=0$, where $h_\alpha^{(\Omega)}(i,j') =   l_{\Omega^{\ast}}(i,j') + 1 + \alpha_{k,r}(a_{\Omega^{\circledast}}(i,j'))$ and $\Omega$ is the diagram obtained from $\Lambda$ by removing the circle in $(i,j)$ ;\item the cell $(i,j)$ is a circle and it is such that $N+1-i+\alpha_{k,r}(j-1)=0$.      
\end{enumerate}
The first criterion being trivial, we turn to the second.  Obviously, $h_\alpha^{(\Omega)}(i,j')=0$ iff there exists a positive integer $\bar k$ such that 
$a_{\Omega^{\circledast}}(i,j')=\bar k(r-1)$ and $ l_{\Omega^{\ast}}(i,j')=\bar k(k+1)-1$. The first condition is equivalent to { $j-j'=\bar k(r-1)-1$}.  The second is equivalent to say that $\Lambda^*_{i+\bar k (k+1)-1}\geq j'$ and that the cell $(i+\bar k k+\bar k,j')$ is empty or a circle.  Suppose further that $\bar k=1$.  Then, we have  shown that $h_\alpha^{(\Omega)}(i,j')=0$ iff $j'=j-r+2$, $\Lambda^*_{i+k}\geq j'$ and $\Lambda^*_{i+k+1}<j'$ (i.e.,  $\Lambda^\circledast_{i+k+1}< j'$ or $\Lambda^\circledast_{i+k+1}=j'$ ), this corresponds to the two hooks given above.  Now, suposse  $\bar k=2$.  On the one hand,  we have $\Lambda^*_{i+2k+1}\geq j'=j-2(r-1)+1=\Lambda_i^\circledast-2(r-1)+1$, i.e, \begin{equation}\label{inegual1}\Lambda_i^\circledast-\Lambda^*_{i+2k+1}\leq {2r-3}.\end{equation}  On the other hand, the admissibility requires $ \Lambda_i^\circledast-\Lambda^*_{i+k}\geq r$ and $\Lambda_{i+k}^*-\Lambda^*_{i+2k}\geq r-1$.  Then, \begin{equation} \label{inegual2}\Lambda_i^\circledast-\Lambda^*_{i+2k+1}\geq\Lambda_i^\circledast-\Lambda^*_{i+2k}\geq 2r-1.\end{equation}  
Inequalities \eqref{inegual1} and \eqref{inegual2} are contradictory, so we conclude that $\bar k$ cannot be equal to $2$.  In the same way, one  easily shows that $\bar k$ cannot be greater than 2.
 
Now consider the third criterion.   As $N+1-i>0$, the factor $N+1-i+\alpha_{k,r}(j-1)$ vanishes iff $j=\bar k(r-1)+1$ and $N=i+\bar k(k+1)-1$, for some positive integer $\bar k$.  Now suppose there is another corner $(i',j')$ such that $N+1-i'+\alpha_{k,r}(j'-1)$.  Then, $j'=\bar k'(r-1)+1$ y $N=i'+\bar k'(k+1)-1$, for some positive integer $\bar k'$.  Without loss of generality, we can assume $i<i'$, which implies $j>j'$, i.e. $\bar k>\bar k'$. Let $n=\bar k-\bar k'$.  Then, $j-j'=\Lambda^{\circledast}_i- \Lambda^{\circledast}_{i'}=n(r-1)$, which implies  $\Lambda^{\circledast}_i- \Lambda^{*}_{i'}=n(r-1)+1$. Using $N=i+\bar k(k+1)-1=i'+\bar k'(k+1)-1$, we get $i'=i+n(k+1)$. 
Also, $\Lambda^{\circledast}_i- \Lambda^{*}_{i+n(k+1)} >\Lambda^{\circledast}_i- \Lambda^{*}_{i+nk}$, thus 
\begin{equation} \label{deseq_1}\Lambda^{\circledast}_i- \Lambda^{*}_{i+nk}\leq n(r-1).\end{equation}  However, by using the admissibility and the fact that  
\begin{equation} \Lambda^{\circledast}_i- \Lambda^{*}_{i+nk}=\Lambda^{\circledast}_i-\Lambda^{*}_{i+k}+\Lambda^{*}_{i+k}-\Lambda^{*}_{i+2k}+\Lambda^{*}_{i+2k}+\ldots+ \Lambda^{*}_{i+(n-1)k} -\Lambda^{*}_{i+nk},
\end{equation} 
one easily shows that
\begin{equation} \label{deseq_2} \Lambda^{\circledast}_i- \Lambda^{*}_{i+nk}\geq r+(n-1)(r-1)=n r-n+1.
\end{equation}
Obviously, Equations \eqref{deseq_1} and \eqref{deseq_2} are contradictory.  Therefore no more than one corner is such that
$N+1-i+\alpha_{k,r}(j-1)=0$.
\end{proof}

\begin{coro} \label{ultimaesquina}  Assume   \eqref{eqInvJack}, \eqref{eqInvAlpha}, and \eqref{eqInvAdm}.   Suppose moreover that the last corner in $\Lambda$'s diagram   is a circle. Let $(\ell, j)$ the coordinates of the last corner.  Then, $Q_\ocircle (P_\La )=0$ only if $N=\ell+k$ and $j=r$. \end{coro}
\begin{proof} 
  According to the previous proposition, as  $(\ell, j)$ cannot be the upper corner of a hook, $Q_\ocircle (P_\Lambda)=0$ only if the condition (iii) is met for the corner $(\ell, j)$.  This means that $Q_\ocircle (P_\Lambda)=0$  only if $\ell=N+1-\bar k(k+1)$ and $j=\bar k(r-1)+1$ for some positive integer  $\bar k$. Now, the admissibility condition requires $\ell+k\geq N$, i.e.,   $$N+1-\bar k(k+1)+k\geq N.$$ This is true iff $\bar k=1$.  Thus, $Q_\ocircle (P_\Lambda )=0$ only if $\ell=N-k$ and $j=r$.    
\end{proof}

\begin{proposition}  \label{decompoL+}   Assume   \eqref{eqInvJack}, \eqref{eqInvAlpha}, and \eqref{eqInvAdm}. Then, $P_\La$ is invariant under translation if and only if $Q_{\Box}(Q_\ocircle P_\La)=0$ and $Q_\ocircle(Q_{\Box} P_\La)=0$.
\end{proposition}
\begin{proof}  
Clearly, $P_\La$ is translationally invariant iff $L_+(P_\La)=0$. Moreover, we know from Lemma \ref{QQL+} that $L_+(P_\La)=Q_{\Box}(Q_\ocircle P_\La)+Q_\ocircle(Q_{\Box} P_\La)$.  Thus,  if $Q_{\Box}(Q_\ocircle P)=0$ and $Q_\ocircle(Q_{\Box} P)=0$ then $L_{+} P=0$. 

It remains to show  that if $L_{+} P=0$, then $Q_{\Box}(Q_\ocircle P_\La)=0$ and $Q_\ocircle(Q_{\Box} P_\La)=0$. In fact, we are going to prove the contrapositive: if $Q_{\Box}(Q_\ocircle P_\La) \neq 0$ or $Q_\ocircle(Q_{\Box} P_\La) \neq 0$ then $L_{+} P \neq 0$.   However, if $Q_{\Box}(Q_\ocircle P_\La) \neq 0$ and  $Q_\ocircle(Q_{\Box} P_\La) =0$, or if $Q_{\Box}(Q_\ocircle P_\La) = 0$ and  $Q_\ocircle(Q_{\Box} P_\La) \neq 0$, then automatically $L_{+} P_\La\neq 0$.  Consequently, we need to prove the following statement: 
\beq \label{proofQQ} Q_{\Box}(Q_\ocircle P_\La) \neq 0\quad\text{and}\quad  Q_\ocircle(Q_{\Box} P_\La) \neq 0\quad \Longrightarrow \quad Q_{\Box}Q_\ocircle(P_\Lambda) + Q_\ocircle Q_{\Box}(P_\Lambda) \neq 0.\eeq

 We assume that $Q_{\Box}(Q_\ocircle P_\La) \neq 0$ and $Q_\ocircle(Q_{\Box} P_\La) \neq 0$. Then,  $Q_\ocircle P_\La \neq 0$ and $Q_\Box P_\La \neq 0$. According to Lemma   \ref{lemaQcircle},  the first equation  implies that there is at least one circle  in the diagram of $\Lambda$  that does not satisfy the conditions (ii) and (iii).  Let $(i,j)$ denote the position of such a circle.  Moreover, according to Lemma \ref{lemaQbox}, the second equation implies that there must be at least one boxed corner in the diagram of $\La$.    Let    $(\bar i, \bar j)$ be its position.   

Let $\Upsilon$ be the superpartition obtained from $\La$  by removing the  circle $(i,j)$ and by converting a box $(\bar i, \bar j)$ into a circle. There is only one way to get $P_\Upsilon$ by acting with $Q_{\Box}Q_\ocircle$ on $P_\Lambda$   by acting with $Q_\ocircle Q_{\Box}$ on $P_\Lambda$. 
Thus, it is enough to verify that the coefficients of the polynomial $P_\Upsilon$  in the expansions of   $Q_{\Box}(Q_\ocircle P_\La)$ and $Q_\ocircle Q_{\Box}(P_\La)$ are not the same (up to a sign).

Let   $\Omega^{1}$ be the superpartition obtained from $\Lambda$ by removing the   circle in $(i,j)$.  Clearly, the coefficient of $P_\Upsilon$ in $Q_{\Box}(Q_\ocircle P_\La)$ is equal to the product of two coefficients:   $c_{\Lambda,\Omega^1}$, the  coefficient  of $P_{\Omega^{1}}$ in  $Q_\ocircle(P_\Lambda)$, and   $ b_{\Omega^1,\Upsilon} $, the coefficient of   $P_\Upsilon$  in   $Q_{\Box}(P_{\Omega^{1}})$.  Similarly, if ${\Omega^{2}}$ denotes the superpartition obtained from $\La$ by converting the box $(\bar i,\bar j)$ into a circle, then the the coefficient of $P_\Upsilon$ in $Q_\ocircle(Q_{\Box} P_\La)$ is the product of the two following coefficients:    $b_{\Lambda,\Omega^{2}}$, the coefficient of $P_{\Omega^{2}}$ in $Q_{\Box} P_\La$, and   $c_{\Omega^{2},\Upsilon}$, the coefficient of $P_\Upsilon$ in $Q_\ocircle(P_{\Omega^{2}})$.  In short, 
\begin{equation}
Q_{\Box} Q_\ocircle(P_\Lambda)= c_{\Lambda, \Omega^{1}} \, b_{\Omega^{1},\Upsilon} \,   P_\Upsilon + \ldots
\end{equation}
\begin{equation}
Q_\ocircle Q_{\Box}(P_\Lambda)=   b_{\Lambda, \Omega^{2}} \, c_{\Omega^{2},\Upsilon} \,  P_\Upsilon  + \ldots
\end{equation}
where $\ldots$ indicates terms linearly independent from $P_\Upsilon$. We recall that the coefficients $b$ and $c$ can be read off the equations in Proposition \ref{actionqQ}.

Now, we need to distinguish two cases: (1) the box is located above the circle in the  diagram of $\Lambda$, which means $\bar i < i$, and (2)  the box is located under the circle in the  diagram of $\Lambda$, which means $\bar i > i$. 

Suppose first that the box is located above the circle,  i.e.,  $\bar i < i$. Obviously, $b_{\Lambda, \Omega^{2}}$ is not zero.   Moreover,  $c_{\Omega^{2},\Upsilon}$ is equal to $c_{ \Lambda, \Omega^{1}}$ .  This can be understood as follows.  These coefficients depend only on $N$, the coordinates of the marked cell, which are $(i,j)$ in both cases, and on ratios of hook-lengths for the cells in the row to the left of the marked cell.   Given that the marked cell is below the cell $(\bar i,\bar j)$, the hook-lengths involved in the coefficients are not affected by any prior transformation $\La\to\Omega^{2}$, so the coefficients are equal.  The situation is not so simple for $b_{\Lambda, \Omega^{2}}$ and $b_{\Omega^{1},\Upsilon} $, so explicit formulas for these coefficients are required.  Up to a sign, they are  
\begin{equation}
d_{\Lambda,\Omega^{2}}  = \left(\prod_{1\leq l \leq \bar j-1}\frac{h_\Lambda^{(\alpha)}(\bar i,l)}{h_{\Omega^2}^{(\alpha)}(\bar i,l)}\right),\qquad
d_{\Omega^{1},\Upsilon}  = \left(\prod_{1\leq l \leq \bar j-1}\frac{h_{\Omega^1}^{(\alpha)}(\bar i,l)}{h_{\Upsilon}^{(\alpha)}(\bar i,l)}\right)
\end{equation}
 It is important to note that
$$h_\Lambda^{(\alpha)}(\bar i,l) = h_{\Omega^1}^{(\alpha)}(\bar i,l) \quad \forall \; 1\leq l \leq \bar j-1, \; \; l\neq j $$
and for $l=j$ we have 
\begin{equation}
\begin{split}
h_\Lambda^{(\alpha)}(\bar i,j)  & = (i- \bar i) + \alpha(\bar j - j +1)\\
h_{\Omega^1}^{(\alpha)}(\bar i,j) & = (i- \bar i-1) + \alpha(\bar j - j +1)
\end{split}
\end{equation}
Also, for $l\neq j$,
$$h_{\Omega^2}^{(\alpha)}(\bar i,l) = h_{\Upsilon}^{(\alpha)}(\bar i,l) \quad \forall \; 1\leq l \leq \bar j-1 , $$
while for $l=j$, 
\begin{equation}
\begin{split}
h_{\Omega^2}^{(\alpha)}(\bar i,j)  & = (i- \bar i) + \alpha(\bar j - j )\\
h_{\Upsilon}^{(\alpha)}(\bar i,j) & = (i- \bar i-1) + \alpha(\bar j - j ).
\end{split}
\end{equation}
After having made basic calculations,  we see that the coefficients $b_{\Lambda, \Omega^{2}}$ and $b_{\Omega^{1},\Upsilon} $ are equal iff $\alpha=0$. We thus conclude  conclude that $b_{\Lambda, \Omega^{2}}\neq\pm b_{\Omega^{1},\Upsilon} $, which in turn implies that $c_{\Lambda, \Omega^{1}} \, b_{\Omega^{1},\Upsilon}   \pm b_{\Lambda, \Omega^{2}} \, c_{\Omega^{2},\Upsilon}\neq 0$.   

The second case,  for which  the square is located under the circle in the $\Lambda$ diagram, is very similar to the case just analyzed.  The only difference for the second case is  that  $b_{\Lambda, \Omega^{2}}= \pm b_{\Omega^{1},\Upsilon} $  and $c_{\Omega^{2},\Upsilon}\neq \pm c_{ \Lambda, \Omega^{1}}$.  Nevertheless, this implies once again that $c_{\Lambda, \Omega^{1}} \, b_{\Omega^{1},\Upsilon}   \pm b_{\Lambda, \Omega^{2}} \, c_{\Omega^{2},\Upsilon}\neq 0$.   

In conclusion, we have proved equation \eqref{proofQQ} and the proposition follows. 
\end{proof}

\begin{proof}[Proof of Theorem \ref{TheoTransla}] 
In what follows,  $P_\Lambda=P_\Lambda(x_1,\ldots,x_N; {\alpha_{k,r}})$, where $\Lambda$ is as in \eqref{eqInvAdm}.  We suppose moreover that the diagram of $\La$ contains exactly $m$ circles.  

According to Proposition \ref{decompoL+}, $P_\La$ is invariant under translation iff it belongs simultaneously to the kernel of $Q_{\Box}\circ Q_\ocircle$ and that of $ Q_\ocircle\circ Q_{\Box}$.
  
Consider first $Q_{\Box}\circ Q_\ocircle(P_\Lambda)=0$.    It is clear that $Q_{\Box}\circ Q_\ocircle(P_\Lambda)=0$ iff  $Q_\ocircle(P_\Lambda)=0$ or, according to lemma \ref{lemaQbox}, $Q_\ocircle(P_\Lambda) $  generates Jack polynomials indexed by superpartitions whose corners are all circles. On the one hand, $Q_\ocircle(P_\Lambda)=0$ iff $\Lambda$ belongs to the set $\mathcal{B}$ formed by all superpartitions satisfying conditions (i),(ii) and (iii) of Lemma \ref{lemaQcircle}.  On the other hand,   $Q_\ocircle(P_\Lambda)\neq 0$ and  $Q_{\Box}\circ Q_\ocircle(P_\Lambda)=0$   iff each corner of $\Lambda$  is a circle such that if we delete it, we  obtain a new superpartition whose corners are all circles, except possibly some that satisfy the conditions ii) or iii) of Lemma \ref{lemaQcircle}  (by assumption  not all circles of $\Lambda$ satisfy these conditions).    We call $\mathcal{C}$ the set of all such superpartitions. Now, by Lemma \ref{LemCorners}, the elimination of a circle does not create a corner with  box iff  the circle is an inner corner.  Then, $\mathcal{C}$ is given by the set of all superpartitions whose corners are all inner circles except possibly some  that satisfy the conditions ii) or iii).  It is interesting to note that the only superpartition having only circled inner corners   is the staircase  $\delta_{m}=(m-1,m-2, \ldots, 1, 0;\emptyset),$  which is $(k,r,N)$-admissible if $N\leq k$, or $N>k$ and $k\geq r-1$.   Therefore,  $Q_{\Box}\circ Q_\ocircle (P_\Lambda)=0$ iff $\Lambda$  belongs to the set $\mathcal{B} $, or the set $\mathcal{C}$.

So far, we have shown that $Q_{\Box}\circ Q_\ocircle (P_\Lambda)=0$  iff $\Lambda \in \mathcal{B} \cup \mathcal{C}$.  It remains to determine the subset  $\mathcal{A} \subset \mathcal{B} \cup \mathcal{C}$ such that  $\Lambda \in \mathcal{A} \Longrightarrow L_+(P_\Lambda)=0$.  The simplest case is $\Lambda\in \mathcal{C}$.  Indeed, since all corners of $\Lambda$ are circles,  we automatically  have $Q_{\Box} (P_\Lambda)=0$, which implies  $Q_\ocircle\circ Q_{\Box}(P_\Lambda)=0$ and $L_+(P_\Lambda)=0$.

We now suppose that $\Lambda\in \mathcal{B}$.  We want to determine the necessary and sufficient criteria for $Q_\ocircle\circ Q_{\Box} (P_\Lambda)=0$.  On the one hand, we know that $Q_{\Box} (P_\Lambda)=0$ iff   all corners of $\Lambda$ are circles. Therefore, $Q_{\Box} (P_\Lambda)=0$ and  $\Lambda\in \mathcal{B}$ iff  all corners are circles that satisfy   conditions (ii) and (iii) of Lemma \ref{lemaQcircle}.   Now, if $\Lambda\in \mathcal{B}$ and has at least one boxed corner in $(i,j)$, then  $Q_{\Box} (P_\Lambda)$ does not vanish and generates $P_\Omega$, where $\Omega$ is the superpartition obtained from $\Lambda$ by converting the box  $(i,j)$ into a circle.    Now, $Q_\ocircle(P_\Omega)$ vanishes  iff all corners of $\Omega$ satisfy any of the three conditions of Lemma \ref{lemaQcircle}.   Since by hypothesis   $\Lambda$  already complies with  these conditions,  $Q_\ocircle(P_\Omega)=0$  iff $(i,j)$ in $\Omega$ is the upper corner of the hook $C_{k,r}$ or $\tilde C_{k,r}$, or it is such that   $i=N+1-\bar k(k+1)$ and $j=\bar k(r-1)+1$ for some positive integer $\bar k$ (what is possible only once).   Applying this result to each boxed corner of $\Lambda$, we get $Q_\ocircle(Q_{\Box} (P_\Lambda))=0$ iff each boxed corner of $\Lambda$ is the upper corner of a  hook $B_{k,r}$ or $\tilde B_{k,r}$, or it is such that   $i=N+1-\bar k(k+1)$ and $j=\bar k(r-1)+1$ for some positive integer $\bar k$.

Finally, let $(\ell,j')$ the coordinates of the last corner $\Lambda\in \mathcal{B}$.  Obviously, if there is a circle in $(\ell,j')$, this circle also corresponds to the last corner of any superpartition $\Omega$ indexing the Jack polynomials generated by $Q_{\Box} (P_\Lambda)$.  According to Corollary \ref{ultimaesquina},  we know that $ Q_\ocircle\circ Q_{\Box} (P_\Lambda)=0$ only if $\ell=N-k$ and $j=r$. On the other hand, if the last corner $\Lambda$ is a box, it is known that $Q_{\Box} (P_\Lambda)$  generates a $P_\Omega$ such that the last corner of $\Om$	is a circle,  so we have once again that  $ Q_\ocircle\circ Q_{\Box} (P_\Lambda)=0$ only if $\ell=N-k$ and $j=r$.  

In summary,     $Q_{\Box}\circ Q_\ocircle(P_\Lambda)=0$ and  $ Q_\ocircle\circ Q_{\Box} (P_\Lambda)=0$ iff: 1) all corners of $\Lambda$ are circles, which are inner corners, except possibly for some circles that satisfy the conditions (ii) and (iii) of Lemma  \ref{lemaQcircle};  or 2) the last corner of $\Lambda$ is located in $(N-k,r)$ and all other corners of $\Lambda$ are the upper corners of hooks type $B_{k,r}$, $\tilde B_{k,r}$, $C_{k,r}$ or $\tilde C_{k,r}$. 
\end{proof}

\subsection{Special cases of invariance}

The previous theorem clearly shows that for $n$, $m$, $k$, $r$, and $N$, the number of  ways to construct  superpartitions that lead to invariant polynomials could be enormous. In general such superpartitions do not have a explicit and compact form.  There are two  notable exceptions however:  (1) when we are dealing with conventional partitions (no circle in the diagrams), and (2) when the maximal length $N$ of the superpartition is limited as $N\leq 2k$.  The first case was studied by Jolicoeur and Luque \cite{jl} .  Below, we rederive very simply  one of their results. For the second case, we  identify  three simple forms of superpartitions associated with invariant  polynomials.  
 
\begin{coro} \label{lempartitioninv}  Let $P_\la=P_{\lambda}(x_1,\ldots,x_N; \alpha_{k,r})$, where $\lambda$ is a $(k,r,N)$-admissible  partition .   
The polynomial $P_{\lambda} $ is invariant under translation if and only if $$\la=\big(((\beta+1)r)^{l}, (\beta r)^{k}, \ldots, r^{k}\big),$$
where $0<\beta$, $ 0\leq l \leq k$,   and  $N=k(\beta+1) + l$. 
\end{coro}
\begin{proof}
As a consequence of Theorem \ref{TheoTransla}, we have that $P_{\lambda}$ is invariant under translation iff the last corner of $\lambda$'s diagram is located at position $(N-k,r)$ and all remaining corners are upper corners of  hooks $B_{k,r}$. Thus, $P_{\lambda}$ is invariant iff $\lambda=(((\beta+1)r)^{l}, (\beta r)^{k}, \ldots, r^{k})$ with $0<\beta$. The admissibility condition requires $0\leq l\leq k$.  Finally,  the condition on the position for the last corner imposes $N=k(\beta+1) + l$.
\end{proof}

\begin{coro}   Assume   \eqref{eqInvJack}, \eqref{eqInvAlpha}, and \eqref{eqInvAdm}.    Suppose moreover that $\Lambda$'s diagram contains  $m$ circles and that $N\leq 2k$. Then, $P_\La$ is invariant under translation if and on if $\Lambda$ has one of the following forms:
\begin{itemize}
 \item[(F1)] $\Lambda= (\emptyset; r^{N-k})$;
 \item[(F2)] $\Lambda= (m-1, m-2, \ldots, 1, 0; \emptyset)$, where  $m\leq N\leq k$ or $N-1\geq k\geq N-m+r-1$ ;
 \item[(F3)] $\Lambda= (r+f-1, r+f-2, \ldots, r-1, g-1, g-2, \ldots, 1, 0; r^{N-k-m})$ \; where $m=f+g+1$, $0\leq f \leq N-k-1$, \; $0\leq g \leq min(k,r-1)$ and  $f\geq g+N-2k-1$.
\end{itemize}
These forms are respectively illustrated in Figures \ref{FigFormF1x}, \ref{FigFormF2}, \ref{FigFormF3} below.
\end{coro}
\begin{proof}
Let us start with the sufficient condition.  According to Theorem \ref{TheoTransla}, if $\La$ is of the form (F1), (F2) or (F3), then $P_\La$ is invariant under translation.  Indeed, 
(F1) trivially satisfies (D1); the only corners in (F2) are inner circles, so (F2) satisfies (D2); in (F3), all corners are inner circles, except one circle located at $(N-k,r)$, so it satisfies  (D2) with $\bar k=1$.

We now tackle the non-trivial part of the demonstration, which is the necessary condition.  For this, let $(\ell,j)$ be the last corner of the $\Lambda$ diagram. There are two obvious cases, depending on whether $(\ell,j)$  is  an inner corner or not.

First, we suppose  that $(\ell,j)$ is a  bordering corner or an outer corner. According to Theorem \ref{TheoTransla}, $P_{\Lambda}$ is invariant under translation only if $N+1-\ell+\alpha_{k,r}(j-1)=0$, where $\alpha_{k,r}=-(k+1)/(r-1)$.  Since $N+1-\ell>0$, we must assume that $j-1=\bar j(r-1)$, where $\bar j$ is a positive integer. Then, the invariance condition requires $N=\ell+\bar j (k+1)-1$.  However, by hypothesis, $N\leq 2k$, so $\bar j =1$ (i.e., $j=r$).  Therefore, the invariance condition and $N\leq 2k$ impose $j=r$ and $\ell=N-k\leq k$, which is compatible with the admissibility.   Now, let $(i,\ell')$ be  the first corner of $\Lambda$ diagram.  Once again, two cases are possible:

\begin{enumerate} 
\item  $(i,\ell')$ is a box.  Suppose $(i,\ell')\neq (\ell,j)$. According to Theorem \ref{TheoTransla},  $P_{\Lambda} $ can be invariant only if we can form a hook $B_{k,r}$ or $\tilde B_{k,r}$ whose respective lengths are either $k+1$ or $k+2$, which is impossible because $\ell\leq k$.  Then, the only possible squared corner is the last corner.  Thus, the invariance and admissibility conditions  impose that the diagram is made of $N-k$ rows with $r$ boxes, corresponding to the first form of the proposition.

 \item  $(i,\ell')$ is a circle.  Referring again to Theorem \ref{TheoTransla} and recalling that  $\ell\leq k$, we see that  $P_{\Lambda}$ is  invariant under translation only if $(i,\ell')=(\ell,j)$ or if  $(i,\ell')$ is a inner circled corner.  The first condition imposes $\La=(r-1;r^{N-k-1})$.  The second imposes that only criterion (D2) can be considered, so all remaining  corners must be circled inner corners.    Consequently, $\Lambda= (r+m-2, r+m-3, \ldots, r, r-1; r^{N-k-m})$ for some $1\leq m\leq N-k$.  This is illustrated in Figure \ref{FigFormF3special}

\end{enumerate}

\begin{figure}[h]
\caption{Form (F1)}\label{FigFormF1x}
\begin{center}
\setlength{\unitlength}{4.5pt}
\vspace{0.5cm}
\begin{picture}(14,16)(0,0)
\put(6,7.8){{$$\tableau[scY]{    &&\bl \ldots    & & \\ \bl \vdots\\ \bl \bl \\ &&\bl \ldots    & & \\   &&\bl \ldots    & & } $$} }

          \put(5.3,14.9){ {$\overbrace{\phantom{xxxxxxxxxx}}^{r} $} }
  \put(-2.5,7.8){{$ N-k
    \left\{\rule[16pt]{0pt}{19pt}\right.$}}  
\end{picture} \phantom{xxxxxxx}
\end{center}
\end{figure}
\vspace{0.0cm}

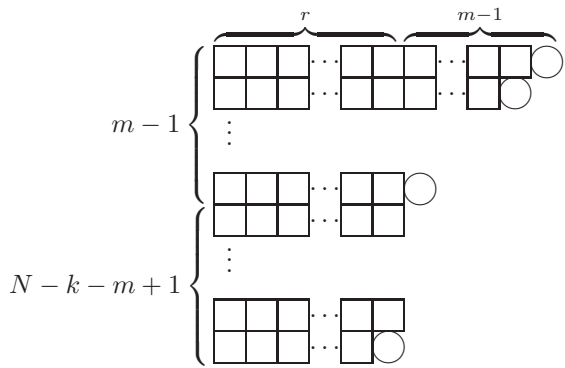
\begin{figure}[h]
\caption{Form (F3) with $g=0$}
\label{FigFormF3special}
\begin{center}
\setlength{\unitlength}{4.3pt}
\vspace{0.5cm}
\begin{picture}(14,26)(0,0)
\put(6,10.5){{$$\tableau[scY]{ &&&\bl \ldots&& &&\bl \ldots &&&\bl \tcercle{} \\   & &&\bl \ldots&& &&\bl \ldots &&\bl \tcercle{} \\\bl \vdots\\ \bl \bl \\ &   &&\bl \ldots&&&\bl \tcercle{}\\ &&&\bl \ldots&& \\\bl \vdots\\ \bl \bl \\    & &&\bl \ldots& & \\  & &&\bl \ldots&&\bl \tcercle{}  } $$} }

          \put(5.3,24.2){ {$\overbrace{\phantom{xxxxxxxxxxxx}}^{r} $} }
					\put(22,24.2){ {$\overbrace{\phantom{xxxxxxxxxx}}^{m-1} $} }
  \put(-12,3.3){{$ N-k-m+1
    \left\{\rule[17pt]{0pt}{17pt}\right.$}}  
		\put(-3,17.5){{$ m-1
    \left\{\rule[17.5pt]{0pt}{17pt}\right.$}}  
\end{picture}
\phantom{xxxxxxxxxxxx}
\end{center}
\end{figure}
\vspace{0.5cm}

Second, we suppose  that  $(\ell,j)$ is an inner corner.  This implies that $j=1$ and as a consequence, criterion (D1) of Theorem \ref{TheoTransla} cannot be satisfied.  Thus, the only option is that the last corner is a circle and criterion (D2) must be satisfied:   all other corners must be inner circles, except for at most one corner, which can be a bordering or outer circle, located  at $(\bar \i,\bar \j)$,  and such that $\bar \i =N+1-\bar k(k+1)$ and $\bar \j=\bar k(r-1)+1$ for some positive integer $\bar k$.  However, we know that $\bar \i<\ell\leq 2k$, so that $\bar k=1$.  In short, if $(\ell,j)$ is an inner corner, then all corners are inner circles, except for at most one non-inner corner, which could be a circle located at $(N-k,r)$.  If all corners are inner ones, without exception, then the only possible superpartition is $$\La=(m-1,m-2,\ldots,1,0;\emptyset),\quad m\leq N, $$ which is the form (F2) illustrated in Figure~\ref{FigFormF2}.   Finally, if there is one exceptional corner, the all possibles superpartitions can be written as $$\La=(r+f-1,r+f-2,\ldots,r, r-1, g-1,g-2,\ldots,0 ;r^{N-k-f}),$$where $$f+g+1=m,\quad g<r,\quad g\leq k,\quad f<N-k.$$     This is the last possible form and it is illustrated in Figure~\ref{FigFormF3}.  Note that the admissibility imposes some additional  restrictions on the forms (F2) and (F3).  The form (F2) is admissible whenever $N\leq k$, while for  $N>k$, it is admissible if $N+r-m-1\leq k$.  
In the case of (F3), the admissibility also requires $f\geq g+N-2k-1$.

\begin{figure}[h]
\caption{Form (F2)}
\label{FigFormF2}
\begin{center}
\setlength{\unitlength}{4.0pt}
\vspace{0.5cm}
\begin{picture}(14,14)(0,0)
\put(6,7.8){{$$\tableau[scY]{    &&\bl \ldots    & &&\bl \tcercle{} \\ &&\bl \ldots    & &\bl \tcercle{} \\ \bl \vdots\\ \bl \bl \\&\bl \tcercle{}\\\bl \tcercle{} } $$} }

  \put(1.0,7.8){{$ m
    \left\{\rule[23pt]{0pt}{19pt}\right.$}}  
\end{picture} \phantom{xxxxxxxxx}
\end{center}
\end{figure}

\begin{figure}[h]\caption{Form (F3)}\label{FigFormF3}
\begin{center}
{\footnotesize
\setlength{\unitlength}{3.8pt}
\vspace{0.5cm}
\begin{picture}(16,36)(-4,-8)
\put(6,7.8){{$$\tableau[scY]{ &&\bl \ldots    & &&\bl \ldots&& &&\bl \ldots &&&\bl \tcercle{} \\ &&\bl \ldots    & &&\bl \ldots&& &&\bl \ldots &&\bl \tcercle{} \\\bl \vdots\\ \bl \bl \\&&\bl \ldots    & &&\bl \ldots&&&\bl \tcercle{}\\&&\bl \ldots    & &&\bl \ldots&& \\\bl \vdots\\ \bl \bl \\ &&\bl \ldots    & &&\bl \ldots& & \\  &&\bl \ldots    & &&\bl \ldots&&\bl \tcercle{} \\ &&\bl \ldots    & &\bl \tcercle{} \\ \bl \vdots\\ \bl \bl \\&\bl \tcercle{}\\\bl \tcercle{} } $$} }

          \put(5.4,29){ {$\overbrace{\phantom{xxxxxxxxxxxxxxxx}}^{r} $} }
					\put(29,29){ {$\overbrace{\phantom{xxxxxxxxxx}}^{f} $} }
  \put(-3,14.8){{$ N-k
    \left\{\rule[49pt]{0pt}{5pt}\right.$}}  
		
	\put(2.2,-6.4){{$ g
    \left\{\rule[18pt]{0pt}{6pt}\right.$}}  
\end{picture}
\phantom{xxxxxxxxxxxxxxxxxxx}
}
\end{center}
\end{figure}
 \vspace{0cm}

 We have demonstrated that only three forms of admissible superpartitions lead to invariant polynomials when $N\leq 2k$.\end{proof}

\subsection{The clustering condition for $k>1$}

Baratta and Forrester have shown that if symmetric Jack polynomials are also invariant under translation, then they almost automatically admit clusters \cite{BarFor}.   In what follows, we generalize their approach   to the case of Jack polynomials with prescribed symmetry.

\begin{proposition} Let $P_{\Lambda}(x_1,\ldots,x_N;\alpha_{k,r})$ be a Jack polynomial with prescribed symmetry AS, where $\Lambda$ is as in \eqref{eqInvAdm} and of bi-degree $(n|m)$ and such that  $N\geq k+m+1$. Suppose moreover that $\Lambda$ is such that $P_{\Lambda}(x_1,\ldots,x_N;\alpha_{k,r})$ is translationally invariant. 
\begin{itemize}
	\item [(i)] If $\ell(\Lambda)>N-k$ then 
\begin{equation*}
P_{\Lambda}(x_1,\ldots,x_N;\alpha_{k,r})\Big|_{x_{N-k+1}=\ldots=x_{N}=z}=0\, .
\end{equation*}
  \item [(ii)] If $\ell(\Lambda)=N-k$ then
\begin{equation*}
P_{\Lambda}(x_1,\ldots,x_N;\alpha_{k,r})\Big|_{x_{N-k+1}=\ldots=x_{N}=z}=\prod_{i=m+1}^{N-k} (x_{i}-z)^{r}Q(x_{1},\ldots,x_{N-k},z)
\end{equation*}
for some polynomial $Q$ of degree $n-(N-k-m)r$.
\end{itemize}
\end{proposition}
\begin{proof}
From the admissibility condition, we know that $P_{\Lambda}(x;\alpha_{k,r})$ is well defined.  Moreover, the condition   $N\geq k+m+1$ ensures that the specialization of the $k$ variables takes place in the set of variables in  which $P_{\Lambda}$  is symmetric. In other words, if $\alpha$ is not a negative rational nor zero, then 
$$   P_{\Lambda}(x;\alpha)\Big|_{x_{N-k+1}=\ldots=x_{N}=z}\neq 0\, .$$
Thus, property (i) is not trivial.  However, if we suppose that $P_{\Lambda}(x;\alpha_{k,r})$ is translationally invariant, then    
\begin{equation}\label{eqtranslation}  
P_{\Lambda}(x_{1},\ldots,x_{N};\alpha_{k,r})\Big|_{x_{N-k+1}=\ldots=x_{N}=z} =P_{\Lambda}(x_{1}-z,\ldots,x_{N-k}-z,0,\ldots,0;\alpha_{k,r})
\end{equation}
Now, by the stability property given in Lemma \ref{stability1}, the last equality can rewritten   as 
\begin{equation}\label{eqtranslation2}  
P_{\Lambda}(x_{1},\ldots,x_{N};\alpha_{k,r})\Big|_{x_{N-k+1}=\ldots=x_{N}=z} =P_{\Lambda}(x_{1}-z,\ldots,x_{N-k}-z;\alpha_{k,r}).
\end{equation}
From this point, two cases are possible:
\begin{itemize}
	\item [(i)] If $\ell(\Lambda)>N-k$, Lemma \ref{stability1} also implies  that the RHS of \eqref{eqtranslation2} is zero, as expected. 
  \item [(ii)] If $\ell(\Lambda)=N-k$, then the RHS of \eqref{eqtranslation2} is not zero.  From the triangularity property of the Jack polynomials with prescribed symmetry in the monomial basis, we can write 
\begin{multline*}
\qquad\qquad P_{\Lambda}(x_{1}-z,\ldots,x_{N-k}-z;\alpha_{k,r})\\=m_{\Lambda}(x_{1}-z,\ldots,x_{N-k}-z)+\sum_{\Gamma<\Lambda}c_{\Lambda,\Gamma}m_{\Gamma}(x_{1}-z,\ldots,x_{N-k}-z).
\end{multline*}
Moreover, according to Theorem  \ref{TheoTransla} and Lemma \ref{ultimaesquina},  the last corner in $\La's$ diagram is located at $(N-k,r)$. This fact, together with $\ell(\Gamma)=N-k$ and $N\geq k+m+1$,   impose that $$\La_{N-k}\geq r\quad \text{and}\quad \Gamma_{N-k}\geq r \quad\text{for all}\quad\Gamma< \Lambda.$$ Hence, $\prod_{i=m+1}^{N-k} (x_{i}-z)^{r}$ divides $ m_{\Gamma}$ and each $m_{\Gamma}$ such that $\Ga<\La$.  This finally implies that $\prod_{i=m+1}^{N-k} (x_{i}-z)^{r}$ divides $P_{\Lambda}(x_{1}-z,\ldots,x_{N-k}-z;\alpha_{k,r})$, 
and the proposition follows. 
\end{itemize}
\end{proof}

The last proposition establishes the clustering properties conjectured in \cite{dlm_cmp2} in the case of translationally invariant polynomials. The next proposition shows that in this case, it is also possible to get more explicit clustering properties  involving only Jack polynomials and not some an indeterminate polynomials $Q$ as before. Note that in some instances, we only form cluster of order $r-1$.  We stress that this is not in contradiction with the previous proposition.  Indeed, more variables could be collected to get order $r$, but  this factorization would not allow us to write explicit formulas in terms of Jack polynomials with prescribed symmetry.

\begin{proposition}\label{propexplicit}  Let $P_{\Lambda}(x_1,\ldots,x_N)$ be a Jack polynomial with prescribed symmetry AS at $\alpha=\alpha_{k,r}$, where  $\Lambda=(\La_1,\ldots,\La_m;\La_{m+1},\ldots,\La_N)$ is as in \eqref{eqInvAdm} and of length $\ell\leq N$.  Suppose that the partition $(\La_{m+1},\ldots,\La_N)$ contains  $f_{0}$ parts equal to $0$. Suppose moreover that $\Lambda$ is such that  $\La_{N-f_0}=r$ and $P_\Lambda(x_,\ldots,x_N)$ is translationally invariant.\\

\noindent (i) If   $\La_m\geq r$ or $m=0$, then 
$$P_\Lambda(x_,\ldots,x_N)\Big|_{x_{N-f_{0}+1}=\ldots=x_{N}=z}=\prod_{i=1}^{N-f_{0}} (x_{i}-z)^{r}\cdot P_{\Lambda- r^\ell}(x_{1}-z,\ldots,x_{N-f_{0}}-z).$$

\noindent (ii) If $\La_{m} = r-1$, then 
$$P_{\Lambda}(x_{1},\ldots,x_{N})\Big|_{x_{N-f_{0}+1}=\ldots=x_{N}=z}=\prod_{i=1}^{N-f_{0}} (x_{i}-z)^{r-1}\cdot P_{\Lambda-(r-1)^\ell}(x_{1}-z,\ldots,x_{N-f_{0}}-z). $$

\noindent (iii) If $\La_{m}=0$, then 
\begin{multline*} \qquad \quad
P_{\Lambda}(x_{1},\ldots,x_{N})\Big|_{x_{m}=x_{N-f_{0}+1}=\ldots=x_{N}=z}\\
=  \prod_{\substack{1\leq i\leq N-f_0\\i\neq m}}  (x_{i}-z)^{v}\cdot P_{\widetilde{\Lambda}}(x_{1}-z,\ldots,x_{m-1}-z,x_{m+1}-z,\ldots,x_{N-f_{0}}-z)
\end{multline*}
where $v=\min(r,\La_{m-1})$, $\widetilde{\Lambda}=\widetilde{C}\Lambda-v^{(\ell-1)}$, and $\widetilde{C}\La=(\La_1,\ldots,\La_{m-1};\La_{m+1},\ldots,\La_N)$ .
\end{proposition}
\begin{proof}

 Proceeding as in the proof of the previous proposition, we use the translation invariance and the stability of the Jack polynomials with prescribed symmetry, and find
\begin{equation}\label{eqtranslation4}  
P_{\Lambda}(x_{1},\ldots,x_{N})\Big|_{x_{N-f_{0}+1}=\ldots=x_{N}=z} =P_{\Lambda}(x_{1}-z,\ldots,x_{N-f_{0}}-z).
\end{equation}

 \noindent (i) If $\lambda_{N-f_{0}}=r$ and $m=0$ or  $m>0$ and $\La_{m} \geq r$, then we can decompose the superpartition $\La$ as $$\La=\widetilde{\La}+ r^\ell,$$ where $\tilde \La$ is some other superpartition, which could be empty, and $r^\ell$ denotes the partition $(r,\ldots,r)$ of length $\ell$.  This allows us to use Lemma \ref{simpleprod} and factorize the RHS of \eqref{eqtranslation4}.  This yields, as expected, 
\begin{equation*}
P_{\Lambda}(x_{1},\ldots,x_{N})\Big|_{x_{N-f_{0}+1}=\ldots=x_{N}=z} =\prod_{i=1}^{N-f_{0}} (x_{i}-z)^{r}\cdot P_{\widetilde{\Lambda}}(x_{1}-z,\ldots,x_{N-f_{0}}-z).
\end{equation*}

\noindent (ii) If $\Lambda_{N-f_{0}}=r$ and $\La_{m} = r-1$, then $\La$ can be decomposed as   $$\La=\widetilde{\La}+ (r-1)^\ell,$$
where, this time, $\widetilde{\La}$ is a non-empty superpartition of length $\ell$ and such that $\tilde \La_m=0$.   Using once again Lemma \ref{simpleprod}, we can factorize RHS of \eqref{eqtranslation4} and get the desired result:
\begin{equation*}
P_{\Lambda}(x_{1},\ldots,x_{N})\Big|_{x_{N-f_{0}+1}=\ldots=x_{N}=z} =\prod_{i=1}^{N-f_{0}} (x_{i}-z)^{r-1}\cdot P^{(\alpha)}_{\widetilde{\Lambda}}(x_{1}-z,\ldots,x_{N-f_{0}}-z).
\end{equation*}

\noindent (iii) Finally, we suppose  $\Lambda_{N-f_{0}}=r$, $\La_{m}=0$, and $v=\min(r,\La_{m-1})$.  In equation \eqref{eqtranslation4}, we set $x_m=z$.  This yields 
$$ P_{\Lambda}(x_{1},\ldots,x_{N})\Big|_{x_{m}=x_{N-f_{0}+1}=\ldots=x_{N}=z}=P_{\Lambda}(x_{1}-z,\ldots,x_{m-1}-z,0,x_{m+1}-z,\ldots,x_{N-f_{0}}-z).
$$ 
According to Lemma \ref{removecircle}, the RHS of the last equation can be simplify as follows 
\begin{equation}\label{eqtranslation5}    P_{\Lambda}(x_{1},\ldots,x_{N})\Big|_{x_{m}=x_{N-f_{0}+1}=\ldots=x_{N}=z}=P_{\widetilde{C}\Lambda}(x_{1}-z,\ldots,x_{m-1}-z,x_{m+1}-z,\ldots,x_{N-f_{0}}-z).\end{equation}
 Now,  we can decompose $\widetilde{C}\La$  as 
    $$  \widetilde{C}\La=\widetilde{\La}+ v^{\ell-1},$$ 
for some superpartition $\widetilde{\La}$ whose length is smaller or equal to $\ell-1$ . 
This allows us to exploit 
   Lemma \ref{simpleprod} and rewrite  the RHS of \eqref{eqtranslation5} as    
\begin{equation*}
\prod_{i=1}^{m-1} (x_{i}-z)^{v}\cdot \prod_{i=m+1}^{N-f_{0}} (x_{i}-z)^{v}\cdot  P^{(\alpha)}_{\widetilde{\Lambda}}(x_{1}-z,\ldots,x_{m-1}-z,x_{m+1}-z,\ldots,x_{N-f_{0}}-z),
\end{equation*}
which is the desired result.
\end{proof}
\vspace{1ex}

Let us consider a non-trivial example in relation with the last proposition.  We choose $k=2$, $r=3$ and $N=8$. Let $\Lambda=(8,7,5;6,3,3)$, i.e.
{\small{$$\Lambda=\tableau[scY]{  &&&&&&&&\bl\tcercle{}\\&&&&&&&\bl\tcercle{}\\&&&&&\\&&&&&\bl\tcercle{}\\&& \\&&}$$}}
Clearly $P_{\Lambda}(x;{-3/2})$ is translationally invariant.    Proposition \ref{propexplicit} then yields  
\begin{equation*}
P_{\Lambda}(x_{1},\ldots,x_{8};{-3/2})\Big|_{x_{7}=x_{8}=z}=\prod_{i=4}^{6} (x_{i}-z)^{3}P^{(-3/2)}_{\widetilde{\Lambda}}(x_{1}-z,\ldots,x_{6}-z)
\end{equation*}
where $\widetilde{\Lambda}=(5,4,2;3)$, i.e., 
{\small{$$\widetilde{\Lambda}=\tableau[scY]{  &&&&&\bl\tcercle{}\\&&&&\bl\tcercle{}\\&&\\&&\bl\tcercle{}}$$}}
Moreover,  $P_{\widetilde{\Lambda}}(x;-3/2)$ is also translationally invariant in $\widetilde{N}=N-k=6$ variables, so that
\begin{equation*}
P^{}_{\widetilde{\Lambda}}(x_{1}-z,\ldots,x_{6}-z; -3/2)=P^{}_{\widetilde{\Lambda}}(x_{1},\ldots,x_{6}; -3/2).
\end{equation*}
Therefore,
\begin{equation*}
P_{\Lambda}^{}(x_{1},\ldots,x_{8};-3/2)\Big|_{x_{7}=x_{8}=z}=\prod_{i=4}^{6} (x_{i}-z)^{3}P^{}_{\widetilde{\Lambda}}(x_{1},\ldots,x_{6};-3/2)
\end{equation*}

The last example is very special because it involves a pair of superpartitions satisfying the following bi-invariance property: $\La$ and $\tilde \La=\La-r^\ell$  are such that both $P_{\Lambda}{(x_1,\ldots,x_N;\alpha_{k,r})}$  and  $P_{\tilde\Lambda}{(x_1,\ldots,x_{N-k};\alpha_{k,r})}$  are invariant   under translation.   By using Theorem \ref{TheoTransla}, one can check that the diagrams given  below define a large family  of pairs of superpartitions satisfying this bi-invariance property.  

\vspace{6ex}

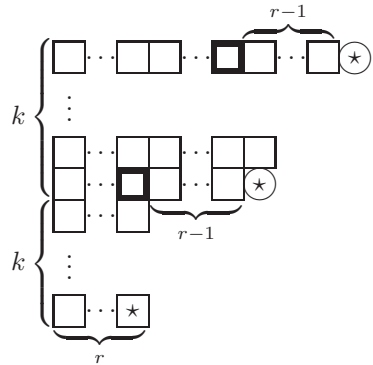
\setlength{\unitlength}{4.3pt}   

\begin{center}
\begin{picture}(19,20)(0,0)
\put(3,5){{$$\tableau[scY]{   &\bl \ldots & && \bl \ldots& \tf& &\bl \ldots& &\bl \tcercle{$\star$}\\ \bl \vspace{-2ex} \vdots \\  \bl \bl \\&\bl \ldots & & & \bl \ldots& & \\  & \bl \ldots & \tf&& \bl \ldots& &\bl \tcercle{$\star$}  \\ &\bl \ldots & \\ \bl \vspace{-2ex} \vdots \\  \bl \bl  \\ &\bl \ldots  & \star }$$} }
         \put(18.9,17.2){ {$\overbrace{\phantom{xxxxxx}}^{r-1} $} }
				\put(10.8,3.9){ {$\underbrace{\phantom{xxxxxx}}_{r-1} $} }
				\put(2.3,-7.1){ {$\underbrace{\phantom{xxxxxx}}_{r} $} }
  \put(-0.7,10.9){{$ k
    \left\{\rule[19.9pt]{0pt}{12.4pt}\right.$}}  
	\put(-0.7,-2){{$ k
    \left\{\rule[19.9pt]{0pt}{9pt}\right.$}}  
\end{picture}
\end{center}

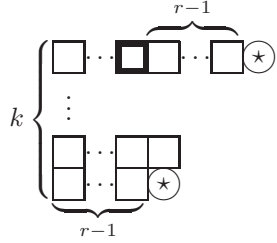
\vspace{1.5cm}

\begin{center}
\begin{picture}(14,17)(0,-2)
\put(3,5){{$$ \tableau[scY]{    & \bl \ldots &\tf& & \bl \ldots & &\bl \tcercle{$\star$}\\ \bl \vspace{-2ex} \vdots \\ \bl \bl \\&  \bl \ldots& & \\  &\bl \ldots & &\bl \tcercle{$\star$}   }$$ } }
        \put(11.3,11.9){ {$\overbrace{\phantom{xxxxxx}}^{r-1} $} }
        \put(3,-1.5){ {$\underbrace{\phantom{xxxxxx}}_{r-1} $} }
  \put(0,5.1){{$ k
    \left\{\rule[19.9pt]{0pt}{13pt}\right.$}}  
\end{picture}
\end{center}

  \begin{appendix}

\section{Examples of admissible and invariant superpartitions}

In this appendix, for the  triplet $(k,r,N)$ given below,  we display all  smallest possible $(k,r,N)$-admissible superpartitions  that lead to  Jack polynomials with prescribed symmetry AS that are translationally invariant and, as a consequence,  admit clusters of size $k$ and order $r$.   The word ``smallest'' refers to the least number of boxes in the corresponding diagrams. 

Let $(k,r,N)=(4,3,15)$.  Suppose first that the number $m$ of circle is zero.  Then, according to Corollary \ref{lempartitioninv}, the smallest possible partition   that is $(k,r,N)-$admissible and indexes an invariant polynomial  is $\lambda=(9^3,6^4,3^4)$.    For higher values of $m$,  one obtains the smallest superpartitions by deleting some squared corners in $\lambda$  and adding circles while keeping conditions C1 and C2 satisfied.   All smallest superpartitions for $(k,r,N)=(4,3,15)$ are given below.  
\vspace{.5cm}

{\small

\setlength{\unitlength}{4.2pt}

\hspace{.5cm}\begin{picture}(14,18)(-6,-2)

\put(-10,5){{$$\tiny\tableau[scY]{   &&&&&&&&\\&&&&&&&&\\&&&&&&&&\\&&&&&\\&&&&&\\&&&&&\\&&&&&\\&&\\&&\\&&\\&&\\ }$$} }
     
\put(20,5){{$$\tiny\tableau[scY]{  &&&&&&&&\\&&&&&&&&\\&&&&&&&\\&&&&&\\&&&&&\\&&&&&\\&&&&\\&&\\&&\\&&\\&&\bl\tcercle{}\\ }$$} }
     
\put(50,5){{$$\tiny\tableau[scY]{   &&&&&&&&\\&&&&&&&&\\&&&&&&&\\&&&&&\\&&&&&\\&&&&&\\&&&&&\bl\tcercle{}\\&&\\&&\\&&\\&&\\ }$$} }

\put(80,5) {{$$\tiny\tableau[scY]{  &&&&&&&&\\&&&&&&&&\\&&&&&&&&\bl\tcercle{}\\&&&&&\\&&&&&\\&&&&&\\&&&&&\\&&\\&&\\&&\\&&\\ }$$} }  

\end{picture}

\vspace{1 cm}

\hspace{.5cm}\begin{picture}(14,18)(-6,-2)

\put(-10,5){{$$\tiny\tableau[scY]{   &&&&&&&&\\&&&&&&&&\\&&&&&&\\&&&&&\\&&&&&\\&&&&&\\&&&&\bl\tcercle{}\\&&\\&&\\&&\\&&\bl\tcercle{}\\ }$$} }

\put(20,5){{$$\tiny\tableau[scY]{   &&&&&&&&\\&&&&&&&\\&&&&&&&\\&&&&&\\&&&&&\\&&&&&\bl\tcercle{}\\&&&&\\&&\\&&\\&&\\&&\bl\tcercle{}\\ }$$} }

\put(50,5){{$$\tiny\tableau[scY]{   &&&&&&&&\\&&&&&&&&\\&&&&&&&\bl\tcercle{}\\&&&&&\\&&&&&\\&&&&&\\&&&&\\&&\\&&\\&&\\&&\bl\tcercle{}\\ }$$} }

\put(80,5){{$$\tiny\tableau[scY]{  &&&&&&&&\\&&&&&&&&\bl\tcercle{}\\&&&&&&&\\&&&&&\\&&&&&\\&&&&&\\&&&&\\&&\\&&\\&&\\&&\bl\tcercle{}\\ }$$} }
\end{picture}

\vspace{1 cm}

\hspace{.5cm}\begin{picture}(14,18)(-6,-2)

\put(-10,5){{$$\tiny\tableau[scY]{   &&&&&&&&\\&&&&&&&&\\&&&&&&&\bl\tcercle{}\\&&&&&\\&&&&&\\&&&&&\\&&&&&\bl\tcercle{}\\&&\\&&\\&&\\&&\\ }$$} }

\put(20,5){{$$\tiny\tableau[scY]{  &&&&&&&&\\&&&&&&&&\bl\tcercle{}\\&&&&&&&\\&&&&&\\&&&&&\\&&&&&\\&&&&&\bl\tcercle{}\\&&\\&&\\&&\\&&\\ }$$} }

\put(50,5){{$$\tiny\tableau[scY]{  &&&&&&&&\\&&&&&&&\\&&&&&&\\&&&&&\\&&&&&\\&&&&&\bl\tcercle{}\\&&&&\bl\tcercle{}\\&&\\&&\\&&\\&&\bl\tcercle{}\\ }$$} }

\put(80,5){{$$\tiny\tableau[scY]{  &&&&&&&&\\&&&&&&&&\\&&&&&&\bl\tcercle{}\\&&&&&\\&&&&&\\&&&&&\\&&&&\bl\tcercle{}\\&&\\&&\\&&\\&&\bl\tcercle{}\\ }$$} }

\end{picture}

\vspace{1 cm}

\hspace{.5cm}\begin{picture}(14,18)(-6,-2)

\put(-10,5){{$$\tiny\tableau[scY]{  &&&&&&&&\\&&&&&&&&\bl\tcercle{}\\&&&&&&\\&&&&&\\&&&&&\\&&&&&\\&&&&\bl\tcercle{}\\&&\\&&\\&&\\&&\bl\tcercle{}\\ }$$} }

\put(20,5){{$$\tiny\tableau[scY]{  &&&&&&&&\\&&&&&&&\\&&&&&&&\bl\tcercle{}\\&&&&&\\&&&&&\\&&&&&\bl\tcercle{}\\&&&&\\&&\\&&\\&&\\&&\bl\tcercle{}\\ }$$} }

\put(50,5){{$$\tiny\tableau[scY]{  &&&&&&&&\bl\tcercle{}\\&&&&&&&\\&&&&&&&\\&&&&&\\&&&&&\\&&&&&\bl\tcercle{}\\&&&&\\&&\\&&\\&&\\&&\bl\tcercle{}\\ }$$} }

\put(80,5){{$$\tiny\tableau[scY]{  &&&&&&&&\\&&&&&&&&\bl\tcercle{}\\&&&&&&&\bl\tcercle{}\\&&&&&\\&&&&&\\&&&&&\\&&&&\\&&\\&&\\&&\\&&\bl\tcercle{}\\ }$$} }

\end{picture}

\vspace{1 cm}

\hspace{.5cm}\begin{picture}(14,18)(-6,-2)

\put(-10,5){{$$\tiny\tableau[scY]{   &&&&&&&&\\&&&&&&&&\bl\tcercle{}\\&&&&&&&\bl\tcercle{}\\&&&&&\\&&&&&\\&&&&&\\&&&&&\bl\tcercle{}\\&&\\&&\\&&\\&&\\ }$$} }

\put(20,5){{$$\tiny\tableau[scY]{  &&&&&&&&\\&&&&&&&\\&&&&&&\bl\tcercle{}\\&&&&&\\&&&&&\\&&&&&\bl\tcercle{}\\&&&&\bl\tcercle{}\\&&\\&&\\&&\\&&\bl\tcercle{}\\ }$$} }

\put(50,5){{$$\tiny\tableau[scY]{   &&&&&&&&\\&&&&&&&\bl\tcercle{}\\&&&&&&\\&&&&&\\&&&&&\\&&&&&\bl\tcercle{}\\&&&&\bl\tcercle{}\\&&\\&&\\&&\\&&\bl\tcercle{}\\ }$$} }

\put(80,5){{$$\tiny\tableau[scY]{   &&&&&&&&\bl\tcercle{}\\&&&&&&&\\&&&&&&\\&&&&&\\&&&&&\\&&&&&\bl\tcercle{}\\&&&&\bl\tcercle{}\\&&\\&&\\&&\\&&\bl\tcercle{}\\ }$$} }

\end{picture}

\vspace{1 cm}

\hspace{.5cm}\begin{picture}(14,18)(-6,-2)

\put(-10,5){{$$\tiny\tableau[scY]{  &&&&&&&&\\&&&&&&&&\bl\tcercle{}\\&&&&&&\bl\tcercle{}\\&&&&&\\&&&&&\\&&&&&\\&&&&\bl\tcercle{}\\&&\\&&\\&&\\&&\bl\tcercle{}\\ }$$} }

\put(20,5){{$$\tiny\tableau[scY]{  &&&&&&&&\bl\tcercle{}\\&&&&&&&\\&&&&&&&\bl\tcercle{}\\&&&&&\\&&&&&\\&&&&&\bl\tcercle{}\\&&&&\\&&\\&&\\&&\\&&\bl\tcercle{}\\ }$$} }

\put(50,5){{$$\tiny\tableau[scY]{   &&&&&&&&\\&&&&&&&\bl\tcercle{}\\&&&&&&\bl\tcercle{}\\&&&&&\\&&&&&\\&&&&&\bl\tcercle{}\\&&&&\bl\tcercle{}\\&&\\&&\\&&\\&&\bl\tcercle{}\\ }$$} }

\put(80,5){{$$\tiny\tableau[scY]{   &&&&&&&&\bl\tcercle{}\\&&&&&&&\\&&&&&&\bl\tcercle{}\\&&&&&\\&&&&&\\&&&&&\bl\tcercle{}\\&&&&\bl\tcercle{}\\&&\\&&\\&&\\&&\bl\tcercle{}\\ }$$} }

\end{picture}

\vspace{1 cm}

\hspace{.5cm}\begin{picture}(14,18)(-6,-2)

\put(-10,5){{$$\tiny\tableau[scY]{  &&&&&&&&\bl\tcercle{}\\&&&&&&&\bl\tcercle{}\\&&&&&&\\&&&&&\\&&&&&\\&&&&&\bl\tcercle{}\\&&&&\bl\tcercle{}\\&&\\&&\\&&\\&&\bl\tcercle{}\\ }$$} }

\put(20,5){{$$\tiny\tableau[scY]{   &&&&&&&&\bl\tcercle{}\\&&&&&&&\bl\tcercle{}\\&&&&&&\bl\tcercle{}\\&&&&&\\&&&&&\\&&&&&\bl\tcercle{}\\&&&&\bl\tcercle{}\\&&\\&&\\&&\\&&\bl\tcercle{}\\ }$$} }

\put(50,5){{$$\tiny\tableau[scY]{   &&&&&&&&\bl\tcercle{}\\&&&&&&&\bl\tcercle{}\\&&&&&&\bl\tcercle{}\\&&&&&\\&&&&&\\&&&&&\bl\tcercle{}\\&&&&\bl\tcercle{}\\&&&\bl\tcercle{}\\&&\\&&\\&&\bl\tcercle{}\\ }$$} }

\put(80,5){{$$\tiny\tableau[scY]{   &&&&&&&&\bl\tcercle{}\\&&&&&&&\bl\tcercle{}\\&&&&&&\bl\tcercle{}\\&&&&&\\&&&&&\\&&&&&\bl\tcercle{}\\&&&&\bl\tcercle{}\\&&\\&&\\&&\\&&\bl\tcercle{}\\\bl\tcercle{}}$$} }

\end{picture}

\vspace{1 cm}

\hspace{.5cm}\begin{picture}(14,18)(-6,-2)
\put(-10,5){{$$\tiny\tableau[scY]{   &&&&&&&&\bl\tcercle{}\\&&&&&&&\bl\tcercle{}\\&&&&&&\bl\tcercle{}\\&&&&&\\&&&&&\\&&&&&\bl\tcercle{}\\&&&&\bl\tcercle{}\\&&&\bl\tcercle{}\\&&\\&&\\&&\bl\tcercle{}\\ \bl\tcercle{}}$$} }

\end{picture}
}
  \end{appendix}

\end{document}